\documentclass[10pt]{amsart}
\usepackage{amssymb}

\usepackage{etex}
\usepackage{subfigure}
\usepackage{xmpmulti}
\usepackage{colortbl,dcolumn}\usepackage[all,cmtip,line]{xy}

\usepackage{tikz}

\usetikzlibrary{calc}
\usetikzlibrary{through}

\newtheorem{proposition}{Proposition}[section]
\newtheorem{lemma}[proposition]{Lemma}
\newtheorem{corollary}[proposition]{Corollary}
\newtheorem{theorem}[proposition]{Theorem}
\newtheorem{remark}[proposition]{Remark}

\theoremstyle{definition}
\newtheorem{definition}[proposition]{Definition}

\newcommand{\selabel}[1]{\label{se:#1}}

\def\<{\leqslant}
\def\>{\geqslant}
\def\a{\alpha}
\def\b{\beta}
\def\d{\delta}

\def\ol{\overline}

\def\s{\sigma}

\def\ot{\otimes}
\def\ra{\rightarrow}

\def\Cal{\mathcal{C}}
\def\bbm{\mathbf{m}}
\def\bbs{\mathbf{s}}
\def\whot{\widehat{\otimes} }
\date{}

\begin{document}
\title{Reconstruction of  tensor categories from their structure invariants}
\author{Hui-Xiang Chen}
\address{School of Mathematical Science, Yangzhou University,
Yangzhou 225002, China}
\email{hxchen@yzu.edu.cn}
%\author{Fred Van Oystaeyen}
%\address{Department of Mathematics and Computer
%Science, University of Antwerp, Middelheimlaan 1, B-2020 Antwerp,
%Belgium}
%\email{fred.vanoystaeyen@ua.ac.be}
\author{Yinhuo Zhang}
\address{Department of Mathematics $\&$  Statistics, University of Hasselt, Universitaire Campus, 3590 Diepeenbeek,Belgium }
\email{yinhuo.zhang@uhasselt.be}

\begin{abstract}
In this paper, we study  tensor (or monoidal) categories of finite rank over an algebraically closed field $\mathbb F$.  Given a tensor category $\mathcal{C}$,  we have two structure invariants of $\mathcal{C}$:  the Green ring (or the representation ring) $r(\mathcal{C})$ and the Auslander algebra  $A(\mathcal{C})$ of $\mathcal{C}$.  We show that a Krull-Schmit  abelian tensor category $\mathcal{C}$ of finite rank is  uniquely determined (up to tensor equivalences)  by its two structure invariants and the associated associator system of $\mathcal{C}$.  In fact, we can reconstruct the tensor category $\mathcal{C}$ from  its two invarinats and the associator system.  More general,  given a  quadruple  $(R, A, \phi, a)$ satisfying certain conditions, where $R$ is a $\mathbb{Z}_+$-ring of rank $n$, $A$ is a finite dimensional $\mathbb F$-algebra with a complete set of $n$ primitive orthogonal idempotents, $\phi$ is an algebra map from $A\otimes_{\mathbb F}A$ to an algebra $M(R, A, n)$  constructed from $A$ and $R$, and $a=\{a_{i,j,l}|1\< i,j,l\<n\}$ is a family of ``invertible" matrices over $A$,  we  can construct a Krull-Schmidt and abelian  tensor category $\mathcal C$  over $\mathbb{F}$  such that $R$ is the Green ring of $\mathcal C$  and $A$ is the Auslander algebra of $\mathcal C$. In this case, $\mathcal C$ has  finitely many indecomposable objects  (up to isomorphisms) and finite dimensional Hom-spaces.  Moreover, we will  give a necessary and sufficient condition for such two tensor  categories to be tensor equivalent.
\end{abstract}

\maketitle
\section*{\bf Introduction}\selabel{1}

Motivated by the recent study of Green rings of Hopf algebras  \cite{Ch, CVZ, LZ, WLZ1, WLZ2}, we ask ourselves the question: to what extend a tensor (or monoidal) category is determined by its ring invariant, the Green ring? We do have examples of Hopf algebras whose Green rings are isomorphic; yet their representation categories are not tensor equivalent,  for example, the group algebras $\mathbb{F}Q_8$ and $\mathbb{F}D_8$.  On the other hand, the representation categories of non-semisimple Hopf algebras of dimension 8 are completely distinguished by their Green rings \cite{Wa}. Thus a natural question arises: apart from the Green ring of a tensor category $\mathcal{C}$, what  structure invariants else do we need in order to determine the category $\mathcal{C}$? In other words, can a tensor category be uniquely determined by some of its structure invariants? In this paper, we  explore the second structure invariant and an associator system associated to a tensor category $\mathcal{C}$ of finite rank which together with the Green ring can determine the tensor category $\mathcal{C}$.  In more detail, given a tensor category $\mathcal{C}$,  there exists a unique skeleton tensor category $\widehat{\Cal}$ (up to tensor equivalence) .  This tensor category   $\widehat{\Cal}$ is uniquely determined by  the two structure invariants, the Green ring $r(\mathcal{C})$ and the Auslander algebra $A(\Cal)$ and a couple $(\phi, a)$, where $\phi$ is an algebra map  from $A(\Cal)\otimes A(\Cal)$ to an algebra $M(\Cal)$ associated to $\Cal$ defining the tensor product of morphisms in  the  category $\widehat{\Cal}$; and $a$ is  a  family of invertible matrices $\{a_{ijk}\}$ over $A(\Cal)$  playing the role of the associativity constraint in the skeleton tensor category $\widehat{\Cal}$.  This new  tensor category $\widehat{\Cal}$ is tensor equivalent with the given tensor category $\Cal$, and is uniquely determined by the quadruple $(r(\Cal), A(\Cal), \phi, a)$. Thus, the study of the aforementioned structure invariants, on the one hand,   can help us to classify tensor categories;  and on the other hand,  it provides a solid basis to deform tensor categories. We say that a tensor category $\mathcal{D}$ is a deformation of a tensor category $\Cal$ if the Green rings of  the two tensor categories   are isomorphic, but the two categories are not tensor equivalent.  In this sense, one may change the Auslander algebra $A(\Cal)$ or the datum $ (\phi, a)$   (not independent of $A(\Cal)$) to obtain deformations of a tensor category.  In the special case where  a  tensor category $\Cal$ is semisimple,  the Auslander algebra $A(\Cal)$ is a direct product of the base field,  and the deformations of $\Cal$ come only from the changing of  the datum $(\phi, a)$. Consequently,  the classification of semisimple (or fusion) categories reduces to the classification of the Green rings (or fusion rules) and the associators.

The paper is organized as follows. In Section 1,  we  introduce $(\bbm, \bbs)$-type matrices, where $\bbm, \bbs\in \mathbb{N}^I$ are finite sequences of nonnegative integers, and define their concatenation sums.  We  recall  the tensor product of two $(\bbm, \bbs)$-type matrices as well.  In Section 2,  we study the $(\bbm,\bbs)$-type matrices over the Auslander algebra $A(\Cal)$ of a tensor category $\Cal$.  To   a tensor category $\Cal$, we associate an algebra $M(\Cal)$ defined by the Clebsch-Gordan coefficients and $(\bbm, \bbs)$-type matrices over $A(\Cal)$.  There exists a natural (but not unique) algebra map $\phi: A(\Cal)\otimes A(\Cal)\longrightarrow M(\Cal)$,  which enables us to define a new tensor product $\whot $ of  two $(\bbm, \bbs)$-type matrices. This new tensor product $\widehat{\otimes}$ will  define the tensor product of two morphisms in the associated tensor category $\widehat{\Cal}$. In Section 3, we use the data obtained in Section 2, namely the Green ring, the Auslander algebra, the algebra map $\phi$,  and the constraint $a=\{a_{ijk}\}$  to construct a skeleton tensor category $\widehat{\Cal}$ of finite rank. If the foregoing data stem from a tensor category  $\Cal$,  we show that the new obtained category $\widehat{\Cal}$ is tensor equivalent to the original category $\Cal$.  In other wards, a tensor category $\Cal$ of finite rank can be reconstructed from  (and  hence determined by)  its invariants $r(\Cal), A(\Cal)$ and $(\phi, a)$.

\section{\bf Preliminaries and Notations}\selabel{2}

Throughout, let $\mathbb{F}$ be an algebraically closed field, $\mathbb Z$ the ring of  integers, and let $\mathbb{N}$ be the set of nonnegative integers. Unless otherwise stated, all algebras are defined over $\mathbb F$.
All rings and algebras are assumed to be associative with identity. The definitions of a $\mathbb{Z}_+$-basis and a $\mathbb{Z}_+$-ring
can be found in \cite{EK, L1, Ostrik03}. For the theory of  (tensor) categories, we refer the reader to \cite{BaKir, EGNO,Ka}.

\begin{definition}\label{+ring}
(1)  Let $R$ be a ring free as a module over $\mathbb Z$.  A $\mathbb{Z}_+$-basis of  $R$
is a $\mathbb Z$-basis $B=\{b_i\}_{i\in I}$ such that for any $i, j\in I$,
$b_ib_j=\sum_lc_{ijl}b_l$ with $c_{ijl}\in{\mathbb N}$.\\
(2) A $\mathbb{Z}_+$-ring is a $\mathbb Z$-algebra with unity endowed with a
fixed $\mathbb{Z}_+$-basis.\\
(3) A $\mathbb{Z}_+$-ring with a $\mathbb{Z}_+$-basis $B$ is a unital
$\mathbb{Z}_+$-ring if $1\in B$.

%(4) A $\mathbb{Z}_+$-module over a $\mathbb{Z}_+$-ring $R$ is a $\mathbb{Z}$-free
%$R$-module $M$ endowed with a fixed basis $\{m_j\}_{j\in J}$ such that
%$b_im_j=\sum_kd_{ijk}m_k$, $d_{ijk}\in\mathbb{Z}_+$.
\end{definition}

For an $\mathbb{F}$-algebra $A$ and two positive integers $m$ and $n$, we denote by $M_{m\times n}(A)$ the $\mathbb F$-space
consisting of all $m\times n$-matrices over $A$, and by $M_m(A):=M_{m\times m}(A)$ the full matrix algebra of $m\times m$-matrices over $A$.
If $B$ is a subspace of $A$, let $M_{m\times n}(B)$ denote the subspace of $M_{m\times n}(A)$ consisting of all
$m\times n$-matrices over $A$ with entries contained in $B$.

In the subsequent, we assume that $A$ is an $\mathbb{F}$-algebra with a set of orthogonal idempotents
$\{e_i\}_{i\in I}$ satisfying $\sum_{i\in I}e_i=1$, where $I$ is a finite index set.

For any $\mathbf{m}=(m_i)_{i\in I}\in {\mathbb N}^{I}$,
let $|\mathbf{m}|:=\sum_{i\in I}m_i$.
For $\mathbf{m}, \mathbf{s}\in{\mathbb N}^{I}$ with $|\mathbf m|>0$
and $|\mathbf s|>0$, an $(\mathbf{m}, \mathbf{s})$-type matrix over $A$
is a block matrix $X=(X_{ij})_{i,j\in I}$
such that $X_{ij}\in M_{m_i\times s_j}(e_iAe_j)$, $i, j\in I$.
If some $m_i=0$ (or $s_j=0$), then there is no row $\{X_{ij}\}_{j\in I}$
(or no  column $\{X_{ij}\}_{i\in I}$) in $X$.
Note that $X$ is an $|\mathbf m|\times|\mathbf s|$-matrix over $A$.
A $(\mathbf{0}, \mathbf{m})$-type matrix over $A$ means a $1\times|\mathbf{m}|$ zero matrix over $A$,
and an $(\mathbf{m}, \mathbf{0})$-type matrix over $A$ means an $|\mathbf{m}|\times 1$ zero matrix over $A$. Of course,  a  $(\mathbf{0}, \mathbf{0})$-type matrix over $A$ means a $1\times 1$ zero matrix over $A$. Here $\mathbf{0}=(0)_{i\in I}$.

For any $i\in I$ and positive integer $m$, let $I_m$ and $I(i, m)$ denote the $m\times m$ identity matrix
over $A$ (or over $\mathbb F$) and $e_iAe_i$, respectively, i.e.,
$$I_m=\left(
\begin{array}{cccc}
1 & 0 & \cdots & 0 \\
0 & 1 & \cdots & 0 \\
\cdots & \cdots & \cdots & \cdots \\
0 & 0 & \cdots & 1 \\
\end{array}\right)_{m\times m} ,\
I(i, m)=\left(
\begin{array}{cccc}
e_i & 0 & \cdots & 0 \\
0 & e_i & \cdots & 0 \\
\cdots & \cdots & \cdots & \cdots \\
0 & 0 & \cdots & e_i \\
\end{array}\right)_{m\times m} .$$

For $\mathbf{m}, \mathbf{s}\in{\mathbb N}^I$ with $\mathbf m\neq\mathbf 0$
and $\mathbf s\neq\mathbf 0$, let $M_{\mathbf{m}\times\mathbf{s}}(A)$ be the
$\mathbb{F}$-space consisting of all $(\mathbf{m}, \mathbf{s})$-type matrices over $A$.
Then $M_{\mathbf{m}}(A):=M_{\mathbf{m}\times\mathbf{m}}(A)$ is an associative
$\mathbb{F}$-algebra with the unity $E_{\mathbf m}=(E_{ij})_{i, j\in I}$ given by
$E_{ij}=0$ for $i\neq j\in I$ with $m_im_j>0$, and $E_{ii}=I(i, m_i)$ for $i\in I$ with $m_i>0$.
For convenience, we let $E_{\mathbf 0}:=0$, the $1\times 1$ zero matrix over $A$.
Then $M_{\mathbf{0}}(A)=\{E_{\mathbf 0}\}=0$.

Now assume that $I=\{1, 2, \cdots, n\}$.
Then an element ${\mathbf m}=(m_i)_{i\in I}\in\mathbb{N}^I$ can be written as ${\mathbf m}=(m_1, m_2, \cdots, m_n)$
and an element $X=(X_{ij})_{i, j\in I}\in M_{\mathbf{m}\times\mathbf{s}}(A)$ can be written as
$$X=\left(\begin{array}{cccc}
X_{11}&X_{12}&\cdots&X_{1n}\\
X_{21}&X_{22}&\cdots&X_{2n}\\
\cdots&\cdots&\cdots&\cdots\\
X_{n1}&X_{n2}&\cdots&X_{nn}\\
\end{array}\right),$$
where $X_{ij}\in M_{m_i\times s_j}(e_iAe_j)$, $1\leqslant i, j\leqslant n$.

It is well-known that $\mathbb{N}^I$ is a monoid with respect to the addition:
$$(m_1, m_2, \cdots, m_n)+(s_1, s_2, \cdots, s_n)=(m_1+s_1, m_2+s_2, \cdots, m_n+s_n).$$

Let $\mathbf{m}$, $\mathbf{m}'$, $\mathbf{s}\in\mathbb{N}^I$.
For any $$X=\left(\begin{array}{cccc}
X_{11}&X_{12}&\cdots&X_{1n}\\
X_{21}&X_{22}&\cdots&X_{2n}\\
\cdots&\cdots&\cdots&\cdots\\
X_{n1}&X_{n2}&\cdots&X_{nn}\\
\end{array}\right)\in M_{\mathbf{m}\times\mathbf{s}}(A),\ \mbox{ where } X_{ij}\in M_{m_i\times s_j}(e_iAe_j),$$
and
$$Y=\left(\begin{array}{cccc}
Y_{11}&Y_{12}&\cdots&Y_{1n}\\
Y_{21}&Y_{22}&\cdots&Y_{2n}\\
\cdots&\cdots&\cdots&\cdots\\
Y_{n1}&Y_{n2}&\cdots&Y_{nn}\\
\end{array}\right)\in M_{\mathbf{m}'\times\mathbf{s}}(A),\ \mbox{ where } Y_{ij}\in M_{m'_i\times s_j}(e_iAe_j).$$
Define $X\underline{\oplus}Y\in M_{(\mathbf{m}+\mathbf{m}')\times\mathbf{s}}(A)$ by
$$X\underline{\oplus}Y:=\left(\begin{array}{cccc}
X_{11}&X_{12}&\cdots&X_{1n}\\
Y_{11}&Y_{12}&\cdots&Y_{1n}\\
X_{21}&X_{22}&\cdots&X_{2n}\\
Y_{21}&Y_{22}&\cdots&Y_{2n}\\
\cdots&\cdots&\cdots&\cdots\\
X_{n1}&X_{n2}&\cdots&X_{nn}\\
Y_{n1}&Y_{n2}&\cdots&Y_{nn}\\
\end{array}\right).$$
Then obviously $(X\underline{\oplus}Y)\underline{\oplus}Z=X\underline{\oplus}(Y\underline{\oplus}Z)$
for any $X\in M_{\mathbf{m}\times\mathbf{s}}(A)$, $Y\in M_{\mathbf{m}'\times\mathbf{s}}(A)$
and $Z\in M_{\mathbf{m}''\times\mathbf{s}}(A)$ with $\mathbf{m}, \mathbf{m}', \mathbf{m}'',
\mathbf{s}\in\mathbb{N}^I$. However, $X\underline{\oplus}Y\neq Y\underline{\oplus}X$ in general.

Let $P_{\mathbf{m},\mathbf{m}'}$ be an $(|\mathbf{m}|+|\mathbf{m}'|)\times(|\mathbf{m}|+|\mathbf{m}'|)$
permutation matrix defined by
$$P_{\mathbf{m},\mathbf{m}'}=\left(\begin{array}{cccccccc}
I_{m_1}&0&\cdots&0&0&0&\cdots&0\\
0&0&\cdots&0&I_{m'_1}&0&\cdots&0\\
0&I_{m_2}&\cdots&0&0&0&\cdots&0\\
0&0&\cdots&0&0&I_{m'_2}&\cdots&0\\
\cdots&\cdots&\cdots&\cdots&\cdots&\cdots&\cdots&\cdots\\
0&0&\cdots&I_{m_n}&0&0&\cdots&0\\
0&0&\cdots&0&0&0&\cdots&I_{m'_n}\\
\end{array}\right).$$
It is obvious that $X\underline{\oplus}Y=P_{\mathbf{m},\mathbf{m}'}\left(\begin{array}{c}
X\\
Y\\
\end{array}\right)$.

In a similar way, we define another sum $X\overline{\oplus}Y\in M_{\mathbf{m}\times(\mathbf{s}+\mathbf{s}')}(A)$ for two matrices
$X\in M_{\mathbf{m}\times\mathbf{s}}(A)$ and  $ Y\in M_{\mathbf{m}\times\mathbf{s}'}(A)$ as follows:
$$X\overline{\oplus}Y:=\left(\begin{array}{ccccccc}
X_{11}&Y_{11}&X_{12}&Y_{12}&\cdots&X_{1n}&Y_{1n}\\
X_{21}&Y_{21}&X_{22}&Y_{22}&\cdots&X_{2n}&Y_{2n}\\
&\cdots&\cdots&\cdots&\cdots&\cdots&\\
X_{n1}&Y_{n1}&X_{n2}&Y_{n2}&\cdots&X_{nn}&Y_{nn}\\
\end{array}\right).$$
Like the sum $\underline{\oplus}$,  the sum $\overline{\oplus}$ is associative, but not commutative in general .
Obviously, $X\overline{\oplus}Y=(X, Y)P_{\mathbf{s}, \mathbf{s}'}^T$, where $P^T$ denotes the transposed
matrix of $P$.

Let $\mathbf{m}_1, \mathbf{m}_2, \cdots, \mathbf{m}_r$, $\mathbf{s}_1, \mathbf{s}_2, \cdots, \mathbf{s}_l\in\mathbb{N}^I$.
For a matrix
$$X=\left(\begin{array}{cccc}
X_{11}&X_{12}&\cdots&X_{1l}\\
X_{21}&X_{22}&\cdots&X_{2l}\\
\cdots&\cdots&\cdots&\cdots\\
X_{r1}&X_{r2}&\cdots&X_{rl}\\
\end{array}\right)$$
over $A$ with $X_{ij}\in M_{\mathbf{m}_i\times\mathbf{s}_j}(A)$, $1\leqslant i\leqslant r$,
$1\leqslant j\leqslant l$, we define an $(\mathbf{m}, \mathbf{s})$-type matrix $\Pi(X)$ over $A$ by setting
%$$\begin{array}{rcl}
%\prod(X)&=&(X_{11}\overline{\oplus}X_{12}\overline{\oplus}\cdots\overline{\oplus}X_{1l})\\
%&&\underline{\oplus}(X_{21}\overline{\oplus}X_{22}\overline{\oplus}\cdots\overline{\oplus}X_{2l})\\
%&&\underline{\oplus}\cdots\\
%&&\underline{\oplus}(X_{r1}\overline{\oplus}X_{r2}\overline{\oplus}\cdots\overline{\oplus}X_{rl})\\
%\end{array}$$
$$\Pi(X):=(X_{11}\overline{\oplus}X_{12}\overline{\oplus}\cdots\overline{\oplus}X_{1l})
\underline{\oplus}(X_{21}\overline{\oplus}X_{22}\overline{\oplus}\cdots\overline{\oplus}X_{2l})
\underline{\oplus}\cdots
\underline{\oplus}(X_{r1}\overline{\oplus}X_{r2}\overline{\oplus}\cdots\overline{\oplus}X_{rl}),$$
where $\mathbf{m}=\mathbf{m}_1+\mathbf{m}_2+\cdots+\mathbf{m}_r$ and
$\mathbf{s}=\mathbf{s}_1+\mathbf{s}_2+\cdots+\mathbf{s}_l$. Obviously, we have
$$\Pi(X)=(X_{11}\underline{\oplus}X_{21}\underline{\oplus}\cdots\underline{\oplus}X_{r1})
\overline{\oplus}(X_{12}\underline{\oplus}X_{22}\underline{\oplus}\cdots\underline{\oplus}X_{r2})
\overline{\oplus}\cdots
\overline{\oplus}(X_{1l}\underline{\oplus}X_{2l}\underline{\oplus}\cdots\underline{\oplus}X_{rl}).$$

For $\mathbf{m}_1, \mathbf{m}_2, \cdots, \mathbf{m}_r\in\mathbb{N}^I$, define a permutation matrix
$P_{\mathbf{m}_1,\cdots,\mathbf{m}_r}$ recursively on $r$ as follows:
for $r=1$, $P_{\mathbf{m}_1}=I_{|\mathbf{m}_1|}$; for $r=2$,
$P_{\mathbf{m}_1, \mathbf{m}_2}$ is defined as before; for $r>2$,
$P_{\mathbf{m}_1,\cdots,\mathbf{m}_r}:=P_{\mathbf{m}_1+\cdots+\mathbf{m}_{r-1}, \mathbf{m}_r}
\left(\begin{array}{cc}
P_{\mathbf{m}_1,\cdots,\mathbf{m}_{r-1}}&0\\
0&I_{|\mathbf{m}_r|}\\
\end{array}\right)$.
Then one can see that
$$\Pi(X)=P_{\mathbf{m}_1,\cdots,\mathbf{m}_r}XP_{\mathbf{s}_1,\cdots,\mathbf{s}_l}^T.$$

Let $A$ and $B$ be two algebras over $\mathbb F$. Let
$$X=\left(\begin{array}{cccc}
x_{11}&x_{12}&\cdots&x_{1s}\\
x_{21}&x_{22}&\cdots&x_{2s}\\
\cdots&\cdots&\cdots&\cdots\\
x_{m1}&x_{m2}&\cdots&x_{ms}\\
\end{array}\right)
\mbox{ and }
Y=\left(\begin{array}{cccc}
y_{11}&y_{12}&\cdots&y_{1s'}\\
y_{21}&y_{22}&\cdots&y_{2s'}\\
\cdots&\cdots&\cdots&\cdots\\
y_{m'1}&y_{m'2}&\cdots&y_{m's'}\\
\end{array}\right)$$
be an $m\times s$-matrix over $A$ and an $m'\times s'$-matrix over $B$, respectively.
Then one can define an $mm'\times ss'$-matrix $X\ot_{\mathbb F}Y$ over
$A\ot_{\mathbb F}B$ in a natural way:
$$\begin{array}{l}
X\ot_{\mathbb F}Y:=\\
\left(\begin{array}{ccccccc}
x_{11}\ot_{\mathbb F} y_{11}&\cdots&x_{1s}\ot_{\mathbb F} y_{11}&\cdots
&x_{11}\ot_{\mathbb F} y_{1s'}&\cdots&x_{1s}\ot_{\mathbb F} y_{1s'}\\
&\cdots&\cdots&\cdots&\cdots&\cdots&\\
x_{m1}\ot_{\mathbb F} y_{11}&\cdots&x_{ms}\ot_{\mathbb F} y_{11}&\cdots
&x_{m1}\ot_{\mathbb F} y_{1s'}&\cdots&x_{ms}\ot_{\mathbb F} y_{1s'}\\
\vdots&\vdots&\vdots&\vdots&\vdots&\vdots&\vdots\\
x_{11}\ot_{\mathbb F} y_{m'1}&\cdots&x_{1s}\ot_{\mathbb F} y_{m'1}&\cdots
&x_{11}\ot_{\mathbb F} y_{m's'}&\cdots&x_{1s}\ot_{\mathbb F} y_{m's'}\\
&\cdots&\cdots&\cdots&\cdots&\cdots&\\
x_{m1}\ot_{\mathbb F} y_{m'1}&\cdots&x_{ms}\ot_{\mathbb F} y_{m'1}&\cdots
&x_{m1}\ot_{\mathbb F} y_{m's'}&\cdots&x_{ms}\ot_{\mathbb F} y_{m's'}\\
\end{array}\right).\\
\end{array}$$
It is obvious that the above tensor product is associative. The following property of the tensor product is easy to check.

\begin{lemma}\label{composition}
Assume that $A$ and $B$ are two $\mathbb F$-algebras. Let $X\in M_{m\times s}(A)$,
$X_1\in M_{s\times t}(A)$, $Y\in M_{m'\times s'}(B)$ and $Y_1\in M_{s'\times t'}(B)$.
Then $(X\ot_{\mathbb F}Y)(X_1\ot_{\mathbb F}Y_1)=(XX_1)\ot_{\mathbb F}(YY_1)$.
\end{lemma}

We need the following definition.

\begin{definition}
Let $r\geqslant 1$ be an integer and
$$X=\left(\begin{array}{cccc}
x_{11}&x_{12}&\cdots&x_{1s}\\
x_{21}&x_{22}&\cdots&x_{2s}\\
\cdots&\cdots&\cdots&\cdots\\
x_{m1}&x_{m2}&\cdots&x_{ms}\\
\end{array}\right)$$
be an $m\times s$-matrix over $A^{\ot_{\mathbb F} r}:=A\ot_{\mathbb F}A\ot_{\mathbb F}\cdots\ot_{\mathbb F} A$,
the tensor product algebra of $r$-folds of $A$ over ${\mathbb F}$.
If there exist
$1\leqslant i_{11}, i_{12}, \cdots, i_{1m}, i_{21}, i_{22}, \cdots, i_{2m},
\cdots, i_{r1}, i_{r2}, \cdots,$ $i_{rm}\leqslant n$ and
$1\leqslant j_{11}, i_{12}, \cdots, i_{1s}, j_{21}, j_{22}, \cdots, j_{2s},
\cdots, j_{r1}, j_{r2}, \cdots, j_{rs}\leqslant n$ such that
$$\begin{array}{rl}
x_{lt}\in&(e_{i_{1l}}\ot_{\mathbb F} e_{i_{2l}}\ot_{\mathbb F} \cdots\ot_{\mathbb F} e_{i_{rl}})
A^{\ot_{\mathbb F} r}(e_{j_{1t}}\ot_{\mathbb F} e_{j_{2t}}\ot_{\mathbb F} \cdots\ot_{\mathbb F} e_{j_{rt}})\\
&=(e_{i_{1l}}Ae_{j_{1t}})\ot_{\mathbb F}(e_{i_{2l}}Ae_{j_{2t}})\ot_{\mathbb F}\cdots\ot_{\mathbb F}(e_{i_{rl}}Ae_{j_{rt}})\\
\end{array}$$
for all $1\leqslant l\leqslant m$ and $1\leqslant t\leqslant s$, then $X$
is called a {\it homogeneous matrix} over $A^{\ot_{\mathbb F}r}$.
\end{definition}

An easy observation leads to the following lemma.

\begin{lemma}\label{homo}
If $X$ is an $(\mathbf{m}, \mathbf{s})$-type matrix over $A$, then $X$ is
a homogeneous $|\mathbf{m}|\times|\mathbf{s}|$-matrix over $A$.
If $X$ is  homogeneous over $A^{\ot r}$ and $Y$ is homogeneous
over $A^{\otimes t}$, then $X\ot_{\mathbb F}Y$ is a homogeneous matrix over $A^{\ot(r+t)}$  for $r, t\geqslant 1$.
\end{lemma}

\begin{definition}
An $(\mathbf{m}, \mathbf{s})$-type matrix $X$ over $A$ is called column-independent
if for any $(\mathbf{s}, \mathbf{l})$-type matrix $Y$ over $A$,
$XY=0$ implies $Y=0$. $X$ is called row-independent if for any $(\mathbf{l}, \mathbf{m})$-type matrix
$Y$ over $A$, $YX=0$ implies $Y=0$.
\end{definition}

\begin{definition}\label{KerCok1}
Let $X$ be an $(\mathbf{m}, \mathbf{s})$-type matrix over $A$,
$\mathbf{m}, \mathbf{s}\in{\mathbb N}^I$.

{\rm (a)} An $(\mathbf{s}, \mathbf{t})$-type matrix $Y$ over $A$ is called a right universal annihilator
of $X$ if it satisfies the following two conditions:\\
\mbox{\hspace{0.2cm}} {\rm (1)} $XY=0$,\\
\mbox{\hspace{0.2cm}} {\rm (2)} if $M$ is an $(\mathbf{s}, \mathbf{l})$-type matrix over $A$ with $XM=0$,
then there is a unique $(\mathbf{t}, \mathbf{l})$-type matrix $Z$ over $A$ such that $M=YZ$.

{\rm (b)} An $(\mathbf{t}, \mathbf{m})$-type matrix $Y$ over $A$ is called a left universal annihilator
of $X$ if it satisfies the following two conditions:\\
\mbox{\hspace{0.2cm}} {\rm (1)} $YX=0$,\\
\mbox{\hspace{0.2cm}} {\rm (2)} if $M$ is an $(\mathbf{l}, \mathbf{m})$-type matrix over $A$ with $MX=0$,
then there is a unique $(\mathbf{l}, \mathbf{t})$-type matrix $Z$ over $A$ such that $M=ZY$.
\end{definition}

\begin{remark}
For an $(\mathbf{m}, \mathbf{s})$-type matrix $X$ over $A$, if an $(\mathbf{s}, \mathbf{t})$-type matrix $Y$
is a right universal annihilator of $X$,  then $Y$ is column-independent. Furthermore,
if an $(\mathbf{s}, \mathbf{t'})$-type matrix $Y_1$ is also a right universal annihilator
of $X$, then  $\mathbf{t'}=\mathbf{t}$
and there is a unique invertible element $Y_2$ in $M_{\mathbf{t}}(A)$
such that $Y_1=YY_2$.
Similarly, if a $(\mathbf{t}, \mathbf{m})$-type matrix $Y$
is a left universal annihilator of $X$ then $Y$ is row-independent. Furthermore,
if a $(\mathbf{t'}, \mathbf{m})$-type matrix $Y_1$ is also a left universal annihilator
of $X$, then $\mathbf{t'}=\mathbf{t}$
and there is a unique invertible element $Y_2$ in $M_{\mathbf{t}}(A)$
such that $Y_1=Y_2Y$.
\end{remark}

\section{\bf Invariants of tensor categories}\selabel{4}

Throughout this section, all categories considered are
Krull-Schmidt and abelian categories over $\mathbb{F}$ with
finitely many indecomposable objects and finite dimensional Hom-spaces.
If $\mathcal C$ is a tensor category, we will assume that $\mathcal C$
is strict and the unit object $\mathbf{1}$ of
$\mathcal C$ is simple. All functors (bifunctors) are assumed to be $\mathbb{F}$-linear ($\mathbb{F}$-bilinear).

For an object $V$ of $\mathcal C$ and a positive integer $l$,
let $lV$ denote the direct sum of
$l$ copies of $V$, and $0V=0$, the zero object of $\mathcal C$.

\begin{definition}\label{Greenring}
	Let $\mathcal C$ be a tensor category. The Green ring $r(\mathcal C)$ of $\mathcal C$
	is the abelian group generated by the
	isomorphism classes $[V]$ of objects $V$ of $\mathcal C$
	modulo the relations $[M\oplus V]=[M]+[V]$.
	The multiplication of $r(\mathcal C)$
	is given by the tensor product of objects, that is,
	$[M][V]=[M\ot V]$, where $M$ and $V$ are objects of $\mathcal C$.
\end{definition}

Assume that $\mathcal C$ is a tensor category, and let $\{V_i|i\in I\}$ be a set
of representatives of the
isomorphism classes of indecomposable objects of $\mathcal C$.
We may assume that $\mathbf{1}\in\{V_i|i\in I\}$.
Then it is easy to see that $r(\mathcal C)$ is a unital $\mathbb{Z}_+$-ring
with a $\mathbb{Z}_+$-basis $\{[V_i]|i\in I\}$.
Furthermore, we assume that $U\ot V\neq 0$ for any nonzero objects $U$ and $V$
of $\mathcal C$.

Let $V=\oplus_{i\in I}V_i$, and $A(\mathcal C)={\rm End}_{\mathcal C}(V)
={\rm Hom}_{\mathcal C}(V, V)$, the Auslander algebra of $ \mathcal C$. Then $A(\mathcal C)$ is a finite
dimensional $\mathbb{F}$-algebra. By the definition of direct sum, there are canonical
morphisms $\pi_i: V\ra V_i$ and $\tau_i: V_i\ra V$, $i\in I$,
such that ${\rm id}_{V}=\sum_{i\in I}\tau_i\circ\pi_i$
and
$$\pi_i\circ\tau_j=\left\{
\begin{array}{ll}
{\rm id}_{V_i},& {\rm if }\ i=j,\\
0,& {\rm else}.\\
\end{array}\right.$$
Let $e_i=\tau_i\circ\pi_i$ in $A(\mathcal C)$, $i\in I$. Obviously, up to
$\mathbb{F}$-algebra isomorphism, $A(\mathcal C)$ is independent of the choices of
representatives $\{V_i|i\in I\}$ of the isomorphism classes of indecomposable
objects of $\mathcal C$. Then we have the following lemma.

\begin{lemma}\label{orthidem}
	$\{e_i|i\in I\}$ is a complete set of orthogonal primitive idempotents in $A(\mathcal C)$.
\end{lemma}

\begin{proof}
	Obviously, $\{e_i|i\in I\}$ is a set of orthogonal idempotents in $A(\mathcal C)$ and $\sum_{i\in I}e_i=1$.
	Let $i\in I$. Then $e_iA(\mathcal C)e_i$ is an $\mathbb{F}$-algebra with the unity $e_i$.
	It is easy to check that the map ${\rm End}(V_i)\ra e_iA(\mathcal C)e_i$,
	$f\mapsto \tau_if\pi_i$ is an algebra isomorphism with the inverse given by
	$g\mapsto \pi_ig\tau_i$, $g\in e_iA(\mathcal C)e_i$.
	Since $V_i$ is an indecomposable object, ${\rm End}(V_i)$ is a local algebra, and
	so is $e_iA(\mathcal C)e_i$. It follows that $e_i$ is a primitive idempotent.
\end{proof}

In the following, we will study the properties of the Auslander algebra $A(\mathcal C)$
of  $\mathcal C$.

\begin{lemma}\label{proj-class}
	Let $i, j\in I$ with $i\neq j$. Then $e_iA(\mathcal C)e_jA(\mathcal C)e_i
	\subseteq {\rm rad}(e_iA(\mathcal C)e_i)$. In particular, $e_iA(\mathcal C)\ncong
	e_jA(\mathcal C)$ as right $A(\mathcal C)$-modules.
\end{lemma}

\begin{proof}
	By Lemma \ref{orthidem}, $e_iA(\mathcal C)e_i$ is a finite dimensional
	local algebra over $\mathbb{F}$. It follows that ${\rm rad}(e_iA(\mathcal C)e_i)$
	is the unique maximal ideal of $e_iA(\mathcal C)e_i$. Obviously,
	$e_iA(\mathcal C)e_jA(\mathcal C)e_i$ is an ideal of $e_iA(\mathcal C)e_i$.
	Hence it is enough to show that $e_i\notin e_iA(\mathcal C)e_jA(\mathcal C)e_i$.
	Indeed, suppose that there are elements $f_1, \cdots, f_s$, $g_1, \cdots, g_s$
	of $A(\mathcal C)$ such that $e_i=\sum_{l=1}^se_if_le_jg_le_i$.
	Then ${\rm id}_{V_i}=\pi_ie_i\tau_i=\sum_{l=1}^s\pi_ie_if_le_jg_le_i\tau_i
	=\sum_{l=1}^s\pi_if_l\tau_j\pi_jg_l\tau_i$. Note that $\pi_jg_l\tau_i$ is a morphism
	from $V_i$ to $V_j$, and $\pi_if_l\tau_j$ is a morphism from
	$V_j$ to $V_i$, $1\leq l\leq s$.
	Therefore, $V_i$ is isomorphic to a direct summand of $sV_j$, which contradicts
	with the hypothesis that $\mathcal C$ is a Krull-Schmidt category since
	$V_i$ and $V_j$ are non-isomorphic indecomposable objects.
\end{proof}

\begin{corollary}\label{rad}
	Let $i, j\in I$ with $i\neq j$. Then for any morphisms $f\in{\rm Hom}(V_i, V_j)$
	and $g\in{\rm Hom}(V_j, V_i)$, $gf\in{\rm rad}({\rm End}(V_i))$.
\end{corollary}

\begin{proof}
	Follows from Lemma \ref{proj-class} and the proof of Lemma \ref{orthidem}.
\end{proof}

\begin{corollary}\label{KrullSch2}
Let $\mathbf{m}, \mathbf{s}\in{\mathbb N}^I$ with $\mathbf{m}\neq\mathbf{0}$ and $\mathbf{s}\neq\mathbf{0}$.
If there is an $(\mathbf{m}, \mathbf{s})$-type matrix $X$ and an $(\mathbf{s}, \mathbf{m})$-type matrix
$Y$ over $A(\mathcal C)$ such that $XY=E_{\mathbf{m}}$ and $YX=E_{\mathbf{s}}$, then $\mathbf{m}=\mathbf{s}$.
\end{corollary}

\begin{proof}
Assume that there exist $X=(X_{ij})\in M_{\mathbf{m}\times\mathbf{s}}(A(\mathcal C))$
and $Y=(Y_{ij})\in M_{\mathbf{s}\times\mathbf{m}}(A(\mathcal C))$ such that $XY=E_{\mathbf m}$ and $YX=E_{\mathbf s}$,
where $X_{ij}\in M_{m_i\times s_j}(e_iA(\mathcal C)e_j)$ and $Y_{ij}\in M_{s_i\times m_j}(e_iA(\mathcal C)e_j)$, $i, j\in I$.
Let $i\in I$ and assume $m_i>0$.  Since $XY=E_{\mathbf{m}}$, we have $\sum_{j\in I}X_{ij}Y_{ji}=I(i, m_i)$.
By Lemma \ref{proj-class}, we have $\sum_{j\neq i}X_{ij}Y_{ji}\in M_{m_i}({\rm rad}(e_iA(\mathcal C)e_i))$.
Hence $\sum_{j\neq i}X_{ij}Y_{ji}$ is a nilpotent matrix
in $M_{m_i}(e_iA(\mathcal C)e_i)$. If $s_i=0$, then $I(i, m_i)=\sum_{j\neq i}X_{ij}Y_{ji}$,
a contradiction. Therefore, $s_i>0$. Then $X_{ii}Y_{ii}=I(i, m_i)-\sum_{j\neq i}X_{ij}Y_{ji}$
is an invertible matrix in $M_{m_i}(e_iA(\mathcal C)e_i)$. It follows that $m_i\leqslant s_i$.
Similarly, one can show that $s_i\leqslant m_i$ for all $i\in I$.
Therefore, ${\mathbf m}={\mathbf s}$.
\end{proof}

\begin{proposition}\label{KerCok}
Let $X$ be an $(\mathbf{m}, \mathbf{s})$-type matrix over $A(\mathcal C)$,
$\mathbf{m}, \mathbf{s}\in{\mathbb N}^I$. Then\\
{\rm (1)} $X$ has a right universal annihilator.\\
{\rm (2)} $X$ has a left universal annihilator.
	
\end{proposition}

\begin{proof}
For any $i, j\in I$, there is a canonical $\mathbb{F}$-linear isomorphism
${\rm Hom}_{\mathcal C}(V_i, V_j)\ra e_jA(\mathcal C)e_i$,
$f\mapsto \tau_jf\pi_i$. Identifying ${\rm Hom}_{\mathcal C}(V_i, V_j)$ with
$e_jA(\mathcal C)e_i$ via the above canonical isomorphism,
an $(\mathbf{m}, \mathbf{s})$-type matrix over $A(\mathcal C)$ can be viewed
as a morphism from $\oplus_{i\in I}s_iV_i$
to $\oplus_{i\in I}m_iV_i$ in $\mathcal C$.
Thus, Part (1) follows from the fact that any morphism of $\mathcal C$
has a kernel, and Part (2) follows from the fact that any morphism
has a cokernel since $\mathcal C$ is an abelian category.
\end{proof}

Under the identification of morphisms with matrices described in the proof
of Proposition \ref{KerCok}, a monomorphism corresponds to a column-independent
matrix, and an epimorphism corresponds to a row-independent matrix.
Since $\mathcal C$ is an abelian category over $\mathbb F$, every morphism is
a composition of an epimorphism followed by a monomorphism. Moreover, for any morphism
$f$ of $\mathcal C$, if ${\rm ker}(f)=0$
then $f={\rm ker}({\rm coker}(f))$; if ${\rm coker}(f)=0$ then $f={\rm coker}({\rm ker}(f))$.
Hence  we have the following proposition.

\begin{proposition}\label{decom}
Let $X$ be an $(\mathbf{m}, \mathbf{s})$-type matrix over $A(\mathcal C)$,
$\mathbf{m}, \mathbf{s}\in{\mathbb N}^I$.\\
{\rm (1)} There is a column-independent $(\mathbf{m}, \mathbf{t})$-type matrix
$X_1$ and a row-independent $(\mathbf{t}, \mathbf{s})$-type matrix $X_2$
such that $X=X_1X_2$.\\
{\rm (2)} If $X$ is column-independent and $Y$ is a left universal annihilator
of $X$, then $X$ is a right universal annihilator of $Y$.\\
{\rm (3)} If $X$ is row-independent and $Y$ is a right universal annihilator
of $X$, then $X$ is a left universal annihilator of $Y$.
\end{proposition}

Now assume that $[V_i][V_j]=\sum_{k\in I}c_{ijk}[V_k]$ in $r(\mathcal C)$.
For any $i, j\in I$, let $\mathbf{c}_{ij}=(c_{ijk})_{k\in I}\in\mathbb{N}^I$.
Define a vector space $M(\mathcal C)$ over $\mathbb{F}$ as follows:
$$M(\mathcal C):=\oplus_{i,i',j,j'\in I}M_{\mathbf{c}_{i'j'}\times\mathbf{c}_{ij}}(A(\mathcal C)).$$
Then $M(\mathcal C)$ is an associative $\mathbb{F}$-algebra with the multiplication
defined as follows: if $X\in M_{\mathbf{c}_{i'j'}\times\mathbf{c}_{ij}}(A(\mathcal C))$ and
$Y\in M_{\mathbf{c}_{i''j''}\times\mathbf{c}_{i_1j_1}}(A(\mathcal C))$,
then $XY$ is the usual matrix product for $(i, j)=(i'',j'')$,
and $XY=0$ for $(i, j)\neq(i'',j'')$.
The unity of $M(\mathcal C)$ is $(E_{\mathbf{c}_{ij}})_{i,j\in I}$,
where $E_{\mathbf{c}_{ij}}\in M_{\mathbf{c}_{ij}}(A(\mathcal C))$ is given in Section 2.
Since $V_i\ot V_j\neq 0$, $\mathbf{c}_{ij}\neq\mathbf{0}$, where $i, j\in I$.

Assume that $I=\{1, 2, \cdots, n\}$ and $V_1=\mathbf{1}$.
Then $[V_1]=[\mathbf{1}]=1$, the identity of the ring $r(\mathcal C)$.
For any $1\leqslant i\leqslant n$, let ${\mathbf{e}}_i=(e_{i1}, e_{i2}, \cdots, e_{in})\in\mathbb{N}^I$, where
$e_{ij}=\d_{ij}$, the Kronecker symbols.
Then, $\mathbf{c}_{1i}=\mathbf{c}_{i1}=\mathbf{e}_{i}$ for all $1\leqslant i\leqslant n$.
In this case, $M_{\mathbf{c}_{1i}\times\mathbf{c}_{1j}}(A(\mathcal C))=M_{\mathbf{c}_{i1}\times\mathbf{c}_{j1}}(A(\mathcal C))=
M_{\mathbf{e}_{i}\times\mathbf{e}_{j}}(A(\mathcal C))=M_{1\times 1}(e_iA(\mathcal C)e_j)=e_iA(\mathcal C)e_j$ for all $1\leqslant i, j\leqslant n$.\
For a $1\times 1$-matrix $(a)$ over $A(\mathcal C)$, we simply write it as $a$.

By the associativity of the multiplication of the Green ring $r(\mathcal C)$, one obtains the following equations:
$$\hspace{3cm}\sum_{k=1}^nc_{ijk}\mathbf{c}_{kl}=\sum_{k=1}^nc_{jlk}\mathbf{c}_{ik},
\ 1\leqslant i, j, l\leqslant n. \hspace{3cm}(*)$$

\begin{proposition}\label{tom}
	There is an algebra map $\phi_{\mathcal C}: A(\mathcal C)\ot_{\mathbb F} A(\mathcal C)\ra M(\mathcal C)$
	such that the following two conditions are satisfied:\\
	{\rm (1)} $\phi_{\mathcal C}(e_i\ot_{\mathbb F}e_j)=E_{\mathbf{c}_{ij}}\in M_{\mathbf{c}_{ij}}(A(\mathcal C))$
	for all $i, j\in I$;\\
	{\rm (2)} $\phi_{\mathcal C}(e_1\ot_{\mathbb F}a)=a\in M_{\mathbf{c}_{1i}\times\mathbf{c}_{1j}}(A(\mathcal C))$
	and $\phi_{\mathcal C}(a\ot_{\mathbb F}e_1)=a\in M_{\mathbf{c}_{i1}\times\mathbf{c}_{j1}}(A(\mathcal C))$
	for all $a\in e_iA(\mathcal C)e_j$, $i, j\in I$.
\end{proposition}

\begin{proof}
	At first, one can easily check that the map
	%Since $\mathcal C$ is a tensor category, there is a bilinear map
	%$${\rm Hom}_{\mathcal C}(U, U')\times{\rm Hom}_{\mathcal C}(W, W')\ra
	%{\rm Hom}_{\mathcal C}(U\otimes W, U'\otimes W'),\ (f, g)\mapsto f\otimes g$$
	%for any objects $U, U', W$, $W'$ of $\mathcal C$, which induces an $\mathbb{F}$-linear map
	%$${\rm Hom}_{\mathcal C}(U, U')\otimes{\rm Hom}_{\mathcal C}(W, W')\ra
	%{\rm Hom}_{\mathcal C}(U\otimes W, U'\otimes W'),\ f\otimes g\mapsto f\otimes g.$$
	%Putting $U=U'=W=W'=V$ in the above, one gets an algebra map
	$$\phi_1: A(\mathcal C)\ot_{\mathbb F}A(\mathcal C)\ra {\rm End}_{\mathcal C}(V\otimes V),\
	a\ot_{\mathbb F}b\mapsto a\otimes b$$
	is a well-defined algebra map.
	Since $\{e_i\otimes e_j\}_{1\leqslant i,j\leqslant n}$ is a complete set of orthogonal idempotents
	of ${\rm End}_{\mathcal C}(V\otimes V)$, we have
	$${\rm End}_{\mathcal C}(V\otimes V)=\oplus_{i',i,j',j=1}^n(e_{i'}\otimes e_{j'})
	{\rm End}_{\mathcal C}(V\otimes V)(e_i\otimes e_j).$$
	Then one can see that the maps
	$$\begin{array}{rcl}
	(e_{i'}\otimes e_{j'}){\rm End}_{\mathcal C}(V\otimes V)(e_i\otimes e_j)&\ra
	&{\rm Hom}_{\mathcal C}(V_i\otimes V_j, V_{i'}\otimes V_{j'})\\
	f&\mapsto &(\pi_{i'}\otimes \pi_{j'})f(\tau_i\otimes \tau_j)\\
	\end{array}$$
	are $\mathbb{F}$-linear isomorphisms, which induce an algebra isomorphism
	$$\phi_2: {\rm End}_{\mathcal C}(V\otimes V)\ra\oplus_{i,i',j,j'=1}^n
	{\rm Hom}_{\mathcal C}(V_i\otimes V_j, V_{i'}\otimes V_{j'})$$
	such that $\phi_2(e_i\otimes e_j)={\rm id}_{V_i\otimes V_j}\in
	{\rm End}_{\mathcal C}(V_i\otimes V_j)$, where the multiplication of $\oplus_{1\leqslant i,i',j,j'\leqslant n}
	{\rm Hom}_{\mathcal C}(V_i\otimes V_j, V_{i'}\otimes V_{j'})$ is defined as follows:
	the product $fg$ of $f\in{\rm Hom}_{\mathcal C}(V_i\otimes V_j, V_{i'}\otimes V_{j'})$
	and $g\in{\rm Hom}_{\mathcal C}(V_{i''}\otimes V_{j''}, V_{i_1}\otimes V_{j_1})$
	is given by $fg=0$ for $(i_1, j_1)\neq (i,j)$, and $fg=f\circ g$, the composition
	of $f$ after $g$ for $(i_1, j_1)=(i,j)$.
	For any $1\leqslant i, j\leqslant n$, we have $V_i\otimes V_j\cong \oplus_{k=1}^nc_{ijk}V_k$. Hence, 	one can fix an isomorphism $\theta_{ij}: V_i\otimes V_j\ra \oplus_{k=1}^nc_{ijk}V_k$
	in $\mathcal C$. Since $V_1\ot V_i=\mathbf{1}\ot V_i=V_i=V_i\ot\mathbf{1}=V_i\ot V_1$, we may choose
	$\theta_{1i}=\theta_{i1}={\rm id}_{V_i}$ for all $1\leqslant i\leqslant n$.
	Thus, $\theta=\{\theta_{ij}\}_{i, j\in I}$ induces an algebra isomorphism
	$${\rm Ind}_{\theta}:\oplus_{i,i',j,j'=1}^n
	{\rm Hom}_{\mathcal C}(V_i\otimes V_j, V_{i'}\otimes V_{j'})\ra
	\oplus_{i,i',j,j'=1}^n{\rm Hom}_{\mathcal C}(\oplus_{k=1}^nc_{ijk}V_k, \oplus_{k=1}^nc_{i'j'k}V_k)$$
	in an obvious way, where the multiplication of $\oplus_{i,i',j,j'=1}^n
	{\rm Hom}_{\mathcal C}(\oplus_{k=1}^nc_{ijk}V_k, \oplus_{k=1}^nc_{i'j'k}V_k)$
	is defined in a similar way as that of $\oplus_{i,i',j,j'=1}^n
	{\rm Hom}_{\mathcal C}(V_i\otimes V_j, V_{i'}\otimes V_{j'})$.
	Under the identifications
	$${\rm Hom}_{\mathcal C}(V_i, V_j)=e_jA(\mathcal C)e_i, \ 1\leqslant i, j\leqslant n,$$
	we may identify
	$${\rm Hom}_{\mathcal C}(\oplus_{k=1}^nc_{ijk}V_k, \oplus_{k=1}^nc_{i'j'k}V_k)
	=M_{\mathbf{c}_{i'j'}\times \mathbf{c}_{ij}}(A(\mathcal C))$$
	as stated in the proof of Proposition \ref{KerCok}. In this case, we have
	$$\oplus_{i,i',j,j'=1}^n{\rm Hom}_{\mathcal C}(\oplus_{k=1}^nc_{ijk}V_k, \oplus_{k=1}^nc_{i'j'k}V_k)
	=\oplus_{i,i',j,j'=1}^nM_{\mathbf{c}_{i'j'}\times \mathbf{c}_{ij}}(A(\mathcal C))
	=M(\mathcal C)$$ as $\mathbb{F}$-algebras. Thus, one gets an algebra map
	$$\phi_{\mathcal C}={\rm Ind}_{\theta}\circ\phi_2\circ\phi_1:
	A(\mathcal C)\ot_{\mathbb F}A(\mathcal C)\ra M(\mathcal C)$$
	such that $\phi_{\mathcal C}(e_i\ot_{\mathbb F}e_j)=E_{\mathbf{c}_{ij}}\in
	M_{\mathbf{c}_{ij}}(A(\mathcal C))$ for all $i, j\in I$.
	Moreover, we have
	$\phi_{\mathcal C}(e_1\ot_{\mathbb F}a)=a$ and $\phi_{\mathcal C}(a\ot_{\mathbb F}e_1)=a$ for any $a\in e_iA(\mathcal C)e_j$,
	where $1\leqslant i, j\leqslant n$. Note that $M_{\mathbf{c}_{1i}\times \mathbf{c}_{1j}}(A(\mathcal C))
	=M_{\mathbf{e}_{i}\times \mathbf{e}_{j}}(A(\mathcal C))=M_{1\times 1}(e_iA(\mathcal C)e_j)=e_iA(\mathcal C)e_j$,
	and similarly $M_{\mathbf{c}_{i1}\times \mathbf{c}_{j1}}(A(\mathcal C))=e_iA(\mathcal C)e_j$.
\end{proof}

\begin{remark} \label{2.9}
	By Proposition \ref{tom}, one gets a quadruple data $(r(\mathcal C), A(\mathcal C), I, \phi_{\mathcal C})$
	attached to $\mathcal C$. Note that $\phi_{\mathcal C}$ depends on the choices of the isomorphisms
	$\theta=\{\theta_{ij}\}_{i, j\in I}$. One can check that if $\theta'=\{\theta'_{ij}\}_{i, j\in I}$ is another
	family of isomorphisms and $\phi'_{\mathcal C}={\rm Ind}_{\theta'}\circ\phi_2\circ\phi_1$, then
	there is an inner automorphism $\psi$ of $M(\mathcal C)$ such that
	$\phi'_{\mathcal C}=\psi\circ\phi_{\mathcal C}$.
\end{remark}

Let $\mathbf{m}=(m_i)_{i\in I}\in\mathbb{N}^I$. For any $i\in I$ with $m_i>0$, and any $1\leqslant k\leqslant m_i$,
define a matrix $Y^{\mathbf m}_{i,k}=(y_1, y_2, \cdots, y_{|\mathbf m|})\in M_{\mathbf{e}_i\times\mathbf{m}}(A(\mathcal C))$ by
$$y_j=\left\{\begin{array}{ll}
e_i,& j=k+\sum_{1\leqslant l<i}m_l,\\
0,& \text{ otherwise},\\
\end{array}\right.$$
where we regard $\sum_{1\leqslant l<i}m_l=0$ when $i=1$. Let
$X^{\mathbf m}_{i,k}=(Y^{\mathbf m}_{i,k})^T$, the transposed matrix of $Y^{\mathbf m}_{i,k}$.
Then $X^{\mathbf m}_{i,k}\in M_{\mathbf{m}\times\mathbf{e}_i}(A(\mathcal C))$. We have
$$Y^{\mathbf m}_{i,k}X^{\mathbf m}_{i',k'}
=\left\{\begin{array}{ll}
e_i=E_{\mathbf{e}_i},& (i,k)=(i',k'),\\
0,& \text{ otherwise},\\
\end{array}\right.$$
and $\sum_{i=1}^n\sum_{1\leqslant k\leqslant m_i}X^{\mathbf m}_{i,k}Y^{\mathbf m}_{i,k}
=\sum_{1\leqslant i\leqslant n,  m_i>0}\sum_{k=1}^{m_i}X^{\mathbf m}_{i,k}Y^{\mathbf m}_{i,k}=E_{\mathbf m}$.

In what follows, for any $\mathbf{m}=(m_i)_{1\leqslant i\leqslant n}$,
$\mathbf{s}=(s_i)_{1\leqslant i\leqslant n}\in{\mathbb N}^I$,
we will identify
$${\rm Hom}_{\mathcal C}(\oplus_{i=1}^nm_iV_i, \oplus_{i=1}^ns_iV_i)
=M_{\mathbf{s}\times \mathbf{m}}(A(\mathcal C))$$
as stated in the proof of Proposition \ref{tom}.
Explicitly, for $f\in {\rm Hom}_{\mathcal C}(\oplus_{i=1}^nm_iV_i, \oplus_{i=1}^ns_iV_i)$,
the corresponding matrix
$$f=\left(\begin{array}{cccc}
f_{11}&f_{12}&\cdots&f_{1m}\\
f_{21}&f_{22}&\cdots&f_{2m}\\
\cdots&\cdots&\cdots&\cdots\\
f_{s1}&f_{s2}&\cdots&f_{sm}\\
\end{array}\right)$$
is determined by $f_{kl}=Y^{\mathbf s}_{i,k_1}fX^{\mathbf m}_{j,k_2}$
if $k=\sum_{1\leqslant t<i}s_t+k_1$ and $l=\sum_{1\leqslant t<j}m_t+k_2$,
where $s=|\mathbf s|$, $m=|\mathbf m|$, $1\leqslant k\leqslant s$,
$1\leqslant l\leqslant m$.
Under the identification, $Y^{\mathbf s}_{i,k_1}\in{\rm Hom}_{\mathcal C}(\oplus_{j=1}^ns_jV_j, V_i)$
is the projection from
$\oplus_{j=1}^ns_jV_j$ to the $k_1$-th $V_i$ of the direct summand $s_iV_i$ of
$\oplus_{j=1}^ns_jV_j$, and $X^{\mathbf m}_{j,k_2}\in{\rm Hom}_{\mathcal C}(V_j, \oplus_{i=1}^nm_iV_i)$
is the embedding of $V_j$ into the $k_2$-th $V_j$ of the direct summand $m_jV_j$ in
$\oplus_{i=1}^nm_iV_i$.

Let $\mathbf{m}_1=(m_{1i})_{1\leqslant i\leqslant n}, \mathbf{m}_2=(m_{2i})_{1\leqslant i\leqslant n},
\mathbf{s}_1=(s_{1i})_{1\leqslant i\leqslant n}, \mathbf{s}_2=(s_{2i})_{1\leqslant i\leqslant n}
\in{\mathbb N}^I$, and let $X\in M_{\mathbf{s}_1\times\mathbf{m}_1}(A(\mathcal C))$
and $Y\in M_{\mathbf{s}_2\times\mathbf{m}_2}(A(\mathcal C))$.
Then we have
$$X\in{\rm Hom}_{\mathcal C}(\oplus_{i=1}^nm_{1i}V_i, \oplus_{i=1}^ns_{1i}V_i)
\mbox{ and } Y\in{\rm Hom}_{\mathcal C}(\oplus_{i=1}^nm_{2i}V_i, \oplus_{i=1}^ns_{2i}V_i)$$
as stated above. Let $X\ot Y$ denote the tensor product of the morphisms
$X$ and $Y$ in $\mathcal C$. Then we have
$$X\ot Y\in{\rm Hom}_{\mathcal C}((\oplus_{i=1}^nm_{1i}V_i)\ot(\oplus_{i=1}^nm_{2i}V_i),
(\oplus_{i=1}^ns_{1i}V_i)\ot(\oplus_{i=1}^ns_{2i}V_i)).$$

For any homogeneous matrix $X=(x_{ij})$ over $A(\mathcal C)\ot_{\mathbb F}A(\mathcal C)$, one can see that
$(\phi_{\mathcal C}(x_{ij}))$ is a well-defined homogeneous matrix over $A(\mathcal C)$, denoted by $\phi_{\mathcal C}(X)$.

For any $\mathbf{m}=(m_i)_{1\leqslant i\leqslant n}$ and $\mathbf{s}=(s_i)_{1\leqslant i\leqslant n}$ in $\mathbb{N}^I$,
define $\mathbf{m}\whot\mathbf{s}\in\mathbb{N}^I$ by
$$\mathbf{m}\whot\mathbf{s}:=\sum_{i, j=1}^nm_is_j\mathbf{c}_{ij}.$$
%$$=(\sum_{i, j\in I}m_is_jc_{ijk})_{k\in I}=(\sum_{i, j\in I}m_is_jc_{ij1},
%\sum_{i, j\in I}m_is_jc_{ij2}, \cdots, \sum_{i, j\in I}m_is_jc_{ijn}).$$
Then $\mathbf{m}\whot\mathbf{s}=\mathbf{0}$ if and only if $\mathbf{m}=\mathbf{0}$ or $\mathbf{s}=\mathbf{0}$.
Moreover, $\mathbf{m}\whot\mathbf{e}_1=\mathbf{e}_1\whot\mathbf{m}=\mathbf{m}$.
One can easily check  that by Eq.$(*)$  $(\mathbf{m}\whot\mathbf{s})\whot\mathbf{t}
=\mathbf{m}\whot(\mathbf{s}\whot\mathbf{t})$ holds for all triples $(\mathbf{m}, \mathbf{s}, \mathbf{t})$ of elements of $\mathbb{N}^I$.

For any $X\in M_{\mathbf{s}_1\times\mathbf{m}_1}(A(\mathcal C))$
and $Y\in M_{\mathbf{s}_2\times\mathbf{m}_2}(A(\mathcal C))$,
define a homogeneous matrix $X\whot Y$ over $A(\mathcal C)$ as follows:\\
\mbox{\hspace{0.2cm}}(1) if $\mathbf{m}_1\whot\mathbf{m}_2=\mathbf{0}$ and $\mathbf{s}_1\whot\mathbf{s}_2=\mathbf{0}$
then $X\whot Y:=0\in M_{1\times 1}(A(\mathcal C))$;\\
\mbox{\hspace{0.2cm}}(2) if $\mathbf{m}_1\whot\mathbf{m}_2=\mathbf{0}$ and $\mathbf{s}_1\whot\mathbf{s}_2\neq\mathbf{0}$
then $X\whot Y:=0\in M_{|\mathbf{s}_1\whot\mathbf{s}_2|\times 1}(A(\mathcal C))$;\\
\mbox{\hspace{0.2cm}}(3) if $\mathbf{m}_1\whot\mathbf{m}_2\neq\mathbf{0}$ and $\mathbf{s}_1\whot\mathbf{s}_2=\mathbf{0}$
then $X\whot Y:=0\in M_{1\times|\mathbf{m}_1\whot\mathbf{m}_2|}(A(\mathcal C))$;\\
\mbox{\hspace{0.2cm}}(4) if $\mathbf{m}_1\whot\mathbf{m}_2\neq\mathbf{0}$ and $\mathbf{s}_1\whot\mathbf{s}_2\neq\mathbf{0}$
then $X\whot Y:=\Pi(\phi_{\mathcal C}(X\ot_{\mathbb F}Y))$.

\begin{lemma}\label{maptensor}
	With the above notations, we have $X\whot Y\in M_{(\mathbf{s}_1\whot\mathbf{s}_2)\times(\mathbf{m}_1\whot\mathbf{m}_2)}(A(\mathcal C))$.
\end{lemma}

\begin{proof}
	If $\mathbf{m}_1\whot\mathbf{m}_2=\mathbf{0}$ or $\mathbf{s}_1\whot\mathbf{s}_2=\mathbf{0}$,
	then obviously	$X\whot Y\in M_{(\mathbf{s}_1\whot\mathbf{s}_2)\times(\mathbf{m}_1\whot\mathbf{m}_2)}(A(\mathcal C))$.
	Now assume that $\mathbf{m}_1\whot\mathbf{m}_2\neq\mathbf{0}$ and $\mathbf{s}_1\whot\mathbf{s}_2\neq\mathbf{0}$.
	Then by the definition of $X\ot_{\mathbb F}Y$ and Lemma \ref{homo}, $X\ot_{\mathbb F}Y$ is a homogeneous
	$|\mathbf{s}_1||\mathbf{s}_2|\times|\mathbf{m}_1||\mathbf{m}_2|$-matrix over $A(\mathcal C)\ot_{\mathbb F}A(\mathcal C)$.
	Hence we may assume $X\ot_{\mathbb F}Y=(a_{lk})_{s\times m}$ with $a_{lk}\in A(\mathcal C)\ot_{\mathbb{F}}A(\mathcal C)$,
	where $s=|\mathbf{s}_1||\mathbf{s}_2|$ and $m=|\mathbf{m}_1||\mathbf{m}_2|$.
	Let $\mathbf{m}_i=(m_{i1}, m_{i2}, \cdots, m_{in})$ and $\mathbf{s}_i=(s_{i1}, s_{i2}, \cdots, s_{in})$,
	$i=1, 2$. Then by the definition of $X\ot_{\mathbb F}Y$, one knows that
	$a_{lk}\in e_{i'}A(\mathcal C)e_i\ot_{\mathbb{F}}e_{j'}A(\mathcal C)e_j\subset A(\mathcal C)\ot_{\mathbb{F}}A(\mathcal C)$ for
	$$|\mathbf{s}_1|(\sum_{0<t<j'}s_{2t}+u)+\sum_{0<t<i'}s_{1t}<l\leqslant
	|\mathbf{s}_1|(\sum_{0<t<j'}s_{2t}+u)+\sum_{0<t\leqslant i'}s_{1t}$$
	with $0\leqslant u\leqslant s_{2j'}-1$, and
	$$|\mathbf{m}_1|(\sum_{0<t<j}m_{2t}+u)+\sum_{0<t<i}m_{1t}<k\leqslant
	|\mathbf{m}_1|(\sum_{0<t<j}m_{2t}+u)+\sum_{0<t\leqslant i}m_{1t}$$
	with $0\leqslant u\leqslant m_{2j}-1$,
	where $1\leqslant i, i', j, j'\leqslant n$, and  $\sum_{0<t<1}s_{gt}=\sum_{0<t<1}m_{gt}=0$
	by convention for $g=1, 2$.
	Let $\mathbf{h}_1, \mathbf{h}_2, \cdots, \mathbf{h}_s$ be a series of elements of $\mathbb{N}^I$
	defined by $\mathbf{h}_l=\mathbf{c}_{i'j'}$ for
	$$|\mathbf{s}_1|(\sum_{0<t<j'}s_{2t}+u)+\sum_{0<t<i'}s_{1t}<l\leqslant
	|\mathbf{s}_1|(\sum_{0<t<j'}s_{2t}+u)+\sum_{0<t\leqslant i'}s_{1t},$$
	where $0\leqslant u\leqslant s_{2j'}-1$,
	and let $\mathbf{t}_1, \mathbf{t}_2, \cdots, \mathbf{t}_m$ be a series of elements of $\mathbb{N}^I$
	defined by $\mathbf{t}_k:=\mathbf{c}_{ij}$ for
	$$|\mathbf{m}_1|(\sum_{0<t<j}m_{2t}+u)+\sum_{0<t<i}m_{1t}<k\leqslant
	|\mathbf{m}_1|(\sum_{0<t<j}m_{2t}+u)+\sum_{0<t\leqslant i}m_{1t},$$
	where $0\leqslant u\leqslant m_{2j}-1$.
	Then by Proposition \ref{tom}(1), we have $\phi_{\mathcal C}(a_{lk})\in M_{\mathbf{h}_l\times \mathbf{t}_k}(A(\mathcal C))$ for all $1\leqslant l\leqslant s$
	and $1\leqslant k\leqslant m$. It follows that $\Pi(\phi_{\mathcal C}(X\ot_{\mathbb F}Y))$ is
	an $(\mathbf{s}, \mathbf{m})$-type matrix over $A(\mathcal C)$, where $\mathbf{s}:=\sum_{i=1}^s\mathbf{h}_i$
	and $\mathbf{m}:=\sum_{i=1}^m\mathbf{t}_i$.
	Finally, a straightforward computation shows that
	$\mathbf{s}=\sum_{i=1}^s\mathbf{h}_i=\sum_{i', j'=1}^ns_{1i'}s_{2j'}\mathbf{c}_{i'j'}=\mathbf{s}_1\whot\mathbf{s}_2$
	and $\mathbf{m}=\sum_{i=1}^m\mathbf{t}_i=\sum_{i, j=1}^nm_{1i}m_{2j}\mathbf{c}_{ij}=\mathbf{m}_1\whot\mathbf{m}_2$.
	Therefore, $X\whot Y\in M_{(\mathbf{s}_1\whot\mathbf{s}_2)\times(\mathbf{m}_1\whot\mathbf{m}_2)}(A(\mathcal C))$.
\end{proof}

\begin{remark}\label{permu}
	With the notations as in the proof of Lemma \ref{maptensor},
	let $P(\mathbf{s}_1, \mathbf{s}_2):=P_{\mathbf{h}_1, \mathbf{h}_2, \cdots, \mathbf{h}_s}$
	and $P(\mathbf{m}_1, \mathbf{m}_2):=P_{\mathbf{t}_1, \mathbf{t}_2, \cdots, \mathbf{t}_m}$
	be the permutation matrices defined as in the previous section. Then
	$$X\whot Y=P(\mathbf{s}_1, \mathbf{s}_2)\phi_{\mathcal C}(X\ot_{\mathbb F}Y)P(\mathbf{m}_1, \mathbf{m}_2)^T.$$
	Furthermore, Let $\ol{P}(\mathbf{s}_1, \mathbf{s}_2)=E_{\mathbf{s}_1\whot\mathbf{s}_2}P(\mathbf{s}_1, \mathbf{s}_2)$
	and $\ol{P}(\mathbf{m}_1, \mathbf{m}_2)=E_{\mathbf{m}_1\whot\mathbf{m}_2}P(\mathbf{m}_1, \mathbf{m}_2)$.
	Therefore,
	$$X\whot Y=\ol{P}(\mathbf{s}_1, \mathbf{s}_2)\phi_{\mathcal C}(X\ot_{\mathbb F}Y)\ol{P}(\mathbf{m}_1, \mathbf{m}_2)^T.$$
	
Note that  without the assumption that $V_i\ot V_j\neq 0$, $1\leqslant i, j\leqslant n$,  we would need to modify 	the definition of $X\whot Y:=\Pi(\phi_{\mathcal C}(X\ot_{\mathbb F}Y))$ when $\mathbf{m}_1\whot\mathbf{m}_2\neq\mathbf{0}$ and $\mathbf{s}_1\whot\mathbf{s}_2\neq\mathbf{0}$. In this case,
	whenever $\mathbf{h}_l=\mathbf{c}_{i'j'}=\mathbf{0}$ (resp., $\mathbf{t}_k:=\mathbf{c}_{ij}=\mathbf{0}$)
	for some $1\leqslant l\leqslant s$ (resp., $1\leqslant k\leqslant m$),
	one deletes the $l$-th row sub-matrices $(\phi_{\mathcal C}(a_{lk}))_{1\leqslant k\leqslant m}$
	(resp., $k$-th column $(\phi_{\mathcal C}(a_{lk}))_{1\leqslant l\leqslant s}$) in the matrix
	$\phi_{\mathcal C}(X\otimes_{\mathbb F}Y)=(\phi_{\mathcal C}(a_{lk}))$. The modified matrix is also denoted
	$\phi_{\mathcal C}(X\otimes_{\mathbb F}Y)$ simply. Then by rearranging the rows and columns of the matrix,
	one defines $X\whot Y:=\Pi(\phi_{\mathcal C}(X\otimes_{\mathbb F}Y))$.
	Moreover, let $\mathbf{h}_{i_1}, \mathbf{h}_{i_2}, \cdots, \mathbf{h}_{i_{s'}}$
	be all the nonzero elements of $\mathbf{h}_1, \mathbf{h}_2, \cdots, \mathbf{h}_s$
	with $1\leqslant i_1<i_2<\cdots<i_{s'}\leqslant s$,
	and let $\mathbf{t}_{j_1}, \mathbf{t}_{j_2}, \cdots, \mathbf{t}_{j_{m'}}$
	be all the nonzero elements of $\mathbf{t}_1, \mathbf{t}_2, \cdots, \mathbf{t}_m$
	with $1\leqslant j_1<j_2<\cdots<j_{m'}\leqslant m$.
	Put $P(\mathbf{s}_1, \mathbf{s}_2):=P_{\mathbf{h}_{i_1}, \mathbf{h}_{i_2}, \cdots, \mathbf{h}_{i_{s'}}}$
	and $P(\mathbf{m}_1, \mathbf{m}_2):=P_{\mathbf{t}_{j_1}, \mathbf{t}_{j_2}, \cdots, \mathbf{t}_{j_{m'}}}$
	 the permutation matrices defined as in the previous section. Then
	$$X\whot Y=P(\mathbf{s}_1, \mathbf{s}_2)\phi_{\mathcal C}(X\ot_{\mathbb F}Y)P(\mathbf{m}_1, \mathbf{m}_2)^T.$$
\end{remark}

\begin{corollary}\label{identitymap}
	Let $\mathbf m$ and $\mathbf s$ be any elements of $\mathbb{N}^I$. Then\\
	$(1)$ $E_{\mathbf m}\whot  E_{\mathbf s}=E_{\mathbf{m}\whot\mathbf{s}}$.\\
	$(2)$ $E_{\mathbf{e}_1}\whot  X=X\whot E_{\mathbf{e}_1}=X$ for any $X\in M_{\mathbf{m}\times\mathbf{s}}(A(\mathcal C))$.
\end{corollary}

\begin{proof}
	Follows from Proposition \ref{tom}(2) and the proof of Lemma \ref{maptensor}.
\end{proof}

\begin{lemma}\label{tensorcomp}
	Let $X\in M_{\mathbf{s}_1\times\mathbf{m}_1}(A(\mathcal C))$,
	$Y\in M_{\mathbf{s}_2\times\mathbf{m}_2}(A(\mathcal C))$,
	$X_1\in M_{\mathbf{m}_1\times\mathbf{t}_1}(A(\mathcal C))$
	and $Y_1\in M_{\mathbf{m}_2\times\mathbf{t}_2}(A(\mathcal C))$. Then
	$(X\whot Y)(X_1\whot Y_1)=(XX_1)\whot (YY_1)$.
\end{lemma}

\begin{proof}
	Follows from Lemma \ref{composition} and Remark \ref{permu}.
\end{proof}

\begin{lemma}\label{diag1}
	Let $i,i',j,j'\in I$, and let $x\in e_{i'}A(\mathcal C)e_i$
	and $y\in e_{j'}A(\mathcal C)e_j$. The following diagram commutes:
	$$\begin{array}{rcl}
	V_i\ot V_j&\xrightarrow{\theta_{ij}}
	&\oplus_{k=1}^nc_{ijk}V_k\\
	x\ot y\downarrow\mbox{\hspace{0.4cm}}&&\hspace{0.5cm}\downarrow x\whot y\\
	V_{i'}\ot V_{j'}&\xrightarrow{\theta_{i'j'}}
	&\oplus_{k=1}^nc_{i'j'k}V_k\\
	\end{array}.$$
\end{lemma}

\begin{proof}
	Follows from the definition of $\phi_{\mathcal C}$,
	since $x\whot y=\phi_{\mathcal C}(x\ot_{\mathbb F}y)$ in this case.
\end{proof}

For any $\mathbf{m}=(m_i)_{1\leqslant i\leqslant n}\in{\mathbb N}^I$, denote by $V^{(\mathbf{m})}$
the (ordered) direct sum $\oplus_{i=1}^nm_iV_i$ of objects of $\mathcal C$.
Then  $V^{(\mathbf{0})}=0$ and $V^{(\mathbf{e}_i)}=V_i$ for all $i\in I$.
Moreover, if $\mathbf{m}\neq\mathbf{0}$, then
$X^{\mathbf m}_{i,k}\in{\rm Hom}_{\mathcal C}(V_i, V^{(\mathbf{m})})$
and $Y^{\mathbf m}_{i,k}\in{\rm Hom}_{\mathcal C}(V^{(\mathbf{m})}, V_i)$
for any $1\leqslant i\leqslant n$ with $m_i>0$ and $1\leqslant k\leqslant m_i$.

For $\mathbf{m}_1=(m_{1i})_{1\leqslant i\leqslant n},
\mathbf{m}_2=(m_{2i})_{1\leqslant i\leqslant n}\in{\mathbb N}^I$,
define a morphism $\theta(\mathbf{m}_1, \mathbf{m}_2)$ from $V^{(\mathbf{m}_1)}\ot V^{(\mathbf{m}_2)}$
to $V^{(\mathbf{m}_1\whot\mathbf{m}_2)}$ in $\mathcal C$ by $\theta(\mathbf{m}_1, \mathbf{m}_2)=0$
if $\mathbf{m}_1=\mathbf{0}$ or $\mathbf{m}_2=\mathbf{0}$, and
$$\theta(\mathbf{m}_1, \mathbf{m}_2)
=\sum\limits_{i,j=1}^n\sum\limits_{1\leqslant k_1\leqslant m_{1i}}\sum\limits_{1\leqslant k_2\leqslant m_{2j}}
(X^{\mathbf{m}_1}_{i,k_1}\whot X^{\mathbf{m}_2}_{j,k_2})\theta_{ij}
(Y^{\mathbf{m}_1}_{i,k_1}\ot Y^{\mathbf{m}_2}_{j,k_2})$$
if $\mathbf{m}_1\neq\mathbf{0}$ and $\mathbf{m}_2\neq\mathbf{0}$.
Then $\theta(\mathbf{e}_i, \mathbf{e}_j)=\theta_{ij}$ for all $i,j\in I$.

\begin{lemma}\label{invert}
	$\theta(\mathbf{m}_1, \mathbf{m}_2)$ is an isomorphism in $\mathcal C$ for any
	$\mathbf{m}_1, \mathbf{m}_2\in\mathbb{N}^I$.
	Moreover, $\theta(\mathbf{e}_1, \mathbf{m})=\theta(\mathbf{m}, \mathbf{e}_1)=E_{\mathbf m}$.
\end{lemma}

\begin{proof}
	If $\mathbf{m}_1=\mathbf{0}$ or $\mathbf{m}_2=\mathbf{0}$, then $\theta(\mathbf{m}_1, \mathbf{m}_2)$
	is an isomorphism since $V^{(\mathbf{m}_1)}\ot V^{(\mathbf{m}_2)}=V^{(\mathbf{m}_1\whot\mathbf{m}_2)}=0$
	in this case. If $\mathbf{m}_1\neq\mathbf{0}$ and $\mathbf{m}_2\neq\mathbf{0}$,
	then it follows from Corollary \ref{identitymap} and Lemma \ref{tensorcomp} that $\theta(\mathbf{m}_1, \mathbf{m}_2)$
	is an isomorphism with the inverse given by
	$$\theta(\mathbf{m}_1, \mathbf{m}_2)^{-1}
	=\sum\limits_{i,j=1}^n\sum\limits_{1\leqslant k_1\leqslant m_{1i}}\sum\limits_{1\leqslant k_2\leqslant m_{2j}}
	(X^{\mathbf{m}_1}_{i,k_1}\ot X^{\mathbf{m}_2}_{j,k_2})\theta_{ij}^{-1}
	(Y^{\mathbf{m}_1}_{i,k_1}\whot Y^{\mathbf{m}_2}_{j,k_2}).$$
	The second statement follows from Corollary \ref{identitymap} and $\theta_{1i}=\theta_{i1}={\rm id}_{V_i}$ for any $i\in I$.
\end{proof}

\begin{lemma}\label{diag2}
	Let $\mathbf{m}_1=(m_{1i})_{i\in I}, \mathbf{m}_2=(m_{2i})_{i\in I},
	\mathbf{s}_1=(s_{1i})_{i\in I}, \mathbf{s}_2=(s_{2i})_{i\in I}\in {\mathbb N}^I$,
	and let $X\in M_{\mathbf{s}_1\times\mathbf{m}_1}(A(\mathcal C))$
	and $Y\in M_{\mathbf{s}_2\times\mathbf{m}_2}(A(\mathcal C))$. Then the following diagram commutes:
	$$\begin{array}{rcl}
	V^{(\mathbf{m}_1)}\ot V^{(\mathbf{m}_2)}&\xrightarrow{\theta(\mathbf{m}_1, \mathbf{m}_2)}
	&V^{(\mathbf{m}_1\whot\mathbf{m}_2)}\\
	X\ot Y\downarrow\mbox{\hspace{0.9cm}}&&\hspace{0.7cm}\downarrow X\whot Y\\
	V^{(\mathbf{s}_1)}\ot V^{(\mathbf{s}_2)}&\xrightarrow{\theta(\mathbf{s}_1, \mathbf{s}_2)}
	&V^{(\mathbf{s}_1\whot\mathbf{s}_2)}\\
	\end{array}.$$
	In particular, $\theta(\mathbf{m}_1, \mathbf{m}_2)(E_{\mathbf{m}_1}\ot E_{\mathbf{m}_2})
	=(E_{\mathbf{m}_1}\whot E_{\mathbf{m}_2})\theta(\mathbf{m}_1, \mathbf{m}_2)
	=E_{\mathbf{m}_1\whot\mathbf{m}_2}\theta(\mathbf{m}_1, \mathbf{m}_2)$.
\end{lemma}

\begin{proof}
	It is enough to show the lemma for $\mathbf{m}_i\neq \mathbf{0}$ and $\mathbf{s}_i\neq \mathbf{0}$,
	$i=1, 2$. In this case, by Lemmas \ref{tensorcomp} and \ref{diag1}, we have
	$$\begin{array}{rl}
	&(X\whot Y)\theta(\mathbf{m}_1, \mathbf{m}_2)\\
	=&\sum\limits_{i,j=1}^n\sum\limits_{1\leqslant k_1\leqslant m_{1i}}\sum\limits_{1\leqslant k_2\leqslant m_{2j}}
	(X\whot Y)(X^{\mathbf{m}_1}_{i,k_1}\whot X^{\mathbf{m}_2}_{j,k_2})\theta_{ij}
	(Y^{\mathbf{m}_1}_{i,k_1}\ot Y^{\mathbf{m}_2}_{j,k_2})\\
	=&\sum\limits_{i,j, i', j'=1}^n\sum\limits_{1\leqslant k_1\leqslant m_{1i}}\sum\limits_{1\leqslant k_2\leqslant m_{2j}}
	\sum\limits_{1\leqslant k'_1\leqslant s_{1i'}}\sum\limits_{1\leqslant k'_2\leqslant s_{2j'}}
	(X^{\mathbf{s}_1}_{i',k'_1}\whot X^{\mathbf{s}_2}_{j',k'_2})\\
	&(Y^{\mathbf{s}_1}_{i',k'_1}XX^{\mathbf{m}_1}_{i,k_1}\whot Y^{\mathbf{s}_2}_{j',k'_2}YX^{\mathbf{m}_2}_{j,k_2})\theta_{ij}
	(Y^{\mathbf{m}_1}_{i,k_1}\ot Y^{\mathbf{m}_2}_{j,k_2})\\
	=&\sum\limits_{i,j, i', j'=1}^n\sum\limits_{1\leqslant k_1\leqslant m_{1i}}\sum\limits_{1\leqslant k_2\leqslant m_{2j}}
	\sum\limits_{1\leqslant k'_1\leqslant s_{1i'}}\sum\limits_{1\leqslant k'_2\leqslant s_{2j'}}
	(X^{\mathbf{s}_1}_{i',k'_1}\whot X^{\mathbf{s}_2}_{j',k'_2})\\
	&\theta_{i'j'}(Y^{\mathbf{s}_1}_{i',k'_1}XX^{\mathbf{m}_1}_{i,k_1}\ot Y^{\mathbf{s}_2}_{j',k'_2}YX^{\mathbf{m}_2}_{j,k_2})
	(Y^{\mathbf{m}_1}_{i,k_1}\ot Y^{\mathbf{m}_2}_{j,k_2})\\
	=&\sum\limits_{i', j'=1}^n\sum\limits_{1\leqslant k'_1\leqslant s_{1i'}}\sum\limits_{1\leqslant k'_2\leqslant s_{2j'}}
	(X^{\mathbf{s}_1}_{i',k'_1}\whot X^{\mathbf{s}_2}_{j',k'_2})\theta_{i'j'}(Y^{\mathbf{s}_1}_{i',k'_1}\ot Y^{\mathbf{s}_2}_{j',k'_2})(X\ot Y)\\
	=&\theta(\mathbf{s}_1, \mathbf{s}_2)(X\ot Y).\\
	\end{array}$$
\end{proof}

For $i,j,l\in I$,  we define a matrix
$a_{i,j,l}\in M_{\mathbf{e}_i\whot\mathbf{e}_j\whot\mathbf{e}_l}(A(\mathcal C))={\rm End}_{\mathcal C}(V^{(\mathbf{e}_i\whot\mathbf{e}_j\whot\mathbf{e}_l)})$ by
$$\begin{array}{rl}
a_{i,j,l}:=&\theta(\mathbf{e}_i, \mathbf{e}_j\whot\mathbf{e}_l)
(E_{\mathbf{e}_i}\ot\theta(\mathbf{e}_j, \mathbf{e}_l))
(\theta(\mathbf{e}_i, \mathbf{e}_j)^{-1}\ot E_{\mathbf{e}_l})
\theta(\mathbf{e}_i\whot\mathbf{e}_j, \mathbf{e}_l)^{-1}\\
=&\theta(\mathbf{e}_i, \mathbf{c}_{jl})
(e_i\ot\theta_{jl})(\theta_{ij}^{-1}\ot e_l)
\theta(\mathbf{c}_{ij}, \mathbf{e}_l)^{-1}
\in{\rm End}_{\mathcal C}(V^{(\mathbf{e}_i\whot\mathbf{e}_j\whot\mathbf{e}_l)}).\\
\end{array}$$
The following proposition holds.

\begin{proposition}\label{Ass}
	$(1)$ $a_{i,j,l}$ is invertible in $M_{\mathbf{e}_i\whot\mathbf{e}_j\whot\mathbf{e}_l}(A(\mathcal C))$
	for all $i,j,l\in I$;\\
	$(2)$ $(x\whot(y\whot z))a_{i, j, l}
	=a_{i', j', l'}((x\whot y)\whot  z)$ for all $i,j,l,i',j',l'\in I$,
	$x\in e_{i'}A(\mathcal C)e_{i}$, $y\in e_{j'}A(\mathcal C)e_{j}$
	and $z\in e_{l'}A(\mathcal C)e_{l}$;\\
	$(3)$ $a_{i, 1, j}=E_{\mathbf{c}_{ij}}$
	for any $i, j\in I$.\\
	$(4)$\ \  $\sum\limits_{j=1}^n\sum\limits_{1\leqslant k\leqslant c_{i_2i_3j}}
	(e_{i_1}\whot  a_{i_2, i_3, i_4}(X^{\mathbf{c}_{i_2i_3}}_{j,k}\whot  e_{k_4}))
	a_{i_1,j,i_4}((e_{i_1}\whot  Y^{\mathbf{c}_{i_2i_3}}_{j,k})a_{i_1,i_2,i_3}\whot  e_{i_4})$\\
	\mbox{\hspace{0.6cm}}$=\sum\limits_{j,j'=1}^n\sum\limits_{1\leqslant k\leqslant c_{i_3i_4j}}
	\sum\limits_{1\leqslant k'\leqslant c_{i_1i_2j'}}
	(e_{i_1}\whot (e_{i_2}\whot  X^{\mathbf{c}_{i_3i_4}}_{j,k}))a_{i_1, i_2, j}
	(X^{\mathbf{c}_{i_1i_2}}_{j',k'}\whot  Y^{\mathbf{c}_{i_3i_4}}_{j,k})$\\
	\mbox{\hspace{5.2cm}}$\cdot a_{j',i_3,i_4}
	((Y^{\mathbf{c}_{i_1i_2}}_{j',k'}\whot  e_{i_3})\whot  e_{i_4})$\\
	for all $i_1, i_2, i_3, i_4\in I$.
\end{proposition}

\begin{proof}
	(1) Follows from Lemma \ref{invert}.
	
	(2) Since $\mathcal C$ is strict, $(x\ot y)\ot z=x\ot(y\ot z)$.
	Hence by Lemmas \ref{diag1} and \ref{diag2}, we have
	$$\begin{array}{rl}
	(x\whot(y\whot z))a_{i,j,l}
	=&(x\whot(y\whot z))\theta(\mathbf{e}_i, \mathbf{c}_{jl})
	(e_i\ot\theta_{jl}))(\theta_{ij})^{-1}\ot e_l)
	\theta(\mathbf{c}_{ij}, \mathbf{e}_l)^{-1}\\
	=&\theta(\mathbf{e}_{i'}, \mathbf{c}_{j'l'})(x\ot(y\whot z))
	(e_i\ot\theta_{jl})(\theta_{ij}^{-1}\ot e_l)
	\theta(\mathbf{c}_{ij}, \mathbf{e}_l)^{-1}\\
	=&\theta(\mathbf{e}_{i'}, \mathbf{c}_{j'l'})
	(e_{i'}\ot\theta_{j'l'})(x\ot y\ot z)
	(\theta_{ij}^{-1}\ot e_l)\theta(\mathbf{c}_{ij}, \mathbf{e}_l)^{-1}\\
	=&\theta(\mathbf{e}_{i'}, \mathbf{c}_{j'l'})
	(e_{i'}\ot\theta_{j'l'})(\theta_{i'j'}^{-1}\ot e_{l'})((x\whot y)\ot z)
	\theta(\mathbf{c}_{ij}, \mathbf{e}_l)^{-1}\\
	=&\theta(\mathbf{e}_{i'}, \mathbf{c}_{j'l'})
	(e_{i'}\ot\theta_{j'l'})(\theta_{i'j'}^{-1}\ot e_{l'})
	\theta(\mathbf{c}_{i'j'}, \mathbf{e}_{l'})^{-1}((x\whot y)\whot z)\\
	=&a_{i',j',l'}((x\whot y)\whot z).\\
	\end{array}$$
		
	(3) Since $\theta_{1i}=\theta_{i1}={\rm id}_{V_i}=e_i$ for all $i\in I$,
	$a_{i,1,j}=\theta_{ij}\theta_{ij}^{-1}={\rm id}_{\oplus_{k=1}^nc_{ijk}V_k}=E_{\mathbf{c}_{ij}}$
	for all $i, j\in I$.
	
	(4) By Lemma \ref{diag2}, we have
	$$\begin{array}{rl}
	&\sum\limits_{j=1}^n\sum\limits_{1\leqslant k\leqslant c_{i_2i_3j}}
	(e_{i_1}\widehat{\ot} a_{i_2, i_3, i_4}(X^{\mathbf{c}_{i_2i_3}}_{j,k}\widehat{\ot} e_{k_4}))
	a_{i_1,j,i_4}((e_{i_1}\widehat{\ot} Y^{\mathbf{c}_{i_2i_3}}_{j,k})a_{i_1,i_2,i_3}\widehat{\ot} e_{i_4})\\
	=&\sum\limits_{j=1}^n\sum\limits_{1\leqslant k\leqslant c_{i_2i_3j}}
	(e_{i_1}\widehat{\ot} a_{i_2, i_3, i_4})
	(e_{i_1}\widehat{\ot}(X^{\mathbf{c}_{i_2i_3}}_{j,k}\widehat{\ot} e_{k_4}))
	\theta(\mathbf{e}_{i_1}, \mathbf{c}_{ji_4})(e_{i_1}\ot\theta_{j,i_4})\\
	&(\theta_{i_1, j}^{-1}\ot e_{i_4})
	\theta(\mathbf{c}_{i_1j}, \mathbf{e}_{i_4})^{-1}
	((e_{i_1}\widehat{\ot} Y^{\mathbf{c}_{i_2i_3}}_{j,k})\widehat{\ot} e_{i_4})
	(a_{i_1,i_2,i_3}\widehat{\ot} e_{i_4})\\
	=&\sum\limits_{j=1}^n\sum\limits_{1\leqslant k\leqslant c_{i_2i_3j}}
	(e_{i_1}\widehat{\ot} a_{i_2, i_3, i_4})\theta(\mathbf{e}_{i_1}, \mathbf{c}_{i_2i_3}\whot\mathbf{e}_{i_4})
	(e_{i_1}\ot(X^{\mathbf{c}_{i_2i_3}}_{j,k}\widehat{\ot} e_{k_4}))
	(e_{i_1}\ot\theta_{j,i_4})\\
	&(\theta_{i_1, j}^{-1}\ot e_{i_4})
	((e_{i_1}\widehat{\ot} Y^{\mathbf{c}_{i_2i_3}}_{j,k})\ot e_{i_4})
	\theta(\mathbf{e}_{i_1}\whot\mathbf{c}_{i_2i_3}, \mathbf{e}_{i_4})^{-1}
	(a_{i_1,i_2,i_3}\widehat{\ot} e_{i_4})\\
	=&\sum\limits_{j=1}^n\sum\limits_{1\leqslant k\leqslant c_{i_2i_3j}}
	\theta(\mathbf{e}_{i_1}, \mathbf{c}_{i_2i_3}\widehat{\ot}\mathbf{e}_{i_4})(e_{i_1}\ot a_{i_2, i_3, i_4})
	(e_{i_1}\ot\theta(\mathbf{c}_{i_2i_3},\mathbf{e}_{i_4}))\\
	&(e_{i_1}\ot(X^{\mathbf{c}_{i_2i_3}}_{j,k}\ot e_{k_4}))
	((e_{i_1}\ot Y^{\mathbf{c}_{i_2i_3}}_{j,k})\ot e_{i_4})\\
	&(\theta(\mathbf{e}_{i_1},\mathbf{c}_{i_2i_3})^{-1}\ot e_{i_4})
	(a_{i_1,i_2,i_3}\ot e_{i_4})
	\theta(\mathbf{e}_{i_1}\whot\mathbf{c}_{i_2i_3}, \mathbf{e}_{i_4})^{-1}\\
	=&\theta(\mathbf{e}_{i_1}, \mathbf{c}_{i_2i_3}\widehat{\ot}\mathbf{e}_{i_4})(e_{i_1}\ot a_{i_2, i_3, i_4})
	(e_{i_1}\ot\theta(\mathbf{c}_{i_2i_3},\mathbf{e}_{i_4}))\\
	&(\theta(\mathbf{e}_{i_1},\mathbf{c}_{i_2i_3})^{-1}\ot e_{i_4})
	(a_{i_1,i_2,i_3}\ot e_{i_4})
	\theta(\mathbf{e}_{i_1}\whot\mathbf{c}_{i_2i_3}, \mathbf{e}_{i_4})^{-1}\\
	=&\theta(\mathbf{e}_{i_1}, \mathbf{c}_{i_2i_3}\whot\mathbf{e}_{i_4})
	(e_{i_1}\ot\theta(\mathbf{e}_{i_2}, \mathbf{c}_{i_3i_4}))(e_{i_1}\ot e_{i_2}\ot\theta_{i_3,i_4})\\
	&(\theta_{i_1, i_2}^{-1}\ot e_{i_3}\ot e_{i_4})
	(\theta(\mathbf{c}_{i_1i_2}, \mathbf{e}_{i_3})^{-1}\ot e_{i_4})
	\theta(\mathbf{e}_{i_1}\whot\mathbf{c}_{i_2i_3}, \mathbf{e}_{i_4})^{-1}\\
	\end{array}$$
	and
	$$\begin{array}{rl}
	&\sum\limits_{j,j'=1}^n\sum\limits_{1\leqslant k\leqslant c_{i_3i_4j}}
	\sum\limits_{1\leqslant k'\leqslant c_{i_1i_2j'}}
	(e_{i_1}\whot (e_{i_2}\whot  X^{\mathbf{c}_{i_3i_4}}_{j,k}))a_{i_1, i_2, j}
	(X^{\mathbf{c}_{i_1i_2}}_{j',k'}\whot  Y^{\mathbf{c}_{i_3i_4}}_{j,k})\\
	&\hspace{4cm}a_{j',i_3,i_4}((Y^{\mathbf{c}_{i_1i_2}}_{j',k'}\whot  e_{i_3})\whot  e_{i_4})\\
	=&\sum\limits_{j,j'=1}^n\sum\limits_{1\leqslant k\leqslant c_{i_3i_4j}}
	\sum\limits_{1\leqslant k'\leqslant c_{i_1i_2j'}}
	(e_{i_1}\whot (e_{i_2}\whot  X^{\mathbf{c}_{i_3i_4}}_{j,k}))
	\theta(\mathbf{e}_{i_1}, \mathbf{c}_{i_2j})
	(e_{i_1}\ot\theta_{i_2,j})\\
	&\hspace{1cm}(\theta_{i_1, i_2}^{-1}\ot e_{j})
	\theta(\mathbf{c}_{i_1i_2}, \mathbf{e}_{j})^{-1}
	(X^{\mathbf{c}_{i_1i_2}}_{j',k'}\whot  Y^{\mathbf{c}_{i_3i_4}}_{j,k})
	\theta(\mathbf{e}_{j'}, \mathbf{c}_{i_3i_4})
	(e_{j'}\ot\theta_{i_3i_4})\\
	&\hspace{1cm}(\theta_{j'i_3}^{-1}\ot e_{i_4})
	\theta(\mathbf{c}_{j'i_3}, \mathbf{e}_{i_4})^{-1}
	((Y^{\mathbf{c}_{i_1i_2}}_{j',k'}\whot  e_{i_3})\whot  e_{i_4})\\
	=&\sum\limits_{j,j'=1}^n\sum\limits_{1\leqslant k\leqslant c_{i_3i_4j}}
	\sum\limits_{1\leqslant k'\leqslant c_{i_1i_2j'}}
	\theta(\mathbf{e}_{i_1}, \mathbf{e}_{i_2}\whot\mathbf{c}_{i_3i_4})
	(e_{i_1}\ot(e_{i_2}\whot  X^{\mathbf{c}_{i_3i_4}}_{j,k}))
	(e_{i_1}\ot\theta_{i_2,j})\\
	&\hspace{1cm}(\theta_{i_1, i_2}^{-1}\ot e_{j})
	(X^{\mathbf{c}_{i_1i_2}}_{j',k'}\ot Y^{\mathbf{c}_{i_3i_4}}_{j,k})
	(e_{j'}\ot\theta_{i_3i_4})\\
	&\hspace{1cm}(\theta_{j'i_3}^{-1}\ot e_{i_4})
	((Y^{\mathbf{c}_{i_1i_2}}_{j',k'}\whot  e_{i_3})\ot e_{i_4})
	\theta(\mathbf{c}_{i_1i_2}\whot\mathbf{e}_{i_3}, \mathbf{e}_{i_4})^{-1}\\
	=&\sum\limits_{j,j'=1}^n\sum\limits_{1\leqslant k\leqslant c_{i_3i_4j}}
	\sum\limits_{1\leqslant k'\leqslant c_{i_1i_2j'}}
	\theta(\mathbf{e}_{i_1}, \mathbf{e}_{i_2}\whot\mathbf{c}_{i_3i_4})
	(e_{i_1}\ot\theta(\mathbf{e}_{i_2}, \mathbf{c}_{i_3i_4}))
	(\theta_{i_1, i_2}^{-1}\ot X^{\mathbf{c}_{i_3i_4}}_{j,k})\\
	&(X^{\mathbf{c}_{i_1i_2}}_{j',k'}\ot Y^{\mathbf{c}_{i_3i_4}}_{j,k})
	(Y^{\mathbf{c}_{i_1i_2}}_{j',k'}\ot \theta_{i_3,i_4})
	(\theta(\mathbf{c}_{i_1i_2}, \mathbf{e}_{i_3})^{-1}\ot e_{i_4})
	\theta(\mathbf{c}_{i_1i_2}\whot\mathbf{e}_{i_3}, \mathbf{e}_{i_4})^{-1}\\
	=&\theta(\mathbf{e}_{i_1}, \mathbf{e}_{i_2}\whot\mathbf{c}_{i_3i_4})
	(e_{i_1}\ot\theta(\mathbf{e}_{i_2}, \mathbf{c}_{i_3i_4}))(\theta_{i_1, i_2}^{-1}\ot\theta_{i_3,i_4})\\
	&(\theta(\mathbf{c}_{i_1i_2}, \mathbf{e}_{i_3})^{-1}\ot e_{i_4})
	\theta(\mathbf{c}_{i_1i_2}\whot\mathbf{e}_{i_3}, \mathbf{e}_{i_4})^{-1}\\
	=&\theta(\mathbf{e}_{i_1}, \mathbf{e}_{i_2}\whot\mathbf{c}_{i_3i_4})
	(e_{i_1}\ot\theta(\mathbf{e}_{i_2}, \mathbf{c}_{i_3i_4}))(e_{i_1}\ot e_{i_2}\ot\theta_{i_3,i_4})\\
	&(\theta_{i_1, i_2}^{-1}\ot e_{i_3}\ot e_{i_4})
	(\theta(\mathbf{c}_{i_1i_2}, \mathbf{e}_{i_3})^{-1}\ot e_{i_4})
	\theta(\mathbf{c}_{i_1i_2}\whot\mathbf{e}_{i_3}, \mathbf{e}_{i_4})^{-1}.\\
	\end{array}$$
	This completes the proof since $\mathbf{c}_{i_1i_2}\whot\mathbf{e}_{i_3}=\mathbf{e}_{i_1}\whot\mathbf{c}_{i_2i_3}$
	and $\mathbf{c}_{i_2i_3}\ot\mathbf{e}_{i_4}=\mathbf{e}_{i_2}\ot\mathbf{c}_{i_3i_4}$.
\end{proof}

\section{\bf Reconstruction of tensor categories}\selabel{2}

Throughout this section, assume that $R$ is a $\mathbb{Z}_+$-ring with a finite unital
$\mathbb{Z}_+$-basis $\{r_i\}_{i\in I}$ such that $r_ir_j\neq 0$ for all $i, j\in I$,
and $A$ is a finite dimensional $\mathbb{F}$-algebra
with a complete set of orthogonal primitive idempotents $\{e_i\}_{i\in I}$.
Assume $I=\{1, 2, \cdots, n\}$ and $r_1=1$, the identity of $R$.
Furthermore,  we assume that $A$ satisfies the following conditions:

(KS) For any $i\neq j$ in $I$, $e_iAe_jAe_i\subseteq{\rm rad}(e_iAe_i)$, the radical of the algebra $e_iAe_i$.

(Dec) Every $(\mathbf{m}, \mathbf{s})$-type matrix $X$ over $A$ can be decomposed into
a produpct $X=X_1X_2$ of a column-independent $(\mathbf{m}, \mathbf{t})$-type matrix  $X_1$
and a row-independent $(\mathbf{t}, \mathbf{s})$-type matrix $X_2$ over $A$.

(RUA) Every $(\mathbf{m}, \mathbf{s})$-type matrix over $A$ has a right universal annihilator.

(LUA) Every $(\mathbf{m}, \mathbf{s})$-type matrix over $A$ has a left universal annihilator.

(CI) If $X$ is column-independent and $Y$ is a left universal annihilator
of $X$, then $X$ is a right universal annihilator of $Y$.

(RI) If $X$ is row-independent and $Y$ is a right universal annihilator
of $X$, then $X$ is a left universal annihilator of $Y$.

Now we construct a category $\widehat{\mathcal C}$ as follows:\\
(1) Ob$(\widehat{\mathcal C}):={\mathbb N}^I$;\\
(2) ${\rm Mor}(\widehat{\mathcal C})$: for ${\mathbf m}=(m_i)_{1\leqslant i\leqslant n}$
and ${\mathbf s}=(s_i)_{1\leqslant i\leqslant n}$ in Ob$(\widehat{\mathcal C})={\mathbb N}^I$, define
$${\rm Hom}_{\widehat{\mathcal C}}({\mathbf m}, {\mathbf s}):=M_{{\mathbf s}\times{\mathbf m}}(A);$$
(3) Composition: for ${\mathbf s}=(s_i)_{1\leqslant i\leqslant n}$,
${\mathbf m}=(m_i)_{1\leqslant i\leqslant n}$, ${\mathbf t}=(t_i)_{1\leqslant i\leqslant n}\in{\mathbb N}^I$,
define the composition
$${\rm Hom}_{\widehat{\mathcal C}}({\mathbf m}, {\mathbf s})\times {\rm Hom}_{\widehat{\mathcal C}}({\mathbf t}, {\mathbf m})
\ra {\rm Hom}_{\widehat{\mathcal C}}({\mathbf t}, {\mathbf s}),\
(X, Y)\mapsto X\circ Y$$
by $X\circ Y=XY\in M_{{\mathbf s}\times{\mathbf t}}(A)={\rm Hom}_{\widehat{\mathcal C}}({\mathbf t}, {\mathbf s})$, the usual product of matrices.

\begin{lemma}\label{directsum}
Let $\mathbf{m}$, $\mathbf{s}\in{\rm Ob}(\widehat{\mathcal C})$. Then $\mathbf{m}+\mathbf{s}$
is a direct sum object of $\mathbf{m}$ and $\mathbf{s}$ in $\widehat{\mathcal C}$.
\end{lemma}

\begin{proof}
For any $1\leqslant i\leqslant n$, let $X_i=(I(i, m_i), 0)\in M_{m_i\times(m_i+s_i)}(e_iAe_i)$,
and $Y_i=(0, I(i, s_i))\in M_{s_i\times(m_i+s_i)}(e_iAe_i)$. Then
$X_iX_i^T=I(i,m_i)$, $Y_iY^T_i=I(i,s_i)$, $X_iY^T_i=0$, $Y_iX_i^T=0$
and $X_i^TX_i+Y_i^TY_i=I(i, m_i+s_i)$, where $Z^T$ denotes
the transpose of  matrix $Z$. Now let
$$X=\left(\begin{array}{cccc}
X_{1}&0&\cdots&0\\
0&X_{2}&\cdots&0\\
\cdots&\cdots&\cdots&\cdots\\
0&0&\cdots&X_{n}\\
\end{array}\right),   \  \ \
Y=\left(\begin{array}{cccc}
Y_{1}&0&\cdots&0\\
0&Y_{2}&\cdots&0\\
\cdots&\cdots&\cdots&\cdots\\
0&0&\cdots&Y_{n}\\
\end{array}\right),$$
where $X\in M_{\mathbf{m}\times(\mathbf{m}+\mathbf{s})}(A)
={\rm Hom}_{\mathcal C}(\mathbf{m}+\mathbf{s}, \mathbf{m})$
and
$Y\in M_{\mathbf{s}\times(\mathbf{m}+\mathbf{s})}(A)
={\rm Hom}_{\mathcal C}(\mathbf{m}+\mathbf{s}, \mathbf{s})$.
Then $X^T\in M_{(\mathbf{m}+\mathbf{s})\times\mathbf{m}}(A)
={\rm Hom}_{\widehat{\mathcal C}}(\mathbf{m}, \mathbf{m}+\mathbf{s})$
and
$Y^T\in M_{(\mathbf{m}+\mathbf{s})\times\mathbf{s}}(A)
={\rm Hom}_{\widehat{\mathcal C}}(\mathbf{s},  \mathbf{m}+\mathbf{s})$.
It is straightforward to check that $XX^T=E_{\mathbf m}$, $YY^T=E_{\mathbf s}$,
$XY^T=0$, $YX^T=0$, and $X^TX+Y^TY=E_{\mathbf{m}+\mathbf{s}}$.
This shows that $\mathbf{m}+\mathbf{s}\cong\mathbf{m}\oplus\mathbf{s}$ in $\widehat{\mathcal C}$.
\end{proof}

\begin{proposition}\label{line}
$\widehat{\mathcal C}$ is an additive category over $\mathbb F$.
\end{proposition}

\begin{proof}
Obviously, ${\rm Hom}_{\widehat{\mathcal C}}(\mathbf{m}, \mathbf{s})=M_{\mathbf{s}\times\mathbf{m}}(A)$
is a vector space over $\mathbb F$ for all
${\mathbf m}, {\mathbf s}\in{\rm Ob}(\widehat{\mathcal C})$, and the composition
$${\rm Hom}_{\widehat{\mathcal C}}({\mathbf m}, {\mathbf s})\times {\rm Hom}_{\widehat{\mathcal C}}({\mathbf t}, {\mathbf m})
\ra {\rm Hom}_{\widehat{\mathcal C}}({\mathbf t}, {\mathbf s}),\
(X, Y)\mapsto X\circ Y$$
is $\mathbb F$-bilinear for all ${\mathbf m}, {\mathbf s}, {\mathbf t}\in{\rm Ob}(\widehat{\mathcal C})$.
Moreover, ${\rm Hom}_{\widehat{\mathcal C}}(\mathbf{0}, {\mathbf m})=M_{\mathbf{m}\times\mathbf{0}}(A)=0$
and ${\rm Hom}_{\widehat{\mathcal C}}(\mathbf{m}, {\mathbf 0})=M_{\mathbf{0}\times\mathbf{m}}(A)=0$
for all $\mathbf{m}\in{\rm Ob}(\widehat{\mathcal C})$. Hence the object $\mathbf{0}$ is a zero object of $\widehat{\mathcal C}$.
It follows from Lemma \ref{directsum} that $\widehat{\mathcal C}$ is an additive category over $\mathbb F$.
\end{proof}

Note that an additive category over $\mathbb F$ is also called an $\mathbb F$-linear category.

\begin{remark}
Let $\mathbf{m}, \mathbf{s}\in{\rm Ob}(\widehat{\mathcal C})$ and $X\in{\rm Hom}_{\widehat{\mathcal C}}(\mathbf{m}, \mathbf{s})
=M_{\mathbf{s}\times\mathbf{m}}(A)$. Then $X$ is a monomorphism in $\widehat{\mathcal C}$ if and only if
$X$ is column-independent, and $X$ is an epimorphism in $\widehat{\mathcal C}$ if and only if
$X$ is row-independent.
\end{remark}

\begin{lemma}\label{iso}
Let ${\mathbf m}=(m_i)_{1\leqslant i\leqslant n}$, ${\mathbf s}=(s_i)_{1\leqslant i\leqslant n}\in {\mathbb N}^I$.
Then ${\mathbf m}\cong {\mathbf s}$ in $\widehat{\mathcal C}$ if and only if ${\mathbf m}={\mathbf s}$.
Consequently, $\mathbf{0}$ is the unique zero object of $\widehat{\mathcal C}$.
\end{lemma}

\begin{proof}
It is enough to show the necessity part.
Assume that ${\mathbf m}\cong {\mathbf s}$. Then there exist $X\in M_{\mathbf{m}\times\mathbf{s}}(A)$
and $Y\in M_{\mathbf{s}\times\mathbf{m}}(A)$ such that $XY=E_{\mathbf m}$ and $YX=E_{\mathbf s}$.
By (KS), an argument similar to the proof of Corollary \ref{KrullSch2} shows that ${\mathbf m}={\mathbf s}$.
\end{proof}

For any ${\mathbf m}\in{\rm Ob}(\widehat{\mathcal C})$ and a nonnegative integer $l$,
it follows from Lemmas \ref{directsum} and \ref{iso} that $l{\mathbf m}$ is exactly the direct sum of $l$ copies of
$\mathbf m$ in $\widehat{\mathcal C}$.

For any $1\leqslant i\leqslant n$, let ${\mathbf{e}}_i=(e_{i1}, e_{i2}, \cdots, e_{in})\in{\rm Ob}(\widehat{\mathcal C})$
 given by $e_{ij}=\d_{ij}$, the Kronecker symbols.
Then by the definition of $\widehat{\mathcal C}$, ${\rm End}_{\widehat{\mathcal C}}({\mathbf{e}}_i)\cong e_iAe_i$,
which is a local algebra since $e_i$ is a primitive idempotent.
Thus, by Lemma \ref{directsum}, one gets the following corollary.

\begin{corollary}\label{ks}
For any $\mathbf{m}=(m_1, m_2, \cdots, m_n)\in{\rm Ob}(\widehat{\mathcal C})$, we have
$$\mathbf{m}\cong m_1{\mathbf{e}}_1\oplus m_2{\mathbf{e}}_2\oplus\cdots\oplus m_n{\mathbf{e}}_n.$$
Moreover, ${\mathbf{e}}_i$ is an indecomposable object of $\widehat{\mathcal C}$ for any $1\leqslant i\leqslant n$.
\end{corollary}

\begin{corollary}\label{IndOb}
Let $\mathbf{m}\in{\rm Ob}(\widehat{\mathcal C})$. Then
$\mathbf{m}$ is an indecomposable object of $\widehat{\mathcal C}$
if and only if $|\mathbf{m}|=1$ if and only if $\mathbf{m}=\mathbf{e}_i$
for some $1\leqslant i\leqslant n$.
Moreover, $\mathbf{e}_1, \mathbf{e}_2, \cdots, \mathbf{e}_n$ are all
non-isomorphic indecomposable objects of $\widehat{\mathcal C}$.
\end{corollary}

\begin{proof}
It follows from Lemma \ref{iso} and Corollary \ref{ks}.
\end{proof}

\begin{proposition}\label{ksc}
$\widehat{\mathcal C}$ is a Krull-Schmidt abelian category over $\mathbb F$.
\end{proposition}

\begin{proof}
By Proposition \ref{line}, we  know that $\widehat{\mathcal C}$ is an additive category over $\mathbb F$.
By (RUA) and (LUA), one knows that every morphism $X\in{\rm Hom}_{\widehat{\mathcal C}}(\mathbf{m}, \mathbf{s})$
has a kernel ${\rm ker}(X)\in{\rm Mor}(\widehat{\mathcal C})$ and a cokernel ${\rm coker}(X)\in{\rm Mor}(\widehat{\mathcal C})$.
By (Dec), every morphism of $\widehat{\mathcal C}$ is a composition of an epimorphism followed by a monomorphism.
Note that ${\rm ker}(X)=0$ (resp. ${\rm coker}(X)=0$) if and only if $X$ is column-independent (resp.
row-independent). Thus, by (CI) and (RI), if ${\rm ker}(X)=0$ then $X={\rm ker}({\rm coker}(X))$;
if ${\rm coker}(X)=0$ then $X={\rm coker}({\rm ker}(X))$.
Therefore, $\widehat{\mathcal C}$ is an abelian category over $\mathbb F$.
Finally, it follows from Lemma \ref{iso} and Corollaries \ref{ks} and \ref{IndOb}
that $\widehat{\mathcal C}$ is a Krull-Schmidt category.
\end{proof}

\begin{corollary}\label{semisimple}
The following statements are equivalent:\\
{\rm (1)} $\widehat{\mathcal C}$ is a semisimple category over $\mathbb F$;\\
{\rm (2)} {\rm dim}$_{\mathbb F}(A)=n$;\\
{\rm (3)} $A\cong {\mathbb F}^n$, the direct product of $n$ copies of $\mathbb F$, as $\mathbb F$-algebra.
\end{corollary}

\begin{proof}
By Corollary \ref{IndOb}, $\mathbf{e}_1, \mathbf{e}_2, \cdots, \mathbf{e}_n$ are all
non-isomorphic indecomposable objects of $\widehat{\mathcal C}$.
Hence by Proposition \ref{ksc}, $\widehat{\mathcal C}$ is a semisimple category over $\mathbb F$
if and only if ${\rm End}_{\widehat{\mathcal C}}(\mathbf{e}_i)\cong\mathbb F$ and
${\rm Hom}_{\widehat{\mathcal C}}(\mathbf{e}_i, \mathbf{e}_j)=0$
for all $1\leqslant i\neq j\leqslant n$.
Since ${\rm Hom}_{\widehat{\mathcal C}}(\mathbf{e}_i, \mathbf{e}_j)=M_{\mathbf{e}_j\times\mathbf{e}_i}(A)=e_jAe_i$
and ${\rm dim}_{\mathbb F}(e_iAe_i)\geqslant 1$ for all $1\leqslant i, j\leqslant n$,
$\widehat{\mathcal C}$ is a semisimple category over $\mathbb F$
if and only if ${\rm dim}_{\mathbb F}(e_iAe_i)=1$ and ${\rm dim}_{\mathbb F}(e_jAe_i)=0$
for all $1\leqslant i\neq j\leqslant n$, which is equivalent to ${\rm dim}_{\mathbb F}(A)=n$.
This shows that (1) and (2) are equivalent, and that (1) (or (2)) implies (3).
It is obvious that (3) implies (2).
\end{proof}

Since $R$ is a $\mathbb{Z}_+$-ring with the unital $\mathbb{Z}_+$-basis
$\{r_i\}_{1\leqslant i\leqslant n}$, we have
$r_ir_j=\sum_{k=1}^nc_{ijk}r_k$ for some $c_{ijk}\in\mathbb{N}$, where $1\leqslant i, j\leqslant n$.
Let $\mathbf{c}_{ij}=(c_{ijk})_{1\leqslant k\leqslant n}=(c_{ij1}, c_{ij2}, \cdots, c_{ijn})$.
Then $\mathbf{c}_{ij}\in\mathbb{N}^I$
and $\mathbf{c}_{ij}\neq\mathbf{0}$ by $r_ir_j\neq 0$, $1\leqslant i, j\leqslant n$.
Moreover, $\mathbf{c}_{1i}=\mathbf{c}_{i1}=\mathbf{e}_i$
for all $1\leqslant i\leqslant n$.

Define  a vector space $M(R, A, I)$ over $\mathbb{F}$  by
$$M(R, A, I):=\oplus_{1\leqslant i,i',j,j'\leqslant n}M_{\mathbf{c}_{i'j'}\times\mathbf{c}_{ij}}(A).$$
Then $M(R, A, I)$ is an associative $\mathbb{F}$-algebra with the multiplication
defined  in Section 2.

In the rest of this section, we assume that dim$(e_1Ae_1)=1$, i.e., $e_1Ae_1=\mathbb{F}e_1$, and assume
that there is an $\mathbb F$-algebra map
$\phi: A\otimes_{\mathbb F}A\ra M(R, A, I)$ such that the following two conditions are satisfied:\\
\mbox{\hspace{0.2cm}} ($\phi$1) $\phi(e_i\otimes_{\mathbb F}e_j)=E_{\mathbf{c}_{ij}}\in
M_{\mathbf{c}_{ij}}(A)$ for all $1\leqslant i, j\leqslant n$;\\
\mbox{\hspace{0.2cm}} ($\phi$2) $\phi(e_1\ot_{\mathbb F}a)=a\in M_{\mathbf{c}_{1i}\times\mathbf{c}_{1j}}(A)$
	and $\phi(a\ot_{\mathbb F}e_1)=a\in M_{\mathbf{c}_{i1}\times\mathbf{c}_{j1}}(A)$
	for all $a\in e_iAe_j$, where $1\leqslant i, j\leqslant n$.\\
Then we have %$\phi(a\ot_{\mathbb F}b)=ab$ for any $a, b\in e_1Ae_1$, and
$\phi(x)\in M_{\mathbf{c}_{i'j'}\times\mathbf{c}_{ij}}(A)$ for any $x\in e_{i'}Ae_{i}\ot_{\mathbb F}e_{j'}Ae_j
\subset A\ot_{\mathbb F}A$, where $1\leqslant i', i, j', j\leqslant n$.

For any homogeneous matrix $X=(x_{ij})$ over $A\ot_{\mathbb F}A$, $\phi(X)=(\phi(x_{ij}))$
is a well-defined homogeneous matrix over $A$.
By the discussion in Section 2, one can define
$\mathbf{m}\whot\mathbf{s}:=\sum_{i, j=1}^nm_is_j\mathbf{c}_{ij}\in{\rm Ob}(\widehat{\mathcal C})$
for any $\mathbf{m}=(m_i)_{1\leqslant i\leqslant n}$ and $\mathbf{s}=(s_i)_{1\leqslant i\leqslant n}$ in Ob$(\widehat{\mathcal C})$.
Let $X\in{\rm Hom}_{\widehat{\mathcal C}}(\mathbf{m}_1, \mathbf{s}_1)$
and $Y\in{\rm Hom}_{\widehat{\mathcal C}}(\mathbf{m}_2, \mathbf{s}_2)$.
Then one also can define $X\whot Y\in {\rm Hom}_{\widehat{\mathcal C}}(\mathbf{m}_1\whot\mathbf{m}_2, \mathbf{s}_1\whot\mathbf{s}_2)$
by
$$X\whot Y:=\Pi(\phi(X\ot_{\mathbb F}Y))=P(\mathbf{s}_1, \mathbf{s}_2)\phi(X\ot_{\mathbb F}Y)P(\mathbf{m}_1, \mathbf{m}_2)^T$$
for $\mathbf{m}_1\whot\mathbf{m}_2\neq\mathbf{0}$ and $\mathbf{s}_1\whot\mathbf{s}_2\neq\mathbf{0}$,
and $X\whot Y:=0$ for otherwise (see Lemma \ref{maptensor} and Remark \ref{permu}). Corollary \ref{identitymap}
and Lemma \ref{tensorcomp} still hold in this case.

Furthermore, we make the following assumption:

{\bf (Ass)} There exists a family of matrices $a_{i,j,l}\in M_{\mathbf{e}_i\whot\mathbf{e}j\whot\mathbf{e}_l}(A)$
indexed by all triples $(i, j, l)$ of elements of $I$ such that
the following conditions are satisfied:\\
\mbox{\hspace{0.2cm}}(1) $a_{i,j,l}$ is invertible in $M_{\mathbf{e}_i\whot\mathbf{e}_j\whot\mathbf{e}_l}(A)$
for all $i,j,l\in I$;\\
\mbox{\hspace{0.2cm}}(2) $(x\whot(y\whot z))a_{i, j, l}
=a_{i', j', l'}((x\whot y)\whot  z)$
for all $i, j, l, i', j', l'\in I$,
$x\in e_{i'}Ae_{i}$, \\
\mbox{\hspace{0.6cm}} $y\in e_{j'}Ae_{j}$ and $z\in e_{l'}Ae_{l}$;\\
\mbox{\hspace{0.2cm}}(3) $a_{i, 1, j}=E_{\mathbf{c}_{ij}}$
for any $i, j\in I$.\\
\mbox{\hspace{0.2cm}}(4)\ \  $\sum\limits_{j=1}^n\sum\limits_{1\leqslant k\leqslant c_{i_2i_3j}}
(e_{i_1}\whot  a_{i_2, i_3, i_4}(X^{\mathbf{c}_{i_2i_3}}_{j,k}\whot  e_{i_4}))
a_{i_1,j,i_4}((e_{i_1}\whot  Y^{\mathbf{c}_{i_2i_3}}_{j,k})a_{i_1,i_2,i_3}\whot  e_{i_4})$\\
\mbox{\hspace{0.6cm}}$=\sum\limits_{j,j'=1}^n\sum\limits_{1\leqslant k\leqslant c_{i_3i_4j}}
\sum\limits_{1\leqslant k'\leqslant c_{i_1i_2j'}}
(e_{i_1}\whot (e_{i_2}\whot  X^{\mathbf{c}_{i_3i_4}}_{j,k}))a_{i_1, i_2, j}
(X^{\mathbf{c}_{i_1i_2}}_{j',k'}\whot  Y^{\mathbf{c}_{i_3i_4}}_{j,k})$\\
\mbox{\hspace{5.2cm}}$\cdot a_{j',i_3,i_4}
((Y^{\mathbf{c}_{i_1i_2}}_{j',k'}\whot  e_{i_3})\whot  e_{i_4})$\\
for all $i_1, i_2, i_3, i_4\in I$, where $X^{\mathbf{m}}_{j,k}$ and $Y^{\mathbf{m}}_{j,k}$
are the matrices defined  in the last section.\\

For any $\mathbf{m}=(m_i)_{i\in I}, \mathbf{s}=(s_i)_{i\in I}, \mathbf{t}=(t_i)_{i\in I}\in\mathbb{N}^I$,
define a matrix $a_{\mathbf{m}, \mathbf{s}, \mathbf{t}}\in M_{\mathbf{m}\whot\mathbf{s}\whot\mathbf{t}}(A)$ by
$a_{\mathbf{m}, \mathbf{s}, \mathbf{t}}=E_{\mathbf{0}}=0$ when $\mathbf{m}\whot\mathbf{s}\whot\mathbf{t}=\mathbf{0}$,
and

\hspace{0.4cm}$a_{\mathbf{m}, \mathbf{s}, \mathbf{t}}$
$$=\sum\limits_{
1\leqslant i, j, l\leqslant n,
m_is_jt_l>0}
\sum\limits_{k_1=1}^{m_i}\sum\limits_{k_2=1}^{s_j}\sum\limits_{k_3=1}^{t_l}
(X^{\mathbf m}_{i,k_1}\whot (X^{\mathbf s}_{j,k_2}\whot  X^{\mathbf t}_{l,k_3}))a_{i,j,l}
((Y^{\mathbf m}_{i,k_1}\whot  Y^{\mathbf s}_{j,k_2})\whot  Y^{\mathbf t}_{l,k_3})$$
when $\mathbf{m}\whot\mathbf{s}\whot\mathbf{t}\neq\mathbf{0}$.
Then we have the following proposition.

\begin{proposition}\label{ass}
$(1)$ $a_{\mathbf{e}_i, \mathbf{e}_j, \mathbf{e}_l}=a_{i,j,l}$ for all $i,j,l\in I$;\\
$(2)$ $a_{\mathbf{m}, \mathbf{s}, \mathbf{t}}$ is invertible in $M_{\mathbf{m}\whot\mathbf{s}\whot\mathbf{t}}(A)$
whenever $\mathbf{m}\whot\mathbf{s}\whot\mathbf{t}\neq\mathbf{0}$;\\
$(3)$ $(X_1\whot (X_2\whot X_3))a_{\mathbf{m}_1, \mathbf{m}_2, \mathbf{m}_3}
=a_{\mathbf{s}_1, \mathbf{s}_2, \mathbf{s}_3}((X_1\whot X_2)\whot  X_3)$
whenever $\mathbf{m}_i, \mathbf{s}_i\in\mathbb{N}^I$ and\\
\mbox{\hspace{0.5cm}} $X_i\in M_{\mathbf{s}_i\times\mathbf{m}_i}(A)={\rm Hom}_{\widehat{\mathcal C}}(\mathbf{m}_i, \mathbf{s}_i)$,
$i=1, 2, 3$;\\
$(4)$ $a_{\mathbf{m}, \mathbf{e}_1, \mathbf{s}}=E_{\mathbf{m}\whot\mathbf{s}}$
for all $\mathbf{m}, \mathbf{s}\in\mathbb{N}^I$;\\
$(5)$ \hspace{1cm}$(E_{\mathbf{m}_1}{\whot} a_{\mathbf{m}_2, \mathbf{m}_3, \mathbf{m}_4})
a_{\mathbf{m}_1, \mathbf{m}_2\whot\mathbf{m}_3, \mathbf{m}_4}
(a_{\mathbf{m}_1, \mathbf{m}_2, \mathbf{m}_3}{\whot} E_{\mathbf{m}_4})$\\
\mbox{\hspace{1.1cm}} $=a_{\mathbf{m}_1, \mathbf{m}_2, \mathbf{m}_3\whot\mathbf{m}_4}a_{\mathbf{m}_1\whot\mathbf{m}_2, \mathbf{m}_3, \mathbf{m}_4}$\\
\mbox{\hspace{0.4cm}}  for all objects $\mathbf{m}_1, \mathbf{m}_2, \mathbf{m}_3, \mathbf{m}_4$ of $\widehat{\mathcal C}$.\\
\end{proposition}

\begin{proof}
(1) It follows from the definition of $a_{\mathbf{e}_i, \mathbf{e}_j, \mathbf{e}_l}$.

(2) Let $\mathbf{m}=(m_i)_{i\in I}, \mathbf{s}=(s_i)_{i\in I}, \mathbf{t}=(t_i)_{i\in I}\in\mathbb{N}^I$.
Assume $\mathbf{m}\whot\mathbf{s}\whot\mathbf{t}\neq\mathbf{0}$. Put\\
\hspace{0.4cm}$b_{\mathbf{m}, \mathbf{s}, \mathbf{t}}$
$$=\sum\limits_{
1\leqslant i, j, l\leqslant n,
m_is_jt_l>0}
\sum\limits_{k_1=1}^{m_i}\sum\limits_{k_2=1}^{s_j}\sum\limits_{k_3=1}^{t_l}
((X^{\mathbf m}_{i,k_1}\whot  X^{\mathbf s}_{j,k_2})\whot X^{\mathbf t}_{l,k_3})a^{-1}_{i,j,l}
(Y^{\mathbf m}_{i,k_1}\whot (Y^{\mathbf s}_{j,k_2}\whot  Y^{\mathbf t}_{l,k_3})),$$
where $a^{-1}_{i,j,l}$ is the inverse matrix of $a_{i,j,l}$ in
$M_{\mathbf{e}_i\whot\mathbf{e}_j\whot\mathbf{e}_l}(A)$ for all $i, j, l\in I$.
Obviousely,  $b_{\mathbf{m}, \mathbf{s}, \mathbf{t}}\in M_{\mathbf{m}\whot\mathbf{s}\whot\mathbf{t}}(A)$.
By Corollary \ref{identitymap} and Lemma \ref{tensorcomp}, a straightforward computation shows that
$a_{\mathbf{m}, \mathbf{s}, \mathbf{t}}b_{\mathbf{m}, \mathbf{s}, \mathbf{t}}
=E_{\mathbf{m}\whot\mathbf{s}\whot\mathbf{t}}$
and $b_{\mathbf{m}, \mathbf{s}, \mathbf{t}}a_{\mathbf{m}, \mathbf{s}, \mathbf{t}}
=E_{\mathbf{m}\whot\mathbf{s}\whot\mathbf{t}}$. Hence $a_{\mathbf{m}, \mathbf{s}, \mathbf{t}}$
is invertible in $M_{\mathbf{m}\whot\mathbf{s}\whot\mathbf{t}}(A)$.

(3) Let $\mathbf{m}=(m_i)_{i\in I}, \mathbf{s}=(s_i)_{i\in I}, \mathbf{t}=(t_i)_{i\in I}\in\mathbb{N}^I$
with $\mathbf{m}\whot\mathbf{s}\whot\mathbf{t}\neq\mathbf{0}$.
Then for any $i',j',l'\in I$ with $m_{i'}s_{j'}t_{l'}>0$, and any $1\leqslant k'_1\leqslant m_{i'}$,
$1\leqslant k'_2\leqslant s_{j'}$ and $1\leqslant k'_3\leqslant t_{l'}$,
by Corollary \ref{identitymap} and Lemma \ref{tensorcomp}, we have
$$\begin{array}{rl}
&a_{\mathbf{m}, \mathbf{s}, \mathbf{t}}
((X^{\mathbf m}_{i',k'_1}\whot  X^{\mathbf s}_{j',k'_2})\whot  X^{\mathbf t}_{l',k'_3})\\
=&\sum\limits_{1\leqslant i, j, l\leqslant n,m_is_jt_l>0}
\sum\limits_{k_1=1}^{m_i}\sum\limits_{k_2=1}^{s_j}\sum\limits_{k_3=1}^{t_l}
(X^{\mathbf m}_{i,k_1}\whot (X^{\mathbf s}_{j,k_2}\whot  X^{\mathbf t}_{l,k_3}))\\
&\cdot a_{i,j,l}((Y^{\mathbf m}_{i,k_1}\whot  Y^{\mathbf s}_{j,k_2})\whot  Y^{\mathbf t}_{l,k_3})
((X^{\mathbf m}_{i',k'_1}\whot  X^{\mathbf s}_{j',k'_2})\whot X^{\mathbf t}_{l',k'_3})\\
=&\sum\limits_{1\leqslant i, j, l\leqslant n,m_is_jt_l>0}
\sum\limits_{k_1=1}^{m_i}\sum\limits_{k_2=1}^{s_j}\sum\limits_{k_3=1}^{t_l}
(X^{\mathbf m}_{i,k_1}\whot (X^{\mathbf s}_{j,k_2}\whot  X^{\mathbf t}_{l,k_3}))\\
&\cdot a_{i,j,l}((Y^{\mathbf m}_{i,k_1}X^{\mathbf m}_{i',k'_1}\whot
Y^{\mathbf s}_{j,k_2}X^{\mathbf s}_{j',k'_2})\whot  Y^{\mathbf t}_{l,k_3}X^{\mathbf t}_{l',k'_3})\\
=&(X^{\mathbf m}_{i',k'_1}\whot (X^{\mathbf s}_{j',k'_2}\whot  X^{\mathbf t}_{l',k'_3}))
a_{i',j',l'}((E_{\mathbf{e}_{i'}}\whot  E_{\mathbf{e}_{j'}})\whot  E_{\mathbf{e}_{l'}})\\
=&(X^{\mathbf m}_{i',k'_1}\whot (X^{\mathbf s}_{j',k'_2}\whot  X^{\mathbf t}_{l',k'_3}))a_{i',j',l'}.
\hspace{3cm}(3.1)\\
\end{array}$$

Now let $\mathbf{m}_i=(m_{ij})_{j\in I}, \mathbf{s}_i=(s_{ij})_{j\in I}\in\mathbb{N}^I$ and
$X_i\in M_{\mathbf{s}_i\times\mathbf{m}_i}(A)$, $i=1, 2, 3$.
If $\mathbf{m}_1\whot\mathbf{m}_2\whot\mathbf{m}_3=\mathbf{0}$ or
$\mathbf{s}_1\whot\mathbf{s}_2\whot\mathbf{s}_3=\mathbf{0}$, then obviously
$(X_1\whot(X_2\whot X_3))a_{\mathbf{m}_1, \mathbf{m}_2, \mathbf{m}_3}
=a_{\mathbf{s}_1, \mathbf{s}_2, \mathbf{s}_3}((X_1\whot X_2)\whot  X_3)$.
Now assume that $\mathbf{m}_1\whot\mathbf{m}_2\whot\mathbf{m}_3\neq\mathbf{0}$ and
$\mathbf{s}_1\whot\mathbf{s}_2\whot\mathbf{s}_3\neq\mathbf{0}$.
Again by Corollary \ref{identitymap} and Lemma \ref{tensorcomp}, we have:
$$\begin{array}{rl}
&X_1\whot(X_2\whot X_3)\\
=&(E_{\mathbf{s}_1}\whot (E_{\mathbf{s}_2}\whot  E_{\mathbf{s}_3}))
(X_1\whot(X_2\whot X_3))(E_{\mathbf{m}_1}\whot (E_{\mathbf{m}_2}\whot  E_{\mathbf{m}_3}))\\
=&\sum\limits_{i,i',j,j',l,l'}\ \sum\limits_{k_1,k_2,k_3, k'_1, k'_2,k'_3}
(X^{\mathbf{s}_1}_{i,k_1}Y^{\mathbf{s}_1}_{i,k_1}\whot
(X^{\mathbf{s}_2}_{j,k_2}Y^{\mathbf{s}_2}_{j,k_2}\whot  X^{\mathbf{s}_3}_{l,k_3}Y^{\mathbf{s}_3}_{l,k_3}))\\
&\cdot(X_1\whot(X_2\whot X_3))(X^{\mathbf{m}_1}_{i',k'_1}Y^{\mathbf{m}_1}_{i',k'_1}\whot
(X^{\mathbf{m}_2}_{j',k'_2}Y^{\mathbf{m}_2}_{j',k'_2}\whot  X^{\mathbf{m}_3}_{l',k'_3}Y^{\mathbf{m}_3}_{l',k'_3}))\\
=&\sum\limits_{i,i',j,j',l,l'}\ \sum\limits_{k_1,k_2,k_3, k'_1, k'_2,k'_3}
(X^{\mathbf{s}_1}_{i,k_1}\whot (X^{\mathbf{s}_2}_{j,k_2}\whot X^{\mathbf{s}_3}_{l,k_3}))\\
&\cdot(Y^{\mathbf{s}_1}_{i,k_1}X_1X^{\mathbf{m}_1}_{i',k'_1}\whot
(Y^{\mathbf{s}_2}_{j,k_2}X_2X^{\mathbf{m}_2}_{j',k'_2}\whot  Y^{\mathbf{s}_3}_{l,k_3}X_3X^{\mathbf{m}_3}_{l',k'_3}))
(Y^{\mathbf{m}_1}_{i',k'_1}\whot
(Y^{\mathbf{m}_2}_{j',k'_2}\whot  Y^{\mathbf{m}_3}_{l',k'_3})).
\end{array}$$
Similarly, we have:
$$\begin{array}{rl}
&(X_1\whot X_2)\whot  X_3\\
=&\sum\limits_{i,i',j,j',l,l'}\ \sum\limits_{k_1,k_2,k_3, k'_1, k'_2,k'_3}
((X^{\mathbf{s}_1}_{i,k_1}\whot  X^{\mathbf{s}_2}_{j,k_2})\whot  X^{\mathbf{s}_3}_{l,k_3})\\
&\cdot((Y^{\mathbf{s}_1}_{i,k_1}X_1X^{\mathbf{m}_1}_{i',k'_1}\whot
Y^{\mathbf{s}_2}_{j,k_2}X_2X^{\mathbf{m}_2}_{j',k'_2})\whot  Y^{\mathbf{s}_3}_{l,k_3}X_3X^{\mathbf{m}_3}_{l',k'_3})
((Y^{\mathbf{m}_1}_{i',k'_1}\whot  Y^{\mathbf{m}_2}_{j',k'_2})\whot  Y^{\mathbf{m}_3}_{l',k'_3}),\\
\end{array}$$
where the sum notation $\sum\limits_{i,i',j,j',l,l'}\ \sum\limits_{k_1,k_2,k_3, k'_1, k'_2,k'_3}$ means
$$\sum\limits_{1\leqslant i,i',j,j',l,l'\leqslant n, s_{1i}s_{2j}s_{3l}m_{1i'}m_{2j'}m_{3l'}>0}\ \
\sum\limits_{k_1=1}^{s_{1i}}\sum\limits_{k_2=1}^{s_{2j}}\sum\limits_{k_3=1}^{s_{3l}}
\sum\limits_{k'_1=1}^{m_{1i'}}\sum\limits_{k'_2=1}^{m_{2j'}}\sum\limits_{k'_3=1}^{m_{3l'}}.$$
Note that $Y^{\mathbf{s}_1}_{i,k_1}X_1X^{\mathbf{m}_1}_{i',k'_1}\in e_iAe_{i'}$,
$Y^{\mathbf{s}_2}_{j,k_2}X_2X^{\mathbf{m}_2}_{j',k'_2}\in e_jAe_{j'}$
and $Y^{\mathbf{s}_3}_{l,k_3}X_3X^{\mathbf{m}_3}_{l',k'_3}\in e_lAe_{l'}$.
Hence by (Ass 2) and Eq.(3.1), we have
$$\begin{array}{rl}
&a_{\mathbf{s}_1, \mathbf{s}_2, \mathbf{s}_3}((X_1\whot X_2)\whot  X_3)\\
=&\sum\limits_{i,i',j,j',l,l'}\ \sum\limits_{k_1,k_2,k_3, k'_1, k'_2,k'_3}a_{\mathbf{s}_1, \mathbf{s}_2, \mathbf{s}_3}
((X^{\mathbf{s}_1}_{i,k_1}\whot  X^{\mathbf{s}_2}_{j,k_2})\whot  X^{\mathbf{s}_3}_{l,k_3})\\
&\cdot((Y^{\mathbf{s}_1}_{i,k_1}X_1X^{\mathbf{m}_1}_{i',k'_1}\whot
Y^{\mathbf{s}_2}_{j,k_2}X_2X^{\mathbf{m}_2}_{j',k'_2})\whot  Y^{\mathbf{s}_3}_{l,k_3}X_3X^{\mathbf{m}_3}_{l',k'_3})
((Y^{\mathbf{m}_1}_{i',k'_1}\whot  Y^{\mathbf{m}_2}_{j',k'_2})\whot  Y^{\mathbf{m}_3}_{l',k'_3})\\
=&\sum\limits_{i,i',j,j',l,l'}\ \sum\limits_{k_1,k_2,k_3, k'_1, k'_2,k'_3}
(X^{\mathbf{s}_1}_{i,k_1}\whot (X^{\mathbf{s}_2}_{j,k_2}\whot  X^{\mathbf{s}_3}_{l,k_3}))a_{i,j,l}\\
&\cdot((Y^{\mathbf{s}_1}_{i,k_1}X_1X^{\mathbf{m}_1}_{i',k'_1}\whot
Y^{\mathbf{s}_2}_{j,k_2}X_2X^{\mathbf{m}_2}_{j',k'_2})\whot  Y^{\mathbf{s}_3}_{l,k_3}X_3X^{\mathbf{m}_3}_{l',k'_3})
((Y^{\mathbf{m}_1}_{i',k'_1}\whot  Y^{\mathbf{m}_2}_{j',k'_2})\whot  Y^{\mathbf{m}_3}_{l',k'_3})\\
=&\sum\limits_{i,i',j,j',l,l'}\ \sum\limits_{k_1,k_2,k_3, k'_1, k'_2,k'_3}
(X^{\mathbf{s}_1}_{i,k_1}\whot (X^{\mathbf{s}_2}_{j,k_2}\whot  X^{\mathbf{s}_3}_{l,k_3}))\\
&\cdot(Y^{\mathbf{s}_1}_{i,k_1}X_1X^{\mathbf{m}_1}_{i',k'_1}\whot  (Y^{\mathbf{s}_2}_{j,k_2}X_2X^{\mathbf{m}_2}_{j',k'_2}\whot  Y^{\mathbf{s}_3}_{l,k_3}X_3X^{\mathbf{m}_3}_{l',k'_3}))
a_{i',j',l'}((Y^{\mathbf{m}_1}_{i',k'_1}\whot  Y^{\mathbf{m}_2}_{j',k'_2})\whot  Y^{\mathbf{m}_3}_{l',k'_3})\\
=&\sum\limits_{i,i',j,j',l,l'}\ \sum\limits_{k_1,k_2,k_3, k'_1, k'_2,k'_3}
(X^{\mathbf{s}_1}_{i,k_1}Y^{\mathbf{s}_1}_{i,k_1}\whot (X^{\mathbf{s}_2}_{j,k_2}Y^{\mathbf{s}_2}_{j,k_2}
\whot  X^{\mathbf{s}_3}_{l,k_3}Y^{\mathbf{s}_3}_{l,k_3}))\\
&\cdot(X_1\whot(X_2\whot X_3))(X^{\mathbf{m}_1}_{i',k'_1}\whot  (X^{\mathbf{m}_2}_{j',k'_2}\whot  X^{\mathbf{m}_3}_{l',k'_3}))
a_{i',j',l'}((Y^{\mathbf{m}_1}_{i',k'_1}\whot  Y^{\mathbf{m}_2}_{j',k'_2})\whot  Y^{\mathbf{m}_3}_{l',k'_3})\\
=&(E_{\mathbf{s}_1}\whot (E_{\mathbf{s}_2}\whot  E_{\mathbf{s}_3}))(X_1\whot(X_2\whot X_3))
a_{\mathbf{m}_1, \mathbf{m}_2, \mathbf{m}_3}\\
=&(X_1\whot(X_2\whot X_3))a_{\mathbf{m}_1, \mathbf{m}_2, \mathbf{m}_3}.\\
\end{array}$$

(4) Let $\mathbf{m}=(m_i)_{i\in I}, \mathbf{s}=(s_i)_{i\in I}\in N^I$.
Since $X^{\mathbf{e}_1}_{1,1}=Y^{\mathbf{e}_1}_{1,1}=e_1=E_{\mathbf{e}_1}$,
by Corollary \ref{identitymap} and (Ass 3), we have
$$\begin{array}{rl}
&a_{\mathbf{m}, \mathbf{e}_1, \mathbf{s}}\\
=&\sum\limits_{1\leqslant i, j\leqslant n, m_is_j>0}\sum\limits_{k_1=1}^{m_i}\sum\limits_{k_2=1}^{s_j}
(X^{\mathbf m}_{i,k_1}\whot (E_{\mathbf{e}_1}\whot  X^{\mathbf s}_{j,k_2}))a_{i,1,j}
((Y^{\mathbf m}_{i,k_1}\whot  E_{\mathbf{e}_1})\whot  Y^{\mathbf s}_{j,k_2})\\
=&\sum\limits_{1\leqslant i, j\leqslant n, m_is_j>0}\sum\limits_{k_1=1}^{m_i}\sum\limits_{k_2=1}^{s_j}
(X^{\mathbf m}_{i,k_1}\whot  X^{\mathbf s}_{j,k_2})E_{\mathbf{e}_i\whot\mathbf{e}_j}
(Y^{\mathbf m}_{i,k_1}\whot  Y^{\mathbf s}_{j,k_2})\\
=&\sum\limits_{1\leqslant i, j\leqslant n, m_is_j>0}\sum\limits_{k_1=1}^{m_i}\sum\limits_{k_2=1}^{s_j}
X^{\mathbf m}_{i,k_1}Y^{\mathbf m}_{i,k_1}\whot  X^{\mathbf s}_{j,k_2}Y^{\mathbf s}_{j,k_2}\\
=&E_{\mathbf m}\whot  E_{\mathbf s}=E_{\mathbf{m}\whot\mathbf{s}}.\\
\end{array}$$

(5) By (Ass 4) and $\mathbf{e}_i\whot\mathbf{e}_j=\mathbf{c}_{ij}$, one can check that
$$\hspace{1cm}(E_{\mathbf{e}_i}\whot  a_{\mathbf{e}_j, \mathbf{e}_l, \mathbf{e}_t})
a_{\mathbf{e}_i, \mathbf{e}_j\whot\mathbf{e}_l, \mathbf{e}_t}
(a_{\mathbf{e}_i, \mathbf{e}_j, \mathbf{e}_l}\whot  E_{\mathbf{e}_t})
=a_{\mathbf{e}_i, \mathbf{e}_j, \mathbf{e}_l\whot\mathbf{e}_t}
a_{\mathbf{e}_i\whot\mathbf{e}_j, \mathbf{e}_l, \mathbf{e}_t}\hspace{1cm} (3.2)$$
for all $i,j,l,t\in I$. Now let $\mathbf{m}_j=(m_{ji})_{i\in I}\in\mathbb{N}^I$, $j=1, 2, 3, 4$. We have
$\mathbf{m}_2\whot\mathbf{m}_3=\sum_{i,j=1}^nm_{2i}m_{3j}\mathbf{c}_{ij}$.
It follows from Part (3) above that we have
$$\begin{array}{rl}
&(E_{\mathbf{m}_1}\whot  a_{\mathbf{m}_2, \mathbf{m}_3, \mathbf{m}_4})
a_{\mathbf{m}_1, \mathbf{m}_2\whot\mathbf{m}_3, \mathbf{m}_4}
(a_{\mathbf{m}_1, \mathbf{m}_2, \mathbf{m}_3}\whot  E_{\mathbf{m}_4})\\
=&(E_{\mathbf{m}_1}\whot  a_{\mathbf{m}_2, \mathbf{m}_3, \mathbf{m}_4})
a_{\mathbf{m}_1, \mathbf{m}_2\whot\mathbf{m}_3, \mathbf{m}_4}\\
&(\sum\limits_{i,j,l}\ \sum\limits_{k_1, k_2, k_3}
(X^{\mathbf{m}_1}_{i,k_1}\whot (X^{\mathbf{m}_2}_{j, k_2}\whot  X^{\mathbf{m}_3}_{l,k_3}))
a_{i,j,l}((Y^{\mathbf{m}_1}_{i,k_1}\whot  Y^{\mathbf{m}_2}_{j, k_2})\whot  Y^{\mathbf{m}_3}_{l,k_3})
\whot  E_{\mathbf{m}_4})\\
=&\sum\limits_{i,j,l,t}\ \sum\limits_{k_1, k_2, k_3, k_4}
(E_{\mathbf{m}_1}\whot  a_{\mathbf{m}_2, \mathbf{m}_3, \mathbf{m}_4})\\
&a_{\mathbf{m}_1, \mathbf{m}_2\whot\mathbf{m}_3, \mathbf{m}_4}
((X^{\mathbf{m}_1}_{i,k_1}\whot (X^{\mathbf{m}_2}_{j, k_2}\whot  X^{\mathbf{m}_3}_{l,k_3}))
\whot  X^{\mathbf{m}_4}_{t,k_4})
\\
&(a_{i,j,l}\whot  e_{t})
(((Y^{\mathbf{m}_1}_{i,k_1}\whot  Y^{\mathbf{m}_2}_{j, k_2})\whot  Y^{\mathbf{m}_3}_{l,k_3})
\whot  Y^{\mathbf{m}_4}_{t,k_4})\\
\end{array}$$

$$\begin{array}{rl}
=&\sum\limits_{i,j,l,t}\ \sum\limits_{k_1, k_2, k_3, k_4}
(E_{\mathbf{m}_1}\whot  a_{\mathbf{m}_2, \mathbf{m}_3, \mathbf{m}_4})\\
&(X^{\mathbf{m}_1}_{i,k_1}\whot ((X^{\mathbf{m}_2}_{j, k_2}\whot  X^{\mathbf{m}_3}_{l,k_3})
\whot  X^{\mathbf{m}_4}_{t,k_4}))a_{\mathbf{e}_{i}, \mathbf{e}_{j}\whot\mathbf{e}_{l}, \mathbf{e}_{t}}\\
&(a_{i,j,l}\whot  e_{t})
(((Y^{\mathbf{m}_1}_{i,k_1}\whot  Y^{\mathbf{m}_2}_{j, k_2})\whot  Y^{\mathbf{m}_3}_{l,k_3})
\whot  Y^{\mathbf{m}_4}_{t,k_4})\\
=&(\sum\limits_{i',j',l'}\ \sum\limits_{k'_1, k'_2, k'_3}
E_{\mathbf{m}_1}\whot (X^{\mathbf{m}_2}_{i',k'_1}\whot (X^{\mathbf{m}_3}_{j', k'_2}\whot  X^{\mathbf{m}_4}_{l',k'_3}))
a_{i',j',l'}((Y^{\mathbf{m}_2}_{i',k'_1}\whot  Y^{\mathbf{m}_3}_{j', k'_2})\whot  Y^{\mathbf{m}_4}_{l',k'_3}))\\
&(\sum\limits_{i,j,l,t}\ \sum\limits_{k_1, k_2, k_3, k_4}
(X^{\mathbf{m}_1}_{i,k_1}\whot ((X^{\mathbf{m}_2}_{j, k_2}\whot  X^{\mathbf{m}_3}_{l,k_3})
\whot  X^{\mathbf{m}_4}_{t,k_4}))a_{\mathbf{e}_{i}, \mathbf{e}_{j}\whot\mathbf{e}_{l}, \mathbf{e}_{t}}\\
&(a_{i,j,l}\whot  e_{t})
(((Y^{\mathbf{m}_1}_{i,k_1}\whot  Y^{\mathbf{m}_2}_{j, k_2})\whot  Y^{\mathbf{m}_3}_{l,k_3})
\whot  Y^{\mathbf{m}_4}_{t,k_4}))\\
=&\sum\limits_{i',j',l',t',i,j,l,t}\ \sum\limits_{k'_1, k'_2, k'_3,k'_4,k_1, k_2, k_3, k_4}
(X^{\mathbf{m}_1}_{t',k'_4}\whot (X^{\mathbf{m}_2}_{i',k'_1}\whot (X^{\mathbf{m}_3}_{j', k'_2}\whot  X^{\mathbf{m}_4}_{l',k'_3})))
(e_{t'}\whot  a_{i',j',l'})\\
&(Y^{\mathbf{m}_1}_{t',k'_4}\whot ((Y^{\mathbf{m}_2}_{i',k'_1}\whot  Y^{\mathbf{m}_3}_{j', k'_2})\whot  Y^{\mathbf{m}_4}_{l',k'_3}))
(X^{\mathbf{m}_1}_{i,k_1}\whot ((X^{\mathbf{m}_2}_{j, k_2}\whot  X^{\mathbf{m}_3}_{l,k_3})
\whot  X^{\mathbf{m}_4}_{t,k_4}))\\
&a_{\mathbf{e}_{i}, \mathbf{e}_{j}\whot\mathbf{e}_{l}, \mathbf{e}_{t}}
(a_{i,j,l}\whot  e_{t})
(((Y^{\mathbf{m}_1}_{i,k_1}\whot  Y^{\mathbf{m}_2}_{j, k_2})\whot  Y^{\mathbf{m}_3}_{l,k_3})
\whot  Y^{\mathbf{m}_4}_{t,k_4})\\
=&\sum\limits_{i,j,l,t,i',j',l',t'}\ \sum\limits_{k_1, k_2, k_3,k_4,k'_1, k'_2, k'_3,k'_4}
(X^{\mathbf{m}_1}_{t',k'_4}\whot (X^{\mathbf{m}_2}_{i',k'_1}\whot (X^{\mathbf{m}_3}_{j', k'_2}\whot  X^{\mathbf{m}_4}_{l',k'_3})))
(e_{t'}\whot  a_{i',j',l'})\\
&(Y^{\mathbf{m}_1}_{t',k'_4}X^{\mathbf{m}_1}_{i,k_1}\whot
((Y^{\mathbf{m}_2}_{i',k'_1}X^{\mathbf{m}_2}_{j, k_2}\whot  Y^{\mathbf{m}_3}_{j', k'_2}X^{\mathbf{m}_3}_{l,k_3})
\whot  Y^{\mathbf{m}_4}_{l',k'_3}X^{\mathbf{m}_4}_{t,k_4}))\\
&a_{\mathbf{e}_{i}, \mathbf{e}_{j}\whot\mathbf{e}_{l}, \mathbf{e}_{t}}
(a_{i,j,l}\whot  e_{t})
(((Y^{\mathbf{m}_1}_{i,k_1}\whot  Y^{\mathbf{m}_2}_{j, k_2})\whot  Y^{\mathbf{m}_3}_{l,k_3})
\whot  Y^{\mathbf{m}_4}_{t,k_4})\\
=&\sum\limits_{i,j,l,t}\ \sum\limits_{k_1, k_2, k_3,k_4}
(X^{\mathbf{m}_1}_{i,k_1}\whot (X^{\mathbf{m}_2}_{j,k_2}\whot (X^{\mathbf{m}_3}_{l, k_3}\whot  X^{\mathbf{m}_4}_{t,k_4})))\\
&(E_{\mathbf{e}_i}\whot  a_{\mathbf{e}_j,\mathbf{e}_l,\mathbf{e}_t})
a_{\mathbf{e}_i, \mathbf{e}_j\whot\mathbf{e}_l, \mathbf{e}_t}
(a_{\mathbf{e}_i,\mathbf{e}_j,\mathbf{e}_l}\whot  E_{\mathbf{e}_t})
(((Y^{\mathbf{m}_1}_{i,k_1}\whot  Y^{\mathbf{m}_2}_{j, k_2})\whot  Y^{\mathbf{m}_3}_{l,k_3})
\whot  Y^{\mathbf{m}_4}_{t,k_4})\\
\end{array}$$
and
$$\begin{array}{rl}
&a_{\mathbf{m}_1, \mathbf{m}_2, \mathbf{m}_3\whot\mathbf{m}_4}a_{\mathbf{m}_1\whot\mathbf{m}_2, \mathbf{m}_3, \mathbf{m}_4}\\
=&\sum\limits_{i,j,l,t}\ \sum\limits_{k_1, k_2, k_3, k_4}a_{\mathbf{m}_1, \mathbf{m}_2, \mathbf{m}_3\whot\mathbf{m}_4}
((X^{\mathbf{m}_1}_{i,k_1}\whot  X^{\mathbf{m}_2}_{j,k_2})\whot (X^{\mathbf{m}_3}_{l,k_3}\whot  X^{\mathbf{m}_4}_{t,k_4}))\\
&((Y^{\mathbf{m}_1}_{i,k_1}\whot  Y^{\mathbf{m}_2}_{j,k_2})\whot (Y^{\mathbf{m}_3}_{l,k_3}\whot  Y^{\mathbf{m}_4}_{t,k_4}))
a_{\mathbf{m}_1\whot\mathbf{m}_2, \mathbf{m}_3, \mathbf{m}_4}\\
=&\sum\limits_{i,j,l,t}\ \sum\limits_{k_1, k_2, k_3, k_4}
(X^{\mathbf{m}_1}_{i,k_1}\whot (X^{\mathbf{m}_2}_{j,k_2}\whot (X^{\mathbf{m}_3}_{l,k_3}\whot  X^{\mathbf{m}_4}_{t,k_4})))
a_{\mathbf{e}_i, \mathbf{e}_j, \mathbf{e}_l\whot\mathbf{e}_t}\\
&a_{\mathbf{e}_i\whot\mathbf{e}_j, \mathbf{e}_l, \mathbf{e}_t}
(((Y^{\mathbf{m}_1}_{i,k_1}\whot  Y^{\mathbf{m}_2}_{j,k_2})\whot  Y^{\mathbf{m}_3}_{l,k_3})\whot  Y^{\mathbf{m}_4}_{t,k_4}).\\
\end{array}$$
By Eq.(3.2), one gets that
$(E_{\mathbf{m}_1}\whot  a_{\mathbf{m}_2, \mathbf{m}_3, \mathbf{m}_4})
a_{\mathbf{m}_1, \mathbf{m}_2\whot\mathbf{m}_3, \mathbf{m}_4}
(a_{\mathbf{m}_1, \mathbf{m}_2, \mathbf{m}_3}\whot  E_{\mathbf{m}_4})
=a_{\mathbf{m}_1, \mathbf{m}_2, \mathbf{m}_3\whot\mathbf{m}_4}
a_{\mathbf{m}_1\whot\mathbf{m}_2, \mathbf{m}_3, \mathbf{m}_4}$.
This completes the proof.
\end{proof}

\begin{remark}\label{Ass4}
(1) A family $\{a_{\mathbf{m}, \mathbf{s}, \mathbf{t}}\in M_{\mathbf{m}\whot\mathbf{s}\whot\mathbf{t}}(A)\}
_{\mathbf{m}, \mathbf{s}, \mathbf{t}\in\mathbb{N}^I}$ of matrices  satisfying the conditions in
Proposition \ref{ass} is uniquely determined by $\{a_{i,j,l}\}_{i,j,l\in I}$.

(2) Using the notations in the proof of Proposition \ref{ass}, the equation of (Ass 4) can be replaced
with Eq.(3.2).
\end{remark}

\begin{corollary}
Let $\mathbf{m}, \mathbf{s}\in\mathbb{N}^I$. Then
$a_{\mathbf{e}_1, \mathbf{m}, \mathbf{s}}=a_{\mathbf{m}, \mathbf{s}, \mathbf{e}_1}=E_{\mathbf{m}\whot\mathbf{s}}$.
In particular, $a_{1,i,j}=a_{i,j,1}=E_{\mathbf{e}_i\whot\mathbf{e}_j}=E_{\mathbf{c}_{ij}}$ for all
$i,j\in I$.
\end{corollary}

\begin{proof}
By Proposition \ref{ass}(5), we have
$$(E_{\mathbf{e}_1}\whot  a_{\mathbf{e}_1, \mathbf{m}, \mathbf{s}})
a_{\mathbf{e}_1, \mathbf{e}_1\whot\mathbf{m}, \mathbf{s}}
(a_{\mathbf{e}_1, \mathbf{e}_1, \mathbf{m}}\whot  E_{\mathbf{s}})
=a_{\mathbf{e}_1, \mathbf{e}_1, \mathbf{m}\whot\mathbf{s}}
a_{\mathbf{e}_1\whot\mathbf{e}_1, \mathbf{m}, \mathbf{s}}.$$
Following  Corollary \ref{identitymap} and Proposition \ref{ass}(4), one obtains:
$a_{\mathbf{e}_1, \mathbf{m}, \mathbf{s}}
a_{\mathbf{e}_1, \mathbf{m}, \mathbf{s}}
=a_{\mathbf{e}_1, \mathbf{m}, \mathbf{s}}$, and so
$a_{\mathbf{e}_1, \mathbf{m}, \mathbf{s}}=E_{\mathbf{m}\whot\mathbf{s}}$.
Similarly, one can show that $a_{\mathbf{m}, \mathbf{s}, \mathbf{e}_1}=E_{\mathbf{m}\whot\mathbf{s}}$.
\end{proof}

\begin{theorem}\label{monoidalcat}
With the tensor products of objects and morphisms defined above,
$\widehat{\mathcal{C}}$ is a tensor category over $\mathbb F$.
\end{theorem}

\begin{proof}
We know already  that $\mathbf{m}\whot\mathbf{s}\in{\rm Ob}(\widehat{\mathcal C})$ for any
$\mathbf{m}, \mathbf{s}\in{\rm Ob}(\widehat{\mathcal C})$, and that
$X\whot Y$ is a morphism from $\mathbf{m}_1\whot \mathbf{m}_2$ to
$\mathbf{s}_1\whot\mathbf{s}_2$ in $\widehat{\mathcal C}$ for any morphisms
$X\in{\rm Hom}_{\widehat{\mathcal C}}(\mathbf{m}_1, \mathbf{s}_1)$
and $Y\in{\rm Hom}_{\widehat{\mathcal C}}(\mathbf{m}_2, \mathbf{s}_2)$.
It follows from Corollary \ref{identitymap} and Lemma \ref{tensorcomp} that
$\whot : \widehat{\mathcal{C}}\times\widehat{\mathcal{C}}\ra\widehat{\mathcal{C}}$,
$(\mathbf{m}, \mathbf{s})\mapsto \mathbf{m}\whot\mathbf{s}$,
$(X, Y)\mapsto X\whot Y$, is an $\mathbb F$-bilinear bifunctor.

For any $\mathbf{m}=(m_i)_{1\leqslant i\leqslant n}\in{\rm Ob}(\widehat{\mathcal C})$, we have
$\mathbf{e}_1\whot\mathbf{m}=\mathbf{m}=\mathbf{m}\whot\mathbf{e}_1$.
Let $l_{\mathbf{m}}: \mathbf{e}_1\whot\mathbf{m}\ra \mathbf{m}$ and
$r_{\mathbf{m}}: \mathbf{m}\whot\mathbf{e}_1\ra \mathbf{m}$ be the identity morphisms,
i.e., $l_{\mathbf{m}}=r_{\mathbf{m}}={\rm id}_{\mathbf m}=E_{\mathbf m}$.
By Corollary \ref{identitymap}, we have ${\rm id}_{\mathbf{e}_1}\whot  X=X\whot{\rm id}_{\mathbf{e}_1}=X$
for any morphism $X\in{\rm Hom}_{\widehat{\mathcal C}}(\mathbf{m}, \mathbf{s})$.
It follows that the following two diagrams
$$\begin{array}{rclcrcl}
\mathbf{e}_1\whot\mathbf{m}&\xrightarrow{l_{\mathbf m}}&\mathbf{m}&&
\mathbf{m}\whot\mathbf{e}_1&\xrightarrow{r_{\mathbf m}}&\mathbf{m}\\
{\rm id}_{\mathbf{e}_1}\whot  X\downarrow\mbox{\hspace{0.3cm}}&&\hspace{0.1cm}\downarrow X& \mbox{ and }&
X\whot{\rm id}_{\mathbf{e}_1}\downarrow\mbox{\hspace{0.3cm}}&&\hspace{0.1cm}\downarrow X\\
\mathbf{e}_1\whot\mathbf{s}&\xrightarrow{l_{\mathbf s}}&\mathbf{s}&&
\mathbf{s}\whot\mathbf{e}_1&\xrightarrow{r_{\mathbf s}}&\mathbf{s}\\
\end{array}$$
commute for any morphism $X: \mathbf{m}\ra\mathbf{s}$ of $\widehat{\mathcal C}$.

Now let $a_{\mathbf{m}, \mathbf{s}, \mathbf{t}}$ be  as before
for any objects $\mathbf{m}, \mathbf{s}, \mathbf{t}$ of $\widehat{\mathcal C}$.
If $\mathbf{m}\whot\mathbf{s}\whot\mathbf{t}=\mathbf{0}$, then
$a_{\mathbf{m}, \mathbf{s}, \mathbf{t}}=E_{\mathbf 0}={\rm id}_{\mathbf 0}$
is an automorphism of $\mathbf{m}\whot\mathbf{s}\whot\mathbf{t}$ in $\widehat{\mathcal C}$.
If $\mathbf{m}\whot\mathbf{s}\whot\mathbf{t}\neq\mathbf{0}$, then it follows from Proposition \ref{ass}(2)
that $a_{\mathbf{m}, \mathbf{s}, \mathbf{t}}$ is an automorphism of
$\mathbf{m}\whot\mathbf{s}\whot\mathbf{t}$ in $\widehat{\mathcal C}$. Therefore, $\{a_{\mathbf{m}, \mathbf{s}, \mathbf{t}}\}$
forms a family of isomorphisms of $\widehat{\mathcal C}$.
Now let $X_1: \mathbf{m}_1\ra\mathbf{s}_1$, $X_2: \mathbf{m}_2\ra\mathbf{s}_2$ and
$X_3: \mathbf{m}_3\ra\mathbf{s}_3$ be morphisms of $\widehat{\mathcal C}$.
By Proposition \ref{ass}(3), the following diagram commutes:
$$\begin{array}{rcl}
\mathbf{m}_1\whot\mathbf{m}_2\whot\mathbf{m}_3&\xrightarrow{a_{\mathbf{m}_1, \mathbf{m}_2, \mathbf{m}_3}}&
\mathbf{m}_1\whot\mathbf{m}_2\whot\mathbf{m}_3\\
(X_1\whot X_2)\whot  X_3\downarrow\mbox{\hspace{1cm}}&&\hspace{1cm}\downarrow X_1\whot(X_2\whot X_3)\\
\mathbf{s}_1\whot\mathbf{s}_2\whot\mathbf{s}_3\mbox{\hspace{0.3cm}}&\xrightarrow{a_{\mathbf{s}_1, \mathbf{s}_2, \mathbf{s}_3}}&
\hspace{0.3cm}\mathbf{s}_1\whot\mathbf{s}_2\whot\mathbf{s}_3.\\
\end{array}$$
This shows that $\{a_{\mathbf{m}, \mathbf{s}, \mathbf{t}}\}$
is a family of natural isomorphisms of $\widehat{\mathcal C}$ indexed by all triples
$(\mathbf{m}, \mathbf{s}, \mathbf{t})$ of objects of $\widehat{\mathcal C}$.

Finally, by Proposition \ref{ass}(4) and (5), one can see that the Triangle Axioms for $a, l, r$
and the Pentagon Axiom for $a$ are satisfied.
Thus, $(\mathcal{C}, \whot , \mathbf{e}_1, a, l, r)$ is a tensor category,
where $\mathbf{e}_1$ is the unit object.
\end{proof}

\begin{remark}\label{GreenRing}
The Green ring $r(\widehat{\mathcal C})$ of $\widehat{\mathcal C}$ is isomorphic to $R$.
More precisely, the additive group homomorphism $r(\widehat{\mathcal C})\ra R$ given by
$[\mathbf{e}_i]\mapsto r_i$, $1\leqslant i\leqslant n$, induces a ring isomorphism.
Let $\mathbf{e}=(1, 1, \cdots, 1)$. Then $\mathbf{e}=\mathbf{e}_1\oplus\mathbf{e}_2\oplus
\cdots\oplus\mathbf{e}_n$ and ${\rm End}_{\widehat{\mathcal C}}(\mathbf{e})\cong A$ as $\mathbb{F}$-algebras.
\end{remark}

\begin{remark}
In  general, $\mathbf{e}_1$ is not necessarily a simple object of $\widehat{\mathcal C}$.
However, the following two statements are equivalent:\\
{\rm (1)} $\mathbf{e}_1$ is a simple object of $\widehat{\mathcal C}$;\\
{\rm (2)}  If $X\in M_{\mathbf{e}_1\times\mathbf{m}}(A)$ is a column-independent matrix,
then either $\mathbf{m}=\mathbf{0}$ and $X=0$, or $\mathbf{m}=\mathbf{e}_1$
and $X=\a E_{\mathbf{e}_1}=\a e_1$ for some $0\neq \a\in\mathbb F$.
\end{remark}

Let $\{a'_{i,j,l}\in M_{\mathbf{e}_i\whot\mathbf{e}_j\whot\mathbf{e}_l}(A)\}$
be another family of matrices indexed by all triples $(i,j,l)$
of elements of $I$ such that the conditions (Ass)(1)-(4) are satisfied.
Let $\{a'_{\mathbf{m}, \mathbf{s}, \mathbf{t}}\in M_{\mathbf{m}\whot\mathbf{s}\whot\mathbf{t}}(A)|
\mathbf{m}, \mathbf{s}, \mathbf{t}\in \mathbb{N}^I\}$
be the family of matrices determined by $\{a'_{i,j,l}\}_{i,j,l\in I}$
as before. Then by Theorem \ref{monoidalcat}, $(\widehat{\mathcal C}, \whot , \mathbf{e}_1, a', l, r)$ is a tensor category as well. We work out when the two tensor categories are tensor equivalent.

\begin{definition}\label{equivalence}
Two families $\{a_{i,j,l}\}_{i,j,l\in I}$ and $\{a'_{i,j,l}\}_{i,j,l\in I}$ of matrices are called equivalent if there exists a family of
invertible matrices $\eta(i, j)$ in $M_{\mathbf{c}_{ij}}(A)$
indexed by all couples $(i,j)$ of elements of $I$
such that $(x\whot y)\eta(i, j)=\eta(i', j')(x\whot y)$
and
$$\begin{array}{rl}
&\sum\limits_{t=1}^n\sum\limits_{1\leqslant k\leqslant c_{ijt}}
a_{i,j,l}(X^{\mathbf{c}_{ij}}_{t,k}\whot  e_l)\eta(t,l)
(Y^{\mathbf{c}_{ij}}_{t,k}\eta(i,j)\whot  e_l)\\
=&\sum\limits_{t=1}^n\sum\limits_{1\leqslant k\leqslant c_{jlt}}
(e_i\whot X^{\mathbf{c}_{jl}}_{t,k})\eta(i,t)(e_i\whot Y^{\mathbf{c}_{jl}}_{t,k}\eta(j,l))
a'_{i,j,l}\\
\end{array}$$
for all $i,j,l,i',j'\in I$,
$x\in e_{i'}Ae_i$ and $y\in e_{j'}Ae_j$.
\end{definition}

\begin{proposition}\label{equi0}
With the above notations, the tensor categories $(\widehat{\mathcal{C}}, \whot , \mathbf{e}_1, a, l, r)$ and $(\widehat{\mathcal{C}}, \whot , \mathbf{e}_1, a', l, r)$ are tensor equivalent  if $\{a_{i,j,l}\}_{i,j,l\in I}$ and $\{a'_{i,j,l}\}_{i,j,l\in I}$ are equivalent.
\end{proposition}

\begin{proof}
For any $\mathbf{m}, \mathbf{s}\in \mathbb{N}^I$, define
$\eta(\mathbf{m}, \mathbf{s})\in M_{\mathbf{m}\whot\mathbf{s}}(A)$ by
$$\eta(\mathbf{m}, \mathbf{s})=\sum\limits_{i,j=1}^n
\sum\limits_{1\leqslant k\leqslant m_i}\sum\limits_{1\leqslant k'\leqslant s_i}
(X^{\mathbf m}_{i,k}\whot  X^{\mathbf s}_{j,k'})\eta(i,j)(Y^{\mathbf m}_{i,k}\whot  Y^{\mathbf s}_{j,k'}).$$
Then an argument similar to the proof of Proposition \ref{ass} shows that
$\{\eta(\mathbf{m}, \mathbf{s})\in M_{\mathbf{m}\whot\mathbf{s}}(A)\}$
is a family of invertible matrices
indexed by all couples $(\mathbf{m}, \mathbf{s})$ of elements of $\mathbb{N}^I$
such that
$(X\whot Y)\eta(\mathbf{m}_1, \mathbf{m}_2)=\eta(\mathbf{s}_1, \mathbf{s}_2)(X\whot Y)$
and
$$\begin{array}{rl}
&a_{\mathbf{m}_1, \mathbf{m}_2, \mathbf{m}_3}\eta(\mathbf{m}_1\whot\mathbf{m}_2, \mathbf{m}_3)
(\eta(\mathbf{m}_1, \mathbf{m}_2)\whot E_{\mathbf{m}_3})\\
=&\eta(\mathbf{m}_1, \mathbf{m}_2\whot\mathbf{m}_3)(E_{\mathbf{m}_1}\whot \eta(\mathbf{m}_2, \mathbf{m}_3))
a'_{\mathbf{m}_1, \mathbf{m}_2, \mathbf{m}_3}\\
\end{array}$$
for all $\mathbf{m}_1, \mathbf{m}_2, \mathbf{m}_3, \mathbf{s}_1, \mathbf{s}_2\in\mathbb{N}^I$,
$X\in M_{\mathbf{s}_1\times\mathbf{m}_1}(A)$ and $Y\in M_{\mathbf{s}_2\times\mathbf{m}_2}(A)$.
Thus, the proposition follows from a straightforward verification.
\end{proof}

\begin{remark}\label{reequivalence}
Using the notations in the proof of Proposition \ref{equi0}, one may replace
the last equation of Definition \ref{equivalence}  with the equation:
$$a_{\mathbf{e}_i, \mathbf{e}_j, \mathbf{e}_l}\eta(\mathbf{e}_i\whot\mathbf{e}_j, \mathbf{e}_l)
(\eta(\mathbf{e}_i, \mathbf{e}_j)\whot E_{\mathbf{e}_l})
=\eta(\mathbf{e}_i, \mathbf{e}_j\whot\mathbf{e}_l)(E_{\mathbf{e}_i}\whot \eta(\mathbf{e}_j, \mathbf{e}_l))
a'_{\mathbf{e}_i, \mathbf{e}_j, \mathbf{e}_l}$$
for all $i, j,l\in I$.
\end{remark}

Assume that $\s$ is a permutation of $I$. Let $P_{\s}=(p_{ij})$ be the corresponding permutation $n\times n$-matrix over $\mathbb Z$, that is,
$$p_{ij}=\left\{\begin{array}{ll}
1,& \mbox{ if } j=\s(i),\\
0,& \mbox{ otherwise},\\
\end{array}\right.$$
where $i, j\in I$.
Let $\mathbf{m}=(m_1, m_2, \cdots, m_n)\in\mathbb{N}^I$. Put
$$\mathbf{m}^{\s}=\mathbf{m}P_{\s}=(m_{\s^{-1}(1)}, m_{\s^{-1}(2)}, \cdots, m_{\s^{-1}(n)}).$$
Then the map $\mathbb{N}^I\ra \mathbb{N}^I$, $\mathbf{m}\mapsto \mathbf{m}^{\s}$,
is an additive monoid isomorphism. Moreover, $\mathbf{e}_i^{\s}=\mathbf{e}_{\s(i)}$
for all $i\in I$.
For $\mathbf{m}=\mathbf{0}$, let $P_{\s}(\mathbf{0})=0$, the $1\times 1$ zero matrix over $\mathbb F$.
For any $\mathbf{m}\neq\mathbf{0}$, let $P_{\s}(\mathbf{m})=(P_{ij})$ be the permutation
$|\mathbf{m}|\times|\mathbf{m}|$-matrix over $\mathbb F$ defined as follows:
if $m_i=0$, then there is no  $\s(i)$-th row $(P_{\s(i)j})_{j\in I}$ and
there is no  $i$-th column $(P_{ji})_{j\in I}$ in $P_{\s}(\mathbf{m})$; if $m_{\s^{-1}(i)}m_j>0$,  then $P_{ij}$ is an $m_{\s^{-1}(i)}\times m_j$-matrix given by
$$P_{ij}=\left\{\begin{array}{ll}
I_{m_j},& \mbox{ if } j=\s^{-1}(i),\\
0,& \mbox{ otherwise }.\\
\end{array}\right.$$
Let $P_{\s}(\mathbf{m})^T$ denote the transposed matrix of $P_{\s}(\mathbf{m})$.
Then $P_{\s}(\mathbf{m})^{-1}=P_{\s}(\mathbf{m})^T=P_{\s^{-1}}(\mathbf{m}^{\s})$.
Obviously, $P_{\s}(\mathbf{e}_i)=I_1$ is the $1\times 1$ identity matrix.

Now let $R'$ be another $\mathbb{Z}_+$-ring with a finite unital
$\mathbb{Z}_+$-basis $\{r'_j\}_{j\in I'}$ such that $r'_ir'_j\neq 0$ for all $i, j\in I'$,
and $A'$ a finite dimensional $\mathbb{F}$-algebra
with a complete set of orthogonal primitive idempotents $\{e'_j\}_{j\in I'}$.
Assume that $I'=\{1, 2, \cdots, n'\}$ and $r'_1=1$, the unity of $R'$,
and that the properties (KS), (Dec), (RUA), (LUA), (RI) and (CI) are satisfied
for $A'$.

Suppose that  $r'_ir'_j=\sum_{k=1}^{n'}c'_{ijk}r'_k$ for some $c'_{ijk}\in\mathbb{Z}_+$, where $1\leqslant i, j\leqslant n'$,
and let $\mathbf{c'}_{ij}=(c'_{ijk})_{1\leqslant k\leqslant n'}=(c'_{ij1}, c'_{ij2}, \cdots, c'_{ijn'})\in \mathbb{N}^{I'}$.
Then one has an $\mathbb F$-algebra $M(R', A', I')$ defined similarly as $M(R, A, I)$  before.  We also assume that ${\rm dim}_{\mathbb F}(e'_1A'e'_1)=1$ and
there is an algebra map $\phi': A'\ot_{\mathbb F}A'\ra M(R', A', I')$
such that the conditions $(\phi 1)$ and $(\phi 2)$ are satisfied for $\phi'$.
Furthermore, assume that there is a family of matrices $\{a'_{i,j,l}\}$
indexed by all triples $(i,j,l)$ of elements of $I'$
such that the conditions $(Ass)(1)$-$(4)$ are satisfied.
Then, by Proposition \ref{ass}, there is a unique family of matrices $\{a'_{\mathbf{m}, \mathbf{s}, \mathbf{t}}\}$
indexed by all triples $(\mathbf{m}, \mathbf{s}, \mathbf{t})$ of elements of $\mathbb{N}^{I'}$ such that the conditions $(1)$-$(5)$ in Proposition \ref{ass} are satisfied.
Thus, one obtains  a tensor category $\widehat{\mathcal{C}'}$ in the same way  as  we did for $\widehat{\mathcal C}$. Denote by $\mathbf{e}'_j$, $1\leqslant j\leqslant n'$, the  corresponding indecomposable objects of $\widehat{\mathcal{C}'}$.

In the next theorem and its proof, in order to simplify the notations,  we shall   use   ${\ot}_{\widehat{\mathcal C}}$ and  ${\ot}_{\widehat{\mathcal{C}'}}$, instead of  ${\whot}_{\widehat{\mathcal C}}$ and  ${\whot}_{\widehat{\mathcal{C}'}}$, to denote  the tensor products in $\widehat{\mathcal{C}}$ and $\widehat{\mathcal{C}'}$, respectively.
For any $\mathbf{m}\in{\rm Ob}(\widehat{\mathcal{C}'})={\mathbb N}^{I'}$, let $E'_{\mathbf m}$ denote the identity element
of the algebra $M_{\mathbf m}(A')$. For any nonzero $\mathbf{m}\in\mathbb{N}^{I'}$,
let $Y'^{\mathbf{m}}_{i,k}\in M_{\mathbf{e}_i\times\mathbf{m}}(A')$
and $X'^{\mathbf{m}}_{i,k}\in M_{\mathbf{m}\times\mathbf{e}_i}(A')$ be the matrices defined in the same way as the matrices $Y^{\mathbf{m}}_{i,k}$ and $X^{\mathbf{m}}_{i,k}$ over $A$  defined after Remark \ref{2.9}   respectively.

\begin{theorem}\label{equi}
With the above notations, $\widehat{\mathcal C}$ and $\widehat{\mathcal{C}'}$ are tensor  equivalent if and only if $n'=n$ (in this case, $I'=I$ and $\mathbf{e}'_i=\mathbf{e}_i$) and there is a permutation $\s$ of $I$ such that the following
conditions are satisfied:\\
{\rm (1)} The additive group homomorphism $R\ra R'$ given by
$r_i\mapsto r'_{\sigma(i)}$, $i\in I$, is a ring isomorphism.\\
{\rm (2)} There is an $\mathbb{F}$-algebra isomorphism $\delta: A\ra A'$
satisfying $\delta(e_i)=e'_{\sigma(i)}$ for all $i\in I$.\\
{\rm (3)} There exists a nonzero scale $\a\in\mathbb F$ and a family of invertible
elements $\varphi_{i,j}$ in the algebra $M_{\mathbf{c}'_{\s(i)\s(j)}}(A')$
indexed by all couples $(i,j)$ of elements of $I$ such that
the following conditions are satisfied:\\
\mbox{\hspace{0.5cm}\rm (a)}  $\varphi_{1, i}
=\varphi_{i, 1}=\a E'_{\mathbf{e}_i^{\s}} (=\a E'_{\mathbf{e}_{\s(i)}}=\a e'_{\s(i)})$;\\
\mbox{\hspace{0.5cm}\rm (b)}
$\varphi_{i', j'}(\d(x)\ot_{\widehat{\mathcal{C}'}}\d(y))=P_{\s}(\mathbf{c}_{i'j'})\d(x\ot_{\widehat{\mathcal{C}}}y)
P_{\s}(\mathbf{c}_{ij})^T\varphi_{i,j}$;\\
\mbox{\hspace{0.5cm}\rm (c)}
$\sum\limits_{t=1}^n\sum\limits_{1\leqslant k\leqslant c_{ijt}}
\d(a_{i,j,l})\d(X^{\mathbf{c}_{ij}}_{t,k}\ot_{\widehat{\mathcal{C}}}e_l)
P_{\s}(\mathbf{c}_{tl})^T\varphi_{t, l}
(Y'^{\mathbf{c}'_{\s(i)\s(j)}}_{\s(t),k}\varphi_{i, j}\ot_{\widehat{\mathcal{C}'}}e'_{\s(l)})$\\
\mbox{\hspace{1cm}} $=\sum\limits_{t=1}^n\sum\limits_{1\leqslant k\leqslant c_{jlt}}
\d(e_i\ot_{\widehat{\mathcal{C}}}X^{\mathbf{c}_{jl}}_{t,k})
P_{\s}(\mathbf{c}_{it})^T\varphi_{i, t}
(e'_{\s(i)}\ot_{\widehat{\mathcal{C}'}}Y'^{\mathbf{c}'_{\s(j)\s(l)}}_{\s(t),k}\varphi_{j, l})
a'_{\s(i), \s(j), \s(l)}$\\
for all $i,j,l,i',j'\in I$, $x\in e_{i'}Ae_i$ and $y\in e_{j'}Ae_j$,
where $\delta: A\ra A'$ is the algebra isomorphism given in $(2)$, $\d(Z)=(\d(z_{ij}))$
for any matrix $Z=(z_{ij})$ over $A$.
\end{theorem}

\begin{proof}
Assume that $(F, \varphi_0, \varphi_2)$ is an $\mathbb{F}$-linear tensor  equivalence  functor from $\widehat{\mathcal C}$ to $\widehat{\mathcal C'}$. Then $(F, \varphi_0, \varphi_2)$ induces a ring isomorphism
$F: r(\widehat{\mathcal C})\ra r(\widehat{\mathcal{C}'})$ given by $F([\mathbf m])=[F(\mathbf m)]$
for any object $\mathbf m$ of $\widehat{\mathcal C}$.
Since $F: \widehat{\mathcal C}\ra \widehat{\mathcal C}'$ is a category equivalence and $\{\mathbf{e}_i|i\in I\}$
is the set of all non-isomorphic indecomposable objects of $\widehat{\mathcal C}$ by Corollary \ref{IndOb},
$\{F(\mathbf{e}_i)|i\in I\}$ is a set of representatives
of the isomorphism classes of indecomposable objects of $\widehat{\mathcal{C}'}$. Again by Corollary \ref{IndOb},
we have $\{F(\mathbf{e}_i)|i\in I\}=\{\mathbf{e}'_j|j\in I'\}$. Hence there is a bijection
$\sigma: I\ra I'$ such that $F(\mathbf{e}_i)=\mathbf{e}'_{\sigma(i)}$ in $\widehat{\mathcal{C}'}$ for any $i\in I$.
This implies $n'=n$, and hence we may regard $I'=I=\{1, 2, \cdots, n\}$. In this case,
$\mathbf{e}'_j=\mathbf{e}_j$ for any $j\in I$, and $\s$ is a permutation of $I$.
Consequently, $F([\mathbf{e}_i])=[F(\mathbf{e}_i)]=[\mathbf{e}_{\sigma(i)}]$ in $r(\mathcal{C}')$, $i\in I$.
Thus, Part (1) follows from Remark \ref{GreenRing}.
Note that $F(\mathbf{e}_1)=\mathbf{e}_1$, and hence $\s(1)=1$.

Since $F: \widehat{\mathcal C}\ra\widehat{\mathcal C}'$ is an
$\mathbb{F}$-linear  equivalence, $F$ induces an $\mathbb{F}$-space isomorphism
$F: {\rm Hom}_{\widehat{\mathcal C}}(\mathbf{e}_j, \mathbf{e}_i)\ra {\rm Hom}_{\widehat{\mathcal{C}'}}(\mathbf{e}_{\s(j)}, \mathbf{e}_{\s(i)})$,
$a\mapsto F(a)$, where $i, j\in I$. In particular,
$F: {\rm End}_{\widehat{\mathcal C}}(\mathbf{e}_i)\ra {\rm End}_{\widehat{\mathcal{C}'}}(\mathbf{e}_{\s(i)})$,
$a\mapsto F(a)$, is an $\mathbb F$-algebra isomorphism for any $i\in I$.
Now we have ${\rm Hom}_{\widehat{\mathcal C}}(\mathbf{e}_j, \mathbf{e}_i)=M_{\mathbf{e}_i\times\mathbf{e}_j}(A)=e_iAe_j$
and ${\rm Hom}_{\widehat{\mathcal{C}'}}(\mathbf{e}_{\s(j)}, \mathbf{e}_{\s(i)})=M_{\mathbf{e}_{\s(i)}\times\mathbf{e}_{\s(j)}}(A')
=e'_{\s(i)}A'e'_{\s(j)}$. Moreover, we have
$$A=\oplus_{1\leqslant i, j\leqslant n}e_iAe_j\ \mbox{ and }
A'=\oplus_{1\leqslant i, j\leqslant n}e'_iA'e'_j=\oplus_{1\leqslant i, j\leqslant n}e'_{\s(i)}A'e'_{\s(j)}.$$
Therefore, the $\mathbb{F}$-linear isomorphisms above induce an $\mathbb{F}$-linear isomorphism
$$\d: A\ra A', \sum_{i,j=1}^na_{ij}\mapsto \sum_{i,j=1}^nF(a_{ij}),\ a_{ij}\in e_iAe_j.$$
Obviously, $\d$ is an $\mathbb F$-algebra isomorphism and $\d(e_i)=e'_{\s(i)}$ for any $i\in I$.
Moreover, $\d(a)=F(a)$ for any $a\in e_iAe_j$, $i, j\in I$.
This shows Part (2).

By Part (2), $\d(e_iAe_j)=e'_{\s(i)}A'e'_{\s(j)}$ for all $i, j\in I$.
For any $X=(X_{ij})\in M_{\mathbf{m}\times\mathbf{s}}(A)$ with $X_{ij}\in M_{m_i\times s_j}(e_iAe_j)$,
$i, j\in I$, we have
$$\begin{array}{rl}
&P_{\s}(\mathbf{m})\d(X)P_{\s}(\mathbf{s})^T=P_{\s}(\mathbf{m})\left(\begin{array}{cccc}
\d(X_{11})&\d(X_{12})&\cdots &\d(X_{1n})\\
\d(X_{21})&\d(X_{22})&\cdots &\d(X_{2n})\\
\cdots&\cdots&\cdots &\cdots\\
\d(X_{n1})&\d(X_{n2})&\cdots &\d(X_{nn})\\
\end{array}\right)P_{\s}(\mathbf{s})^T\\
=&\left(\begin{array}{cccc}
\d(X_{\s^{-1}(1)\s^{-1}(1)})&\d(X_{\s^{-1}(1)\s^{-1}(2)})&\cdots &\d(X_{\s^{-1}(1)\s^{-1}(n)})\\
\d(X_{\s^{-1}(2)\s^{-1}(1)})&\d(X_{\s^{-1}(2)\s^{-1}(2)})&\cdots &\d(X_{\s^{-1}(2)\s^{-1}(n)})\\
\cdots&\cdots&\cdots &\cdots\\
\d(X_{\s^{-1}(n)\s^{-1}(1)})&\d(X_{\s^{-1}(n)\s^{-1}(2)})&\cdots &\d(X_{\s^{-1}(n)\s^{-1}(n)})\\
\end{array}\right),\\
\end{array}$$
and hence $P_{\s}(\mathbf{m})\d(X)P_{\s}(\mathbf{s})^T\in M_{\mathbf{m}^{\s}\times\mathbf{s}^{\s}}(A')$.

Let $i\in I$, and let $m$ be a positive integer.
Since $F(\mathbf{e}_i)=\mathbf{e}_{\s(i)}$ in $\mathcal{C}'$, we have
$$\begin{array}{rl}
F(m\mathbf{e}_i)=&F(\mathbf{e}_i\oplus \mathbf{e}_i\oplus\cdots\oplus\mathbf{e}_i)\\
=&F(\mathbf{e}_i)\oplus F(\mathbf{e}_i)\oplus\cdots\oplus F(\mathbf{e}_i)\\
=&\mathbf{e}_{\s(i)}\oplus \mathbf{e}_{\s(i)}\oplus\cdots\oplus \mathbf{e}_{\s(i)}\\
=&m\mathbf{e}_{\s(i)}\\
\end{array}$$
by Lemma \ref{iso}. Therefore, we may regard that $F$ preserves the order of summands in the equation
$F(m\mathbf{e}_i)=m\mathbf{e}_{\s(i)}$. This means that
$F(Y^{m\mathbf{e}_i}_{i,k})=Y'^{m\mathbf{e}_{\s(i)}}_{\s(i),k}$,
and consequently, $F(X^{m\mathbf{e}_i}_{i,k})=X'^{m\mathbf{e}_{\s(i)}}_{\s(i),k}$,
where $1\leqslant k\leqslant m$.

Let $i, j\in I$, and $m$ and $s$ be positive integers. Then
${\rm Hom}_{\widehat{\mathcal C}}(s\mathbf{e}_j, m\mathbf{e}_i)=M_{m\times s}(e_iAe_j)$
and ${\rm Hom}_{\widehat{\mathcal{C}'}}(s\mathbf{e}_{\s(j)}, m\mathbf{e}_{\s(i)})=M_{m\times s}(e'_{\s(i)}A'e'_{\s(j)})$.
Hence the map
$$F: M_{m\times s}(e_iAe_j)\ra M_{m\times s}(e'_{\s(i)}A'e'_{\s(j)}), \
X\mapsto F(X)$$
is an $\mathbb F$-linear isomorphism.
Let $X=(x_{kl})\in M_{m\times s}(e_iAe_j)={\rm Hom}_{\widehat{\mathcal C}}(s\mathbf{e}_j, m\mathbf{e}_i)$
with $x_{kl}\in e_iAe_j={\rm Hom}_{\widehat{\mathcal C}}(\mathbf{e}_j, \mathbf{e}_i)$. Then
$X=\sum_{k=1}^m\sum_{l=1}^sX^{m\mathbf{e}_i}_{i,k}Y^{m\mathbf{e}_i}_{i,k}
XX^{s\mathbf{e}_j}_{j,l}Y^{s\mathbf{e}_j}_{j,l}=\sum_{k=1}^m\sum_{l=1}^sX^{m\mathbf{e}_i}_{i,k}x_{kl}
Y^{s\mathbf{e}_j}_{j,l}$,
and  hence
$$\begin{array}{rcl}
F(X)&=&\sum_{k=1}^m\sum_{l=1}^sF(X^{m\mathbf{e}_i}_{i,k})F(x_{kl})F(Y^{s\mathbf{e}_j}_{j,l})\\
&=&\sum_{k=1}^m\sum_{l=1}^sX'^{m\mathbf{e}_{\s(i)}}_{\s(i),k}\d(x_{kl})Y'^{s\mathbf{e}_{\s(j)}}_{\s(j),l}\\
&=&(\d(x_{kl}))=\d(X).\\
\end{array}$$

Let $\mathbf{m}, \mathbf{s}\in{\rm Ob}(\widehat{\mathcal C})$. Then $\mathbf{m}=\oplus_{i=1}^nm_i\mathbf{e}_i$,
and hence $F(\mathbf{m})=\oplus_{i=1}^nF(m_i\mathbf{e}_i)=\oplus_{i=1}^nm_i\mathbf{e}_{\s(i)}
\cong\oplus_{i=1}^nm_{\s^{-1}(i)}\mathbf{e}_i=\mathbf{m}^{\s}$.
Hence $F(\mathbf{m})=\mathbf{m}^{\s}$ by Lemma \ref{iso}.
If $\pi_i$ is the projection from $\mathbf{m}$ to $m_i\mathbf{e}_i$,
then $F(\pi_i)$ is the projection from $\mathbf{m}^{\s}$ to $m_i\mathbf{e}_{\s(i)}$.
If $\tau_i$ is the embedding from $m_i\mathbf{e}_i$ into $\mathbf{m}$,
then $F(\tau_i)$ is the embedding from $m_i\mathbf{e}_{\s(i)}$ into $\mathbf{m}^{\s}$.
Thus, a similar argument to the above one shows that the $\mathbb F$-linear isomorphism
$F: {\rm Hom}_{\widehat{\mathcal C}}(\mathbf{s}, \mathbf{m})=M_{\mathbf{m}\times \mathbf{s}}(A)
\ra M_{\mathbf{m}^{\s}\times \mathbf{s}^{\s}}(A')={\rm Hom}_{\widehat{\mathcal{C}'}}(\mathbf{s}^{\s}, \mathbf{m}^{\s})$
is given by
$$F(X)=P_{\s}(\mathbf{m})\d(X)P_{\s}(\mathbf{s})^T.$$
In particular, we have $F(X^{\mathbf{m}}_{i,k})=X'^{\mathbf{m}^{\s}}_{\s(i),k}$
and $F(Y^{\mathbf{m}}_{i,k})=Y'^{\mathbf{m}^{\s}}_{\s(i),k}$
for any $\mathbf{0}\neq\mathbf{m}=(m_i)_{i\in I}\in\mathbb{N}^I$,
$1\leqslant i\leqslant n$ with $m_i>0$ and $1\leqslant k\leqslant m_i$.

Since $(F, \varphi_2, \varphi_0)$ is a tensor functor from $\widehat{\mathcal C}$ to $\widehat{\mathcal{C}'}$,
$\varphi_0: \mathbf{e}_1\ra F(\mathbf{e}_1)$ is an isomorphism and
$$\varphi_2(\mathbf{m}, \mathbf{s}): F(\mathbf{m})\ot_{\widehat{\mathcal{C}'}}F(\mathbf{s})
\ra F(\mathbf{m}\ot_{\widehat{\mathcal{C}}}\mathbf{s}), \ \mathbf{m, s}\in \mathbb{N}^I$$
form a family of natural isomorphisms
subject to the following three  commutative diagrams (denoted  (D3.1), (D3.2) and (D3.3), respectively):
$$\begin{array}{ccc}
(F(\mathbf{m})\ot_{\widehat{\mathcal{C}'}}F(\mathbf{s}))\ot_{\widehat{\mathcal{C}'}}F(\mathbf{t})
&\xrightarrow{a'_{F(\mathbf{m}), F(\mathbf{s}), F(\mathbf{t})}}
&F(\mathbf{m})\ot_{\widehat{\mathcal{C}'}}(F(\mathbf{s})\ot_{\widehat{\mathcal{C}'}}F(\mathbf{t}))\\
\mbox{\hspace{2cm}}\downarrow\varphi_2(\mathbf{m}, \mathbf{s})\ot_{\widehat{\mathcal{C}'}}{\rm id}_{F(\mathbf{t})}
&&\hspace{1.8cm}\downarrow{\rm id}_{F(\mathbf m)}\ot_{\widehat{\mathcal{C}'}}\varphi_2(\mathbf{s}, \mathbf{t})\\
F(\mathbf{m}\ot_{\widehat{\mathcal{C}}}\mathbf{s})\ot_{\widehat{\mathcal{C}'}}F(\mathbf{t})
&  \mathrm{(D3.1)} &F(\mathbf{m})\ot_{\widehat{\mathcal{C}'}}F(\mathbf{s}\ot_{\widehat{\mathcal{C}}}\mathbf{t})\\
\mbox{\hspace{1.5cm}}\downarrow\varphi_2(\mathbf{m}\ot_{\widehat{\mathcal{C}}}\mathbf{s}, \mathbf{t})
& &\hspace{1.3cm}\downarrow\varphi_2(\mathbf{m}, \mathbf{s}\ot_{\widehat{\mathcal{C}}}\mathbf{t})\\
F((\mathbf{m}\ot_{\widehat{\mathcal{C}}}\mathbf{s})\ot_{\widehat{\mathcal{C}}}\mathbf{t})
&\xrightarrow{F(a_{\mathbf{m}, \mathbf{s}, \mathbf{t}})}
&F(\mathbf{m}\ot_{\widehat{\mathcal{C}}}(\mathbf{s}\ot_{\widehat{\mathcal{C}}}\mathbf{t}))\\
\end{array}  $$
$$\begin{array}{ccc}
\mathbf{e}_1\ot_{\widehat{\mathcal{C}'}}F(\mathbf{m})&\xrightarrow{l_{F(\mathbf{m})}}&F(\mathbf{m})\\
\mbox{\hspace{1.7cm}}\downarrow\varphi_0\ot_{\widehat{\mathcal{C}'}}{\rm id}_{F(\mathbf m)}&&\hspace{0.9cm}\uparrow F(l_{\mathbf m})\\
F(\mathbf{e}_1)\ot_{\widehat{\mathcal{C}'}}F(\mathbf{m})&\xrightarrow{\varphi_2(\mathbf{e}_1, \mathbf{m})}&F(\mathbf{e}_1\ot_{\widehat{\mathcal{C}}}\mathbf{m})\\
\end{array} \ \  \ \mathrm{(D3.2)}$$
and
$$\begin{array}{ccc}
F(\mathbf{m})\ot_{\widehat{\mathcal{C}'}}\mathbf{e}_1&\xrightarrow{r_{F(\mathbf{m})}}&F(\mathbf{m})\\
\mbox{\hspace{1.7cm}}\downarrow{\rm id}_{F(\mathbf m)}\ot_{\widehat{\mathcal{C}'}}\varphi_0
&&\hspace{0.9cm}\uparrow F(r_{\mathbf m})\\
F(\mathbf{m})\ot_{\widehat{\mathcal{C}'}}F(\mathbf{e}_1)&\xrightarrow{\varphi_2(\mathbf{m}, \mathbf{e}_1)}
&F(\mathbf{m}\ot_{\widehat{\mathcal{C}}}\mathbf{e}_1)\\
\end{array} \ \  \ \mathrm{(D3.3)}$$
for all objects $\mathbf{m}, \mathbf{s}, \mathbf{t}$ in $\widehat{\mathcal C}$.

Since $F(\mathbf{e}_1)=\mathbf{e}_1$, $\varphi_0$ is an automorphism of $\mathbf{e}_1$ in $\widehat{\mathcal{C}'}$.
However, ${\rm End}_{\widehat{\mathcal{C}'}}(\mathbf{e}_1)=e'_1A'e'_1={\mathbb F}e'_1\cong\mathbb{F}$.
Hence there is a nonzero scale $\b\in\mathbb F$ such that $\varphi_0=\b{\rm id}_{\mathbf{e}_1}=\b e'_1$.
Let $\a=\b^{-1}$.

By Part (1), one gets that $c_{ijk}=c'_{\s(i)\s(j)\s(k)}$ for all $i, j, k\in I$,
and hence $F(\mathbf{c}_{ij})=\mathbf{c}_{ij}^{\s}=\mathbf{c}'_{\s(i)\s(j)}$ for all $i, j\in I$.
Now for any $\mathbf{m}, \mathbf{s}\in{\rm Ob}(\widehat{\mathcal C})$, we have
$$\begin{array}{rcl}
F(\mathbf{m}\ot_{\widehat{\mathcal{C}}}\mathbf{s})&=&(\mathbf{m}\ot_{\widehat{\mathcal{C}}}\mathbf{s})^{\s}
=(\sum_{i, j=1}^nm_is_j\mathbf{c}_{ij})^{\s}\\
&=&\sum_{i, j=1}^nm_is_j\mathbf{c}_{ij}^{\s}=\sum_{i, j=1}^nm_is_j\mathbf{c}'_{\s(i)\s(j)}\\
&=&\sum_{i, j=1}^nm_{\s^{-1}(i)}s_{\s^{-1}(j)}\mathbf{c}'_{ij}=\mathbf{m}^{\s}\ot_{\widehat{\mathcal{C}'}}\mathbf{s}^{\s}\\
&=&F(\mathbf{m})\ot_{\widehat{\mathcal{C}'}}F(\mathbf{s}).\\
\end{array}$$
Hence $\varphi_2(\mathbf{m}, \mathbf{s})$ is an invertible element in the algebra
$M_{(\mathbf{m}\ot_{\widehat{\mathcal{C}}}\mathbf{s})^{\s}}(A')$.
Let $\varphi_{i,j}=\varphi_2(\mathbf{e}_i,\mathbf{e}_j)$ for any $i,j\in I$.
Since $(\mathbf{e}_i\ot_{\widehat{\mathcal{C}}}\mathbf{e}_j)^{\s}=(\mathbf{c}_{ij})^{\s}=\mathbf{c}'_{\s(i)\s(j)}$,
we see that $\varphi_{i,j}$ is an invertible matrix in $M_{\mathbf{c}'_{\s(i)\s(j)}}(A')$.

Since $\{\varphi_2(\mathbf{m}, \mathbf{s})\}$ is a family of natural isomorphisms,
we have $\varphi_2(\mathbf{e}_{i'}, \mathbf{e}_{j'})(F(x)\ot_{\widehat{\mathcal{C}'}}F(y))
=F(x\ot_{\widehat{\mathcal{C}}}y)\varphi_2(\mathbf{e}_i, \mathbf{e}_j)$ for any $i,i',j,j'\in I$, $x\in e_{i'}Ae_i$ and $y\in e_{j'}Ae_j$,
that is,
$$\varphi_{i', j'}(\d(x)\ot_{\widehat{\mathcal{C}'}}\d(y))
=P_{\s}(\mathbf{c}_{i'j'})\d(x\ot_{\widehat{\mathcal{C}}}y)
P_{\s}(\mathbf{c}_{ij})^T \varphi_{i, j}.$$
In a similar manner,  we obtain
$\varphi_2(\mathbf{m}, \mathbf{s})(F(X^{\mathbf{m}}_{i,k})\ot_{\widehat{\mathcal{C}'}}F(X^{\mathbf{s}}_{j,l}))
=F(X^{\mathbf{m}}_{i,k}\ot_{\widehat{\mathcal{C}}}X^{\mathbf{s}}_{j,l})\varphi_2(\mathbf{e}_i, \mathbf{e}_j)$ for  $\mathbf{0\neq m}, \mathbf{0\neq s}\in\mathbb{N}^I$,
that is,
$$\varphi_2(\mathbf{m}, \mathbf{s})(X'^{\mathbf{m}^{\s}}_{\s(i),k}\ot_{\widehat{\mathcal{C}'}}X'^{\mathbf{s}^{\s}}_{\s(j),l})
=P_{\s}(\mathbf{m}\ot_{\widehat{\mathcal C}}\mathbf{s})\d(X^{\mathbf{m}}_{i,k}\ot_{\widehat{\mathcal{C}}}X^{\mathbf{s}}_{j,l})
P_{\s}(\mathbf{c}_{ij})^T\varphi_{i, j}.$$
This implies that
$$\begin{array}{rl}
\hspace{1cm}\varphi_2(\mathbf{m}, \mathbf{s})=&
\sum\limits_{i,j=1}^n\sum\limits_{1\leqslant k\leqslant m_i}\sum\limits_{1\leqslant l\leqslant s_j}
P_{\s}(\mathbf{m}\ot_{\widehat{\mathcal C}}\mathbf{s})\d(X^{\mathbf{m}}_{i,k}\ot_{\widehat{\mathcal C}}X^{\mathbf{s}}_{j,l})\\
&\hspace{3cm}\cdot P_{\s}(\mathbf{c}_{ij})^T\varphi_{i, j}
(Y'^{\mathbf{m}^{\s}}_{\s(i),k}\ot_{\widehat{\mathcal{C}'}}Y'^{\mathbf{s}^{\s}}_{\s(j),l})\hspace{1cm} (3.3)\\
\end{array}$$
Now letting $\mathbf{m}=\mathbf{e}_i$, $\mathbf{s}=\mathbf{e}_j$ and $\mathbf{t}=\mathbf{e}_l$
in the diagram (D3.1) and using Eq.(3.3), one obtains:
$$\begin{array}{rl}
&\sum\limits_{t=1}^n\sum\limits_{1\leqslant k\leqslant c_{ijt}}
\d(a_{i,j,l})\d(X^{\mathbf{c}_{ij}}_{t,k}\ot_{\widehat{\mathcal C}}e_l)
P_{\s}(\mathbf{c}_{tl})^T\varphi_{t, l}
(Y'^{\mathbf{c}'_{\s(i)\s(j)}}_{\s(t),k}\varphi_{i, j}\ot_{\widehat{\mathcal{C}'}}e'_{\s(l)})\\
=&\sum\limits_{t=1}^n\sum\limits_{1\leqslant k\leqslant c_{jlt}}
\d(e_i\ot_{\widehat{\mathcal C}}X^{\mathbf{c}_{jl}}_{t,k})
P_{\s}(\mathbf{c}_{it})^T\varphi_{i, t}
(e'_{\s(i)}\ot_{\widehat{\mathcal{C}'}}Y'^{\mathbf{c}'_{\s(j)\s(l)}}_{\s(t),k}\varphi_{j, l})
a'_{\s(i), \s(j), \s(l)}.\\
\end{array}$$

Note that the left and the right unit constraints are the identities in both of $\widehat{\mathcal C}$ and $\widehat{\mathcal{C}'}$.
Putting $\mathbf{m}=\mathbf{e}_i$ in the diagrams (D3.2) and (D3.3), one gets that
$\varphi_{1,i}=\a e'_{\s(i)}$ and $\varphi_{i,1}=\a e'_{\s(i)}$, $i\in I$.
Thus, we have shown that Part (3) holds.

Conversely, assume that $n'=n$, i.e. $I'=I$, and that there is a permutation $\s$ of $I$
such that the three conditions in the theorem are satisfied.

Define a functor $F: \widehat{\mathcal C}\ra \widehat{\mathcal{C}'}$ by
$$F(\mathbf{m})=\mathbf{m}^{\s},\ F(X)=P_{\s}(\mathbf{m})\d(X)P_{\s}(\mathbf{s})^T,$$
where $\mathbf{m}, \mathbf{s}\in{\rm Ob}(\widehat{\mathcal C})=\mathbb{N}^I$,
$X\in M_{\mathbf{m}\times\mathbf{s}}(A)={\rm Hom}_{\widehat{\mathcal C}}(\mathbf{s}, \mathbf{m})$.

Since the map $\mathbb{N}^I\ra \mathbb{N}^I$, $\mathbf{m}\mapsto \mathbf{m}^{\s}$,
is an additive monoid isomorphism, $F: {\rm Ob}(\widehat{\mathcal C})\ra {\rm Ob}(\widehat{\mathcal{C}'})$
is a bijection.
Moreover, $F(\mathbf{e}_i)=\mathbf{e}_{\s(i)}$
for all $i\in I$. By Part (2), $\d(e_iAe_j)=e'_{\s(i)}A'e'_{\s(j)}$ and the restriction map
$\d: e_iAe_j \ra e'_{\s(i)}A'e'_{\s(j)}$ is an $\mathbb F$-linear isomorphism.
By the discussion before, the map
$$F: {\rm Hom}_{\widehat{\mathcal C}}(\mathbf{s}, \mathbf{m})\ra
{\rm Hom}_{\widehat{\mathcal{C}'}}(\mathbf{s}^{\s}, \mathbf{m}^{\s}),\
X\mapsto P_{\s}(\mathbf{m})\d(X)P_{\s}(\mathbf{s})^T$$
is an $\mathbb F$-linear isomorphism. Moreover,
$F(XY)=F(X)F(Y)$ for any $X\in {\rm Hom}_{\widehat{\mathcal C}}(\mathbf{s}, \mathbf{m})$
and $Y\in {\rm Hom}_{\widehat{\mathcal C}}(\mathbf{t}, \mathbf{s})$.
Thus, $F$ is a well-defined $\mathbb F$-linear functor from $\widehat{\mathcal C}$ to
$\widehat{\mathcal{C}'}$. Obviously, $F$ is a category isomorphism functor.

By Part (1), one have $\s(1)=1$ and $F(\mathbf{c}_{ij})=\mathbf{c}'_{\s(i)\s(j)}$ for all $i, j\in I$,
as shown before. Hence $F(\mathbf{e}_1)=\mathbf{e}_1$. Let $\varphi_0=\a^{-1}{\rm id}_{\mathbf{e}_1}=\a^{-1}e'_1$.
Then $\varphi_0$ is an isomorphism  from $\mathbf{e}_1$ to $F(\mathbf{e}_1)$ in $\widehat{\mathcal{C}'}$.

For any $\mathbf{m}, \mathbf{s}\in{\rm Ob}(\widehat{\mathcal C})=\mathbb{N}^I$, define a matrix
$\varphi_2(\mathbf{m}, \mathbf{s})\in M_{(\mathbf{m}\ot_{\widehat{\mathcal C}}\mathbf{s})^{\s}}(A')
=M_{\mathbf{m}^{\s}\ot_{\widehat{\mathcal{C}'}}\mathbf{s}^{\s}}(A')$ as follows:
$\varphi_2(\mathbf{m}, \mathbf{s})=0$ if $\mathbf{m}=\mathbf{0}$ or $\mathbf{s}=\mathbf{0}$;
and
$$\begin{array}{rl}
\varphi_2(\mathbf{m}, \mathbf{s})=&
\sum\limits_{i,j=1}^n\sum\limits_{1\leqslant k\leqslant m_i}\sum\limits_{1\leqslant l\leqslant s_j}
P_{\s}(\mathbf{m}\ot_{\widehat{\mathcal C}}\mathbf{s})\d(X^{\mathbf{m}}_{i,k}\ot_{\widehat{\mathcal{C}}}X^{\mathbf{s}}_{j,l})\\
&\hspace{3cm}\cdot P_{\s}(\mathbf{c}_{ij})^T\varphi_{i, j}
(Y'^{\mathbf{m}^{\s}}_{\s(i),k}\ot_{\widehat{\mathcal{C}'}}Y'^{\mathbf{s}^{\s}}_{\s(j),l})\\
\end{array}$$
if $\mathbf{m}\neq\mathbf{0}$ and $\mathbf{s}\neq\mathbf{0}$.
Then $\varphi_2(\mathbf{m}, \mathbf{s})$ can be regarded as an isomorphism
from $F(\mathbf{m})\ot_{\widehat{\mathcal{C}'}}F(\mathbf{s})$ to $F(\mathbf{m}\ot_{\widehat{\mathcal{C}}}\mathbf{s})$
in $\widehat{\mathcal{C}'}$.
By an argument similar to the proof of Proposition \ref{ass}, one can check from  Condition (3b) that the morphisms
$$\varphi_2(\mathbf{m}, \mathbf{s}): F(\mathbf{m})\ot_{\widehat{\mathcal{C}'}}F(\mathbf{s})\ra F(\mathbf{m}\ot_{\widehat{\mathcal{C}}}\mathbf{s}), \ \mathbf{m}, \mathbf{s}\in \mathbb{N}^I$$
form a family of natural isomorphisms, where $\mathbf{m}, \mathbf{s}$ run over ${\rm Ob}(\widehat{\mathcal C})$.
From Condition (3a), one can check that the diagrams (D3.2) and (D3.3) are commutative.
Similarly, from Condition (3c), one can show that the diagram (D3.1) commutes.
It follows that $(F, \varphi_0, \varphi_2)$ is a tensor functor
from $\widehat{\mathcal C}$ to $\widehat{\mathcal{C}'}$.
Consequently, $(F, \varphi_0, \varphi_2)$ is a tensor equivalence functor
from $\widehat{\mathcal C}$ to $\widehat{\mathcal{C}'}$.
\end{proof}

\begin{remark}\label{phi-2}
In Theorem \ref{equi}, by replacing $\varphi_{i,j}$
with $\a^{-1}\varphi_{i,j}$, we may assume that $\a=1$.
\end{remark}

Now we return to the case in Section 2. Let $\mathcal C$ be a Krull-Schmit abelian tensor category of finite rank.
In what follows, we keep the notations of the last section.  The Green ring $r(\mathcal C)$ of $\mathcal C$
is a $\mathbb{Z}_+$-ring with a finite unital $\mathbb{Z}_+$-basis $\{[V_i]|i\in I\}$;
the Auslander algebra $A(\mathcal C)$ of $\mathcal C$ is a finite dimensional $\mathbb F$-algebra
with a complete set $\{e_i|i\in I\}$ of orthogonal primitive
elements, where $I=\{1, 2, \cdots, n\}$, $V_1$ is the unit object $\mathbf{1}$ of $\mathcal C$.
From Lemma \ref{proj-class} and Propositions \ref{KerCok} and \ref{decom},
one knows that the conditions (KS), (RUA), (LUA), (Dec), (RI) and (CI) are satisfied for $(A(\mathcal C), \{e_i|i\in I\})$.
Furthermore, let the algebra map $\phi_{\mathcal C}$ and the family
$\{a_{i,j,l}\}_{i,j,l\in I}$ of matrices $a_{i,j,l}\in M_{\mathbf{e}_i\ot\mathbf{e}_j\ot\mathbf{e}_l}(A(\mathcal C))$
be given as in Section 2.  By Propositions \ref{tom} and \ref{Ass},
the conditions $(\phi 1)$ and $(\phi 2)$ are satisfied for $\phi_{\mathcal C}$,
and the conditions $(Ass)$ are satisfied for $\{a_{i,j,l}\}_{i,j,l\in I}$.
Thus, from the data $(r(\mathcal C), A(\mathcal C), \phi_{\mathcal C}, \{a_{i,j,l}\}_{i,j,l\in I})$,
one can construct a tensor category $\widehat{\mathcal C}$ over $\mathbb F$.
In what follows, let $\ot$ and $\whot $
denote the tensor products  in $\mathcal C$ and $\widehat{\mathcal C}$  respectively as before.
For any $\mathbf{m}, \mathbf{s}\in\mathbb{N}^I$,
$\theta(\mathbf{m}, \mathbf{s}): V^{(\mathbf m)}\ot V^{(\mathbf s)}\ra V^{(\mathbf{m}\whot\mathbf{s})}$ is
an isomorphism in $\mathcal C$ given in the last section.
Then we have the following  lemma.

\begin{lemma}\label{associator}
Let $a_{\mathbf{m}, \mathbf{s},\mathbf{t}}$ be the associativity constraint of
$\widehat{\mathcal C}$ induced by $\{a_{i,j,l}\}_{i,j,l\in I}$ as before,
$\mathbf{m}, \mathbf{s},\mathbf{t}\in\mathbb{N}^I$. Then
$$a_{\mathbf{m}, \mathbf{s}, \mathbf{t}}\theta(\mathbf{m}\whot\mathbf{s}, \mathbf{t})
(\theta(\mathbf{m}, \mathbf{s})\ot E_{\mathbf{t}})
=\theta(\mathbf{m}, \mathbf{s}\whot\mathbf{t})
(E_{\mathbf{m}}\ot\theta(\mathbf{s}, \mathbf{t})).$$
\end{lemma}

\begin{proof}
We may assume that $\mathbf{m}\whot\mathbf{s}\whot\mathbf{t}\neq\mathbf 0$. Then by Lemma \ref{diag2}, we have
$$\begin{array}{rl}
&a_{\mathbf{m}, \mathbf{s}, \mathbf{t}}\theta(\mathbf{m}\whot\mathbf{s}, \mathbf{t})
(\theta(\mathbf{m}, \mathbf{s})\ot E_{\mathbf{t}})\\
=&\sum\limits_{i, j, l=1}^n
\sum\limits_{1\leqslant k_1\leqslant m_i}\sum\limits_{1\leqslant k_2\leqslant s_j}
\sum\limits_{1\leqslant k_3\leqslant t_l}
(X^{\mathbf m}_{i,k_1}\whot (X^{\mathbf s}_{j,k_2}\whot X^{\mathbf t}_{l,k_3}))a_{i,j,l}\\
&((Y^{\mathbf m}_{i,k_1}\whot Y^{\mathbf s}_{j,k_2})\whot Y^{\mathbf t}_{l,k_3})
\theta(\mathbf{m}\whot\mathbf{s}, \mathbf{t})
(\theta(\mathbf{m}, \mathbf{s})\ot E_{\mathbf{t}})\\
=&\sum\limits_{i, j, l=1}^n
\sum\limits_{1\leqslant k_1\leqslant m_i}\sum\limits_{1\leqslant k_2\leqslant s_j}
\sum\limits_{1\leqslant k_3\leqslant t_l}
(X^{\mathbf m}_{i,k_1}\whot (X^{\mathbf s}_{j,k_2}\whot X^{\mathbf t}_{l,k_3}))a_{i,j,l}\\
&\theta(\mathbf{c}_{ij}, \mathbf{e}_l)
((Y^{\mathbf m}_{i,k_1}\whot Y^{\mathbf s}_{j,k_2})\ot Y^{\mathbf t}_{l,k_3})
(\theta(\mathbf{m}, \mathbf{s})\ot E_{\mathbf{t}})\\
=&\sum\limits_{i, j, l=1}^n
\sum\limits_{1\leqslant k_1\leqslant m_i}\sum\limits_{1\leqslant k_2\leqslant s_j}
\sum\limits_{1\leqslant k_3\leqslant t_l}
(X^{\mathbf m}_{i,k_1}\whot (X^{\mathbf s}_{j,k_2}\whot X^{\mathbf t}_{l,k_3}))a_{i,j,l}\\
&\theta(\mathbf{c}_{ij}, \mathbf{e}_l)(\theta_{ij}\ot e_l)
(Y^{\mathbf m}_{i,k_1}\ot Y^{\mathbf s}_{j,k_2}\ot Y^{\mathbf t}_{l,k_3})\\
=&\sum\limits_{i, j, l=1}^n
\sum\limits_{1\leqslant k_1\leqslant m_i}\sum\limits_{1\leqslant k_2\leqslant s_j}
\sum\limits_{1\leqslant k_3\leqslant t_l}
(X^{\mathbf m}_{i,k_1}\whot (X^{\mathbf s}_{j,k_2}\whot X^{\mathbf t}_{l,k_3}))
\theta(\mathbf{e}_i, \mathbf{c}_{jl})\\
&(e_i\ot\theta_{j,l})
(Y^{\mathbf m}_{i,k_1}\ot Y^{\mathbf s}_{j,k_2}\ot Y^{\mathbf t}_{l,k_3})\\
=&\sum\limits_{i, j, l=1}^n
\sum\limits_{1\leqslant k_1\leqslant m_i}\sum\limits_{1\leqslant k_2\leqslant s_j}
\sum\limits_{1\leqslant k_3\leqslant t_l}
\theta(\mathbf{m}, \mathbf{s}\whot\mathbf{t})
(X^{\mathbf m}_{i,k_1}\ot(X^{\mathbf s}_{j,k_2}\whot X^{\mathbf t}_{l,k_3}))\\
&(e_i\ot\theta_{j,l})
(Y^{\mathbf m}_{i,k_1}\ot Y^{\mathbf s}_{j,k_2}\ot Y^{\mathbf t}_{l,k_3})\\
=&\sum\limits_{i, j, l=1}^n
\sum\limits_{1\leqslant k_1\leqslant m_i}\sum\limits_{1\leqslant k_2\leqslant s_j}
\sum\limits_{1\leqslant k_3\leqslant t_l}
\theta(\mathbf{m}, \mathbf{s}\whot\mathbf{t})(E_{\mathbf m}\ot\theta(\mathbf{s},\mathbf{t}))\\
&(X^{\mathbf m}_{i,k_1}\ot X^{\mathbf s}_{j,k_2}\ot X^{\mathbf t}_{l,k_3})
(Y^{\mathbf m}_{i,k_1}\ot Y^{\mathbf s}_{j,k_2}\ot Y^{\mathbf t}_{l,k_3})\\
=&\theta(\mathbf{m}, \mathbf{s}\whot\mathbf{t})(E_{\mathbf m}\ot\theta(\mathbf{s},\mathbf{t})).\\
\end{array}$$
\end{proof}

Finally, we show that the two tensor categories $\widehat{\mathcal C}$ and $\mathcal C$ are tensor equivalent.

\begin{theorem}\label{main}
$\widehat{\mathcal C}$ and $\mathcal C$ are tensor equivalent.
\end{theorem}

\begin{proof}
For any $\mathbf{m}\in{\mathbb N}^I$, let $V^{(\mathbf m)}\in{\rm Ob}(\mathcal C)$ be defined as in Section 2. Then
$${\rm Hom}_{\mathcal C}(V^{(\mathbf m)}, V^{(\mathbf s)})=M_{{\mathbf s}\times {\mathbf m}}(A(\mathcal C))$$
for all $\mathbf{m}, \mathbf{s}\in{\mathbb N}^I$.
Define a functor $F: \widehat{\mathcal C}\ra \mathcal C$ by
$$F(\mathbf m)=V^{(\mathbf m)}, \ F(X)=X,$$
where $\mathbf{m}, \mathbf{s}\in{\rm Ob}(\widehat{\mathcal C})={\mathbb N}^I$,
$X\in M_{\mathbf{s}\times \mathbf{m}}(A(\mathcal C))$.
Obviously, $F$ is a well-defined $\mathbb F$-linear functor,
and $F(\mathbf{e}_1)=V_1=\mathbf{1}$. Moreover,
$F$ is fully faithful. Since $\mathcal C$ is a Krull-Schmidt category
and $\{V_1, V_2, \cdots, V_n\}$ is a set of representatives of isomorphism classes of
indecomposable objects of $\mathcal C$, $F$ is essentially surjective.
It follows from \cite[Proposition XI.1.5]{Ka} that $F: \widehat{\mathcal C}\ra \mathcal C$
is an equivalence of categories.

For any $\mathbf{m}, \mathbf{s}\in{\rm Ob}(\widehat{\mathcal C})$, put
$$\varphi_0={\rm id}_{\mathbf 1}=e_1, \ \varphi_2(\mathbf{m}, \mathbf{s})=\theta(\mathbf{m}, \mathbf{s}).$$
Then $\varphi_0: \mathbf{1}\ra F(\mathbf{e}_1)$ is an isomorphism in
$\mathcal C$. It follows from Lemma \ref{diag2} that
$$\varphi_2(\mathbf{m}, \mathbf{s}): F(\mathbf{m})\ot F(\mathbf{s})\ra F(\mathbf{m}\whot\mathbf{s}),\  \mathbf{m, s}\in \mathbb{N}^I$$
form a family of natural isomorphisms indexed by all couples $(\mathbf{m}, \mathbf{s})$
of objects of $\widehat{\mathcal C}$.
Note that the left and the right unit constraints are the identities in both
$\widehat{\mathcal C}$ and ${\mathcal C}$. Since $\theta(\mathbf{m}, \mathbf{e}_1)
=\theta(\mathbf{e}_1, \mathbf{m})=E_{\mathbf m}$ by Lemma \ref{invert}, the following two diagrams
$$\begin{array}{ccc}
\mathbf{1}\ot F(\mathbf{m})&\xrightarrow{l_{F(\mathbf{m})}}&F(\mathbf{m})\\
\mbox{\hspace{1.7cm}}\downarrow\varphi_0\ot{\rm id}_{F(\mathbf m)}&&\hspace{0.9cm}\uparrow F(l_{\mathbf m})\\
F(\mathbf{e}_1)\ot F(\mathbf{m})&\xrightarrow{\varphi_2(\mathbf{e}_1, \mathbf{m})}&F(\mathbf{e}_1\whot\mathbf{m})\\
\end{array}$$
and
$$\begin{array}{ccc}
F(\mathbf{m})\ot\mathbf{1}&\xrightarrow{r_{F(\mathbf{m})}}&F(\mathbf{m})\\
\mbox{\hspace{1.7cm}}\downarrow{\rm id}_{F(\mathbf m)}\ot\varphi_0
&&\hspace{0.9cm}\uparrow F(r_{\mathbf m})\\
F(\mathbf{m})\ot F(\mathbf{e}_1)&\xrightarrow{\varphi_2(\mathbf{m}, \mathbf{e}_1)}
&F(\mathbf{m}\whot\mathbf{e}_1)\\
\end{array}$$
commute for all object $\mathbf{m}$ of $\widehat{\mathcal C}$.
By Lemma \ref{associator}, we have
$$a_{\mathbf{m}, \mathbf{s}, \mathbf{t}}\theta(\mathbf{m}\whot\mathbf{s}, \mathbf{t})
(\theta(\mathbf{m}, \mathbf{s})\ot E_{\mathbf{t}})
=\theta(\mathbf{m}, \mathbf{s}\whot\mathbf{t})
(E_{\mathbf{m}}\ot\theta(\mathbf{s}, \mathbf{t}))$$
for $\mathbf{m}, \mathbf{s}, \mathbf{t}\in{\mathbb N}^I$.
Since $\mathcal C$ is a strict tensor category and
$F(a_{\mathbf{m}, \mathbf{s}, \mathbf{t}})=a_{\mathbf{m}, \mathbf{s}, \mathbf{t}}$, the above
equation shows that the following diagram
$$\begin{array}{ccc}
(F(\mathbf{m})\ot F(\mathbf{s}))\ot F(\mathbf{t})
&\xrightarrow{a_{F(\mathbf{m}), F(\mathbf{s}), F(\mathbf{t})}}
&F(\mathbf{m})\ot(F(\mathbf{s})\ot F(\mathbf{t}))\\
\mbox{\hspace{1.6cm}}\downarrow\varphi_2(\mathbf{m}, \mathbf{s})\ot{\rm id}_{F(\mathbf{t})}
&&\hspace{1.8cm}\downarrow{\rm id}_{F(\mathbf m)}\ot\varphi_2(\mathbf{s}, \mathbf{t})\\
F(\mathbf{m}\whot\mathbf{s})\ot F(\mathbf{t})
&&F(\mathbf{m})\ot F(\mathbf{s}\whot\mathbf{t})\\
\mbox{\hspace{0.9cm}}\downarrow\varphi_2(\mathbf{m}\whot\mathbf{s}, \mathbf{t})
&&\hspace{1.3cm}\downarrow\varphi_2(\mathbf{m}, \mathbf{s}\whot\mathbf{t})\\
F((\mathbf{m}\whot\mathbf{s})\whot\mathbf{t})
&\xrightarrow{F(a_{\mathbf{m}, \mathbf{s}, \mathbf{t}})}
&F(\mathbf{m}\whot(\mathbf{s}\whot\mathbf{t}))\\
\end{array}$$
commute for all objects $\mathbf{m}, \mathbf{s}, \mathbf{t}$ of $\widehat{\mathcal C}$.
Thus, we have proven that $(F, \varphi_0, \varphi_2): \widehat{\mathcal C}\ra \mathcal{C}$
is a tensor functor. It follows that
$\widehat{\mathcal C}$ and $\mathcal C$ are tensor equivalent.
\end{proof}

From Lemmas \ref{orthidem} and \ref{proj-class}, $A(\mathcal C)$ is isomorphic to
a quotient of ${\mathbb F}Q$ for some quiver $Q$ with $Q_0=I$ since $\mathbb{F}$ is algebraically closed.

In \cite{CVZ}, we computed the Green ring of Taft algebra $H_n(q)$, i.e., the Green ring of
the category $\mathcal C$ of finite dimensional left modules over $H_n(q)$, where $n$
is a positive integer greater than $1$, and $q$ is an $n$-th primitive root of unity
in $\mathbb{F}$. Let $n=2$. Then $H_2(-1)$ is exactly the Sweedler's 4-dimensional Hopf algebra $H_4$.

By \cite[Theorem 2.5]{CVZ} or \cite[Page 467]{Cib},
there are $4$ non-isomorphic indecomposable finite dimensional
$H_4$-modules: $V_1=M(1, 0)$, $V_2=M(2, 1)$, $V_3=M(1, 1)$ and $V_4=M(2, 0)$.
From the structures of these indecomposable modules, one can easily see that
dim$_{\mathbb F}{\rm Hom}_{H_4}(V_i, V_j)=1$ or $0$, where
$1\leqslant i,j\leqslant 4$. Let $V=\oplus_{i=1}^4V_i$.
Then the corresponding $\mathbb{F}$-algebra $A:=A(\mathcal C)={\rm End}_{H_4}(V)$ can be described by quiver as follows.
Let $Q$ be the quiver:
$$
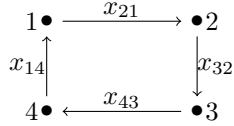
\begin{tikzpicture}[scale=1.0]
\path (0.2,0) node(1) {$1$} (0.4,0) node(a1) {$\bullet$} (2.4,0) node(b1) {$\bullet$} (2.6,0) node (c1) {$2$};
\path (0.2,-1.2) node(2) {$4$} (0.4,-1.2) node(a2) {$\bullet$} (2.4,-1.2) node(b2) {$\bullet$} (2.6,-1.2) node (c2) {$3$};
\path (0.15, -0.6) node(.) {$x_{14}$} (1.4, 0.15) node(.) {$x_{21}$} (1.4, -1.05) node(.) {$x_{43}$} (2.65, -0.6) node(.) {$x_{32}$};
\draw[->] (a1) --(b1);
\draw[->] (b1) --(b2);
\draw[->] (b2) --(a2);
\draw[->] (a2) --(a1);
\end{tikzpicture}$$
and let $J$ be the ideal of the quiver algebra ${\mathbb F}Q$ generated by $x_{32}x_{21}$ and $x_{14}x_{43}$.
Then $A\cong{\mathbb F}Q/J$ as $\mathbb{F}$-algebras. Regarding $A={\mathbb F}Q/J$. Then $A$ is of dimension 10
with an $\mathbb F$-basis $B=\{e_1, e_2, e_3, e_4, x_{21}, x_{32}, x_{43}, x_{14}, x_{43}x_{32}, x_{21}x_{14}\}$ subject to $x_{32}x_{21}=0$ and $x_{14}x_{43}=0$.
Hence $e_iAe_i={\mathbb F}e_i$, $1\<i\<4$, $e_2Ae_1={\mathbb F}x_{21}$, $e_3Ae_2={\mathbb F}x_{32}$,
$e_4Ae_3={\mathbb F}x_{43}$, $e_1Ae_4={\mathbb F}x_{14}$, $e_4Ae_2={\mathbb F}x_{43}x_{32}$,
$e_2Ae_4={\mathbb F}x_{21}x_{14}$ and $e_iAe_j=0$ for the other cases.
Let $V_1=\mathbb{F}v_1$, $V_3=\mathbb{F}v_3$, and let $\{v_{i1}, v_{i2}\}$ be the standard basis of $V_i$ as given in \cite{CVZ},
$i=2, 4$. Under the identification $e_iAe_j={\rm Hom}_{H_4}(V_j, V_i)$, we have that
$e_i={\rm id}_{V_i}$ and $x_{ij}\in{\rm Hom}_{H_4}(V_j, V_i)$ given by $x_{21}(v_1)=v_{22}$, $x_{43}(v_3)=v_{42}$,
$x_{32}(v_{21})=v_{3}$, $x_{32}(v_{22})=0$, $x_{14}(v_{41})=v_{1}$ and $x_{14}(v_{42})=0$.

By \cite{CVZ}, the Green ring $r(\mathcal C)$ is a commutative ring with a $\mathbb Z$-basis $\{r_i=[V_i]|1\<i\<4\}$,
where $r_1=1$, the unity of the ring $r(\mathcal C)$.
Moreover, $r_1r_i=r_i$, $1\<i\<4$, $r_3^2=r_1$, $r_2^2=r_4^2=r_2r_4=r_2+r_4$,
$r_2r_3=r_4$ and $r_3r_4=r_2$. Hence $\mathbf{c}_{1i}=\mathbf{c}_{i1}=\mathbf{e}_{i}$, $1\<i\<4$,
$\mathbf{c}_{33}=\mathbf{e}_{1}$, $\mathbf{c}_{22}=\mathbf{c}_{44}=\mathbf{c}_{24}=\mathbf{c}_{42}=\mathbf{e}_{2}+\mathbf{e}_{4}$,
$\mathbf{c}_{23}=\mathbf{c}_{32}=\mathbf{e}_{4}$ and $\mathbf{c}_{34}=\mathbf{c}_{43}=\mathbf{e}_{2}$
in ${\mathbb N}^4$.  The multiplication of the associative $\mathbb{F}$-algebra
$M(\mathcal C):=\oplus_{i,j,k,l=1}^4M_{\mathbf{c}_{ij}\times\mathbf{c}_{kl}}(A)$
is defined as follows: if $X\in M_{\mathbf{c}_{ij}\times\mathbf{c}_{kl}}(A)$ and
$Y\in M_{\mathbf{c}_{i'j'}\times\mathbf{c}_{k'l'}}(A)$,
then $XY$ is the usual matrix product for $(k, l)=(i',j')$,
and $XY=0$ for $(k, l)\neq(i',j')$.
In order to describe the algebra map $\phi_{\mathcal C}: A\ot_{\mathbb F}A\ra M(\mathcal C)$,
one need to choose a family of $H_4$-module isomorphisms
$\theta_{ij}: V_i\ot V_j\ra\oplus_{k=1}^4c_{ijk}V_k$, $1\<i,j\<4$.
In what follows, let $\ot=\ot_{\mathbb F}$ and $\phi=\phi_{\mathcal C}$. Define $H_4$-module isomorphisms $\theta_{ij}$ by
$$\begin{array}{l}
\theta_{1i}: V_1\ot V_i\ra V_i,\ v_1\ot v\mapsto v,\ \theta_{i1}: V_i\ot V_1\ra V_i,\ v\ot v_1\mapsto v,\ 1\<i\<4.\\
\theta_{32}: V_3\ot V_2\ra V_4,\ v_3\ot v_{21}\mapsto v_{41},\ v_3\ot v_{22}\mapsto v_{42}.\\
\theta_{34}: V_3\ot V_4\ra V_2,\ v_3\ot v_{41}\mapsto v_{21},\ v_3\ot v_{42}\mapsto v_{22}.\\
\theta_{23}: V_2\ot V_3\ra V_4,\ v_{21}\ot v_3\mapsto v_{41},\ v_{22}\ot v_3\mapsto -v_{42}.\\
\theta_{43}: V_4\ot V_3\ra V_2,\ v_{41}\ot v_3\mapsto v_{21},\ v_{42}\ot v_3\mapsto -v_{22}.\\
\theta_{33}: V_3\ot V_3\ra V_1,\ v_3\ot v_3\mapsto v_1.\\
\theta_{22}: V_2\ot V_2\ra V_2\oplus V_4,\ v_{21}\ot v_{22}\mapsto v_{21},\ v_{22}\ot v_{22}\mapsto v_{22},\\
 \hspace{3.9cm} v_{21}\ot v_{21}\mapsto v_{41},\ v_{22}\ot v_{21}\mapsto v_{21}-v_{42}.\\
\theta_{24}: V_2\ot V_4\ra V_2\oplus V_4,\ v_{21}\ot v_{41}\mapsto v_{21},\ v_{21}\ot v_{42}\mapsto v_{22}-v_{41},\\
 \hspace{3.9cm} v_{22}\ot v_{41}\mapsto v_{41},\ v_{22}\ot v_{42}\mapsto v_{42}.\\
\theta_{42}: V_4\ot V_2\ra V_2\oplus V_4,\ v_{41}\ot v_{21}\mapsto v_{21},\ v_{42}\ot v_{21}\mapsto v_{41}-v_{22},\\
 \hspace{3.9cm} v_{41}\ot v_{22}\mapsto v_{41},\ v_{42}\ot v_{22}\mapsto v_{42}.\\
\end{array}$$
   $$\begin{array}{l}
\theta_{44}: V_4\ot V_4\ra V_2\oplus V_4,\ v_{42}\ot v_{41}\mapsto v_{21},\ v_{42}\ot v_{42}\mapsto v_{22},\\
 \hspace{3.9cm} v_{41}\ot v_{41}\mapsto v_{41},\ v_{41}\ot v_{42}\mapsto v_{42}-v_{21}.\\
\end{array}$$
Then by the proof of Proposition \ref{tom}, $\phi: A\ot A\ra M(\mathcal C)$ is given by
$\phi(x\ot y)=\theta_{il}(x\ot y)\theta_{jk}^{-1}\in M_{\mathbf{c}_{il}\times\mathbf{c}_{jk}}(A)$
for any $x\in e_iAe_j$ and $y\in e_lAe_k$. Obviously,
$\phi(e_i\ot e_j)=E_{{\mathbf c}_{ij}}\in M_{{\mathbf c}_{ij}}(A)$,
$\phi(e_1\ot x)=x\in M_{\mathbf{c}_{1i}\times\mathbf{c}_{1j}}(A)$ and $\phi(x\ot e_1)=x\in M_{\mathbf{c}_{i1}\times\mathbf{c}_{j1}}(A)$
for all $x\in e_iAe_j$, $1\<i, j\<4$.  In general, it is not hard to find the expression of $\phi(x\ot y)$ for any $x,y\in A$. For instance,
$\phi(e_2\ot x_{32})=\left(\begin{array}{cc}
                           0 & -e_4 \\
                           \end{array}
                           \right)\in M_{{\mathbf c}_{23}\times{\mathbf c}_{22}}(A)$,
$\phi(x_{21}\ot x_{21})=\left(\begin{array}{c}
                              x_{21} \\
                              0 \\
                              \end{array}
                              \right)\in M_{{\mathbf c}_{22}\times{\mathbf c}_{11}}(A)$,
$\phi(x_{21}\ot x_{32})=-x_{43}x_{32}\in M_{{\mathbf c}_{23}\times{\mathbf c}_{12}}(A)$,
$\phi(x_{32}\ot x_{43})=x_{21}x_{14}\in M_{{\mathbf c}_{34}\times{\mathbf c}_{23}}(A)$
and $\phi(x_{43}\ot x_{32})=-x_{21}x_{14}\in M_{{\mathbf c}_{43}\times{\mathbf c}_{32}}(A)$.

\section*{ACKNOWLEDGMENTS}

We thank FWO and BOF of UHasselt for their financial support. The first author is also supported by NSFC  (Grant No.11571298, 11711530703). The second author is financed by FWO (N1518617).

\end{document}